\documentclass[reqno,12pt]{amsart}
\usepackage{}
\usepackage{amssymb}
\usepackage{amsmath}
\usepackage{amsthm}
\usepackage[usenames]{color}
\usepackage{a4wide}

\newcommand{\R}{\mathbb{R}}

\newtheorem{theorem}{Theorem}[section]

\newtheorem{lemma}[theorem]{Lemma}

\newtheorem{remark}[theorem]{Remark}

\numberwithin{equation}{section}
\usepackage{hyperref}
\usepackage{marginnote}

\usepackage{color}

\begin{document}

\title[Convergence to rarefaction wave for Landau equation]{Small Knudsen rate of convergence to rarefaction wave for the Landau equation}

\author[R.-J. Duan]{Renjun Duan}
\address[R.-J. Duan]{Department of Mathematics, The Chinese University of Hong Kong, Shatin, Hong Kong,
        People's Republic of China}
\email{rjduan@math.cuhk.edu.hk}

\author[D.-C. Yang]{Dongcheng Yang}
\address[D.-C. Yang]{School of Mathematical Sciences,
        South China Normal University, Guangzhou 510631, People's Republic of China}
\email{dcmath@sina.com}

\author[H.-J. Yu]{Hongjun Yu}
\address[H.-J. Yu]{School of Mathematical Sciences, South China Normal University, Guangzhou 510631, People's Republic of China}
\email{yuhj2002@sina.com}

\begin{abstract}
In this paper, we are concerned with the hydrodynamic limit to rarefaction waves of the compressible Euler system for the Landau equation with Coulomb potentials as the Knudsen number $\epsilon>0$ is vanishing. Precisely, whenever $\epsilon>0$ is small, for the Cauchy problem on the Landau equation with suitable initial data involving a scaling parameter $a\in [\frac{2}{3},1]$, we construct the unique global-in-time uniform-in-$\epsilon$ solution around a local Maxwellian whose fluid quantities are the rarefaction wave of the corresponding Euler system. In the meantime, we establish the convergence of solutions to the Riemann rarefaction wave uniformly away from $t=0$ at a rate  $\epsilon^{\frac{3}{5}-\frac{2}{5}a}|\ln \epsilon|$ as $\epsilon\to 0$. The proof is based on the refined energy approach combining \cite{G1} and \cite{LTP2} under the scaling transformation $(t,x)\to (\epsilon^{-a}t,\epsilon^{-a}x)$.
\end{abstract}


\date{\today}
\maketitle

\setcounter{tocdepth}{1}
\tableofcontents
\thispagestyle{empty}


\section{Introduction}

The Landau equation is one of the most fundamental equations in plasma physics. A lot of great contributions in the mathematical study of the spatially inhomogeneous Landau equation have been made by many people, for instance, Lions \cite{Lions}, Villani \cite{Vi96}, Alexander-Villani \cite{AV04},  Degond-Lemou \cite{DL} and Guo \cite{G1}. In particular, Guo \cite{G1} gave the first proof for constructing the global classical solutions close to a constant equilibrium state in a periodic box, and later Strain and Guo \cite{SG06,SG} established the large time asymptotic behavior of those global solutions. Since then, the spatially inhomogeneous perturbation theory of the Landau equation around global Maxwellians was further developed in different settings, for instance, see Yu \cite{YHJ}, Carrapatoso-Tristani-Wu \cite{CTW}, Carrapatoso-Mischler \cite{CaMi}, Guo-Hwang-Jang-Ouyang \cite{GHJO}, and Duan-Liu-Sakamoto-Strain \cite{DLSS}. In the case of the whole space, the global classical solution near vacuum was also constructed by Luk \cite{Luk}. Recently, lots of research has been done into investigating the regularity of solutions to the spatially inhomogeneous Landau equation for general initial data under certain conditions, see Golse-Imbert-Mouhot-Vasseur \cite{GIMV} and Henderson-Snelson \cite{HeSn}, for instance. In this paper, we would rather consider another interesting topic on the hydrodynamic limit of the Landau equation for which quite few results are known although it has been extensively studied in the Boltzmann theory, cf.~Grad \cite{Gra}, Golse \cite{Golse} and Saint-Raymond \cite{SR}.

\subsection{Problem}
We consider the following one-dimensional Landau equation
\begin{equation}
\label{1.1}
\partial_{t}F+v_{1}\partial_{x}F=\frac{1}{\epsilon}Q(F,F),
\end{equation}
where the unknown $F=F(t,x,v)\geq0$ stands for the density
distribution function for the gas particles with space position
$x\in\mathbb{R}$ and velocity $v=(v_{1},v_{2},v_{3})\in\mathbb{R}^{3}$ at time $t>0$. On the right hand side of \eqref{1.1}, the parameter $\epsilon>0$ is the Knudsen number which is proportional to the mean free path, and the Landau collision operator $Q(\cdot,\cdot)$ is a bilinear integro-differential operator acting only on velocity variables, taking the form of
\begin{equation}
\label{1.2}
Q(F_{1},F_{2})(v)=\nabla_{v}\cdot\int_{\mathbb{R}^{3}}\phi(v-v_{*})\left\{F_{1}(v_{*})\nabla_{v}F_{2}(v)-\nabla_{v_{*}}F_{1}(v_{*})F_{2}(v)\right\}dv_{*}.
\end{equation}
The non-negative matrix $\phi$ in the integral above is given by
\begin{align}
\label{1.3}
\phi(v)=\Big(I-\frac{v\otimes v}{|v|^{2}}\Big)|v|^{\gamma+2}, \quad \gamma\geq -3,
\end{align}
where $I$ is the $3\times3$ identity matrix and $v\otimes v$ is the tensor product.  Note that \eqref{1.2} in the case
$\gamma=-3$ corresponds to the original (Fokker-Planck)-Landau collision operator for Coulomb potentials, see \cite{AV04,DL,G1}. Through the paper, we are  focused on the very soft potentials case $-3\leq\gamma<-2$,
since it is similar to treat the other cases $\gamma\geq -2$ in an easier way for which the linearized Landau operator has the spectral gap.

Formally, when the Knudsen number $\epsilon$ tends to zero, the limit of the
Landau equation \eqref{1.1} gives rise to the one-dimensional compressible Euler system
\begin{equation}
\label{1.4}
\begin{cases}
\rho_{t}+(\rho u_{1})_{x}=0,
\\
(\rho u_{1})_{t}+(\rho u_{1}^{2})_{x}+p_{x}=0,
\\
(\rho u_{i})_{t}+(\rho u_{1}u_{i})_{x}=0, ~~i=2,3,
\\
\big\{\rho (e+\frac{|u|^{2}}{2})\big\}_{t}+\big\{\rho u_{1}(e+\frac{|u|^{2}}{2})+pu_{1}\big\}_{x}=0,
\end{cases}
\end{equation}
where
\begin{equation}
\label{1.5}
\begin{cases}
\rho(t,x)=\int_{\mathbb{R}^{3}}\psi_{0}(v)F\,dv,
\\
\rho u_{i}(t,x)=\int_{\mathbb{R}^{3}}\psi_{i}(v)F\,dv, \quad \mbox{for $i=1,2,3$,}
\\
\rho(e+\frac{1}{2}|u|^{2})(t,x)=\int_{\mathbb{R}^{3}}\psi_{4}(v)F\,dv.
\end{cases}
\end{equation}
Here $\rho=\rho(t,x)>0$ is the mass density, $u=u(t,x)=(u_{1},u_{2},u_{3})$ is the fluid velocity,
$e=e(t,x)$ is the internal energy, and $p=R\rho\theta$  is the pressure, where $R$ is the gas constant that we will set to be $\frac{2}{3}$ throughout the
paper for convenience and $\theta=\theta(t,x)>0$ is the temperature related to the internal energy $e$ by $e=\frac{3}{2}R\theta=\theta$.  Moreover, the five collision invariants $\psi_{i}(v)$ $(i=0,1,2,3,4)$ are
given by
$$
\psi_{0}(v)=1, \quad \psi_{i}(v)=v_{i}~(i=1,2,3),\quad \psi_{4}(v)=\frac{1}{2}|v|^{2},
$$
satisfying
\begin{equation}
\label{1.6}
\int_{\mathbb{R}^{3}}\psi_{i}(v)Q(F,F)\,dv=0,\quad \mbox{for $i=0,1,2,3,4$.}
\end{equation}

The rigorous mathematical justification of establishing the hydrodynamic limit to the Euler system \eqref{1.4} for the Landau equation \eqref{1.1} in a general setting is an outstanding open problem in kinetic theory, which is similar to the case of the Boltzmann equation with or without angular cutoff, cf.~\cite{Golse,Gra,SR}. Regarding the topic on solutions with basic wave patterns (cf.~\cite{LP,SM}), there have been extensive studies of global existence and large time asymptotic behavior of solutions (cf.~\cite{CN,HXY,LTP1,LTP3,Yu}) and small Knudsen rate of convergence (cf.~\cite{HWY,HWWY,LX,XinZeng,Yu1}) in the context of the cutoff Boltzmann equation; some relevant literature will be reviewed in detail later on. However, to the best of our knowledge, few results on this topic are known for either the non-cutoff Boltzmann or Landau equation, essentially due to the effect of grazing singularity of both collision operators on non-trivial profiles with even small space variations connecting two distinct global Maxwellians, that makes it necessary to develop new perturbation approaches beyond the situation where solutions are close to a constant equilibrium (cf.~\cite{AMUXY,GrSt,G1}).   Recently, the first and third authors of this paper studied in \cite{DuanY} the nonlinear stability as well as the large time asymptics of rarefaction waves for the Landau equation \eqref{1.1} with Coulomb potentials. In the present work, we expect to further study the hydrodynamic limit with rarefaction waves of the one-dimensional Landau equation \eqref{1.1} as Knudsen number $\epsilon>0$ is sufficiently small.

\subsection{Macro-micro decomposition}
For our purpose above, as in \cite{LTP1,LTP2}, we define the local Maxwellian $M$ associated with the solution $F$ to the equation
\eqref{1.1} in terms of the fluid quantities of $F$ as in \eqref{1.5}  by
\begin{equation}
\label{1.7}
M=M_{[\rho,u,\theta](t,x)}(v)=\frac{\rho(t,x)}{(2\pi R\theta(t,x))^{3/2}}\exp\left(-\frac{|v-u(t,x)|^{2}}{2R\theta(t,x)}\right).
\end{equation}
We denote an $L^{2}_{v}(\mathbb{R}^{3})$ inner product as
$
\langle h,g\rangle =\int_{\mathbb{R}^{3}}h(v)g(v)\,dv.
$
Then, considering the linearized Landau operator around the local Maxwellian $M$ of the form
\begin{equation}
\label{def.LM}
L_{M}h=Q(h,M)+Q(M,h),
\end{equation}
the macroscopic kernel space
is spanned by the following
five pairwise-orthogonal base
\begin{equation}
\label{1.8}
\begin{cases}
\chi_{0}(v)=\frac{1}{\sqrt{\rho}}M,
\\
\chi_{i}(v)=\frac{v_{i}-u_{i}}{\sqrt{R\rho\theta}}M, \quad \mbox{for $i=1,2,3$,}
\\
\chi_{4}(v)=\frac{1}{\sqrt{6\rho}}\left(\frac{|v-u|^{2}}{R\theta}-3\right)M,
\\
\langle \chi_{i},\frac{\chi_{j}}{M}\rangle=\delta_{ij},
\quad i,j=0,1,2,3,4.
\end{cases}
\end{equation}
In terms of these five orthonormal functions, we define the macroscopic projection $P_{0}$
and the microscopic projection $P_{1}$ as follows
\begin{equation}
\label{1.9}
P_{0}h=\sum_{i=0}^{4}\langle h,\frac{\chi_{i}}{M}\rangle\chi_{i},\quad P_{1}h=h-P_{0}h.
\end{equation}
A function $h(v)$ is called microscopic or non-fluid if
\begin{equation}
\label{1.10}
\int_{\mathbb{R}^{3}}h(v)\psi_{i}(v)\,dv=0, \quad \mbox{for $i=0,1,2,3,4$}.
\end{equation}
Initiated by Liu-Yu \cite{LTP1} and developed by Liu-Yang-Yu \cite{LTP2}, for a non-trivial solution profile connecting two different global Maxwellians at $x=\pm\infty$,
we decompose the equation \eqref{1.1} and its solution with respect to the local Maxwellian \eqref{1.7} as
\begin{equation}
\label{1.11}
F=M+G, \quad P_{0}F=M, \quad P_{1}F=G,
\end{equation}
where the local Maxwellian $M$ as \eqref{1.7} and $G=G(t,x,v)$ represent the macroscopic and microscopic
component in the solution respectively. Then the equation \eqref{1.1} becomes
\begin{equation}
\label{1.12}
\partial_{t}(M+G)+v_{1}\partial_{x}(M+G)=\frac{1}{\epsilon}Q(G,M)+\frac{1}{\epsilon}Q(M,G)+\frac{1}{\epsilon}Q(G,G)
\end{equation}
due to $Q(M,M)=0$. Multiplying \eqref{1.12} by the collision invariants $\psi_{i}(v)$ ($i=0,1,2,3,4$)
and integrating the resulting equations with respect to $v$ over $\mathbb{R}^{3}$, one gets
the following macroscopic system
\begin{equation}
\label{1.13}
\begin{cases}
\rho_{t}+(\rho u_{1})_{x}=0,
\\
(\rho u_{1})_{t}+(\rho u_{1}^{2})_{x}+p_{x}=-\int_{\mathbb{R}^{3}} v^{2}_{1}G_{x}\,dv,
\\
(\rho u_{i})_{t}+(\rho u_{1}u_{i})_{x}=-\int_{\mathbb{R}^{3}} v_{1}v_{i}G_{x}\,dv, ~~i=2,3,
\\
\big\{\rho (\theta+\frac{|u|^{2}}{2})\big\}_{t}+\big\{\rho u_{1}(\theta+\frac{|u|^{2}}{2})+pu_{1}\big\}_{x}
=-\int_{\mathbb{R}^{3}} \frac{1}{2}v_{1}|v|^{2}G_{x}\,dv.
\end{cases}
\end{equation}
Here we have used \eqref{1.5}, \eqref{1.6} and the fact that $G_{t}$ is microscopic by \eqref{1.10}.

Applying the projection operator $P_{1}$ to \eqref{1.12} and using \eqref{1.11}, we obtain the following microscopic system
\begin{align}
\label{1.14}
\partial_{t}G+P_{1}(v_{1}\partial_{x}G)+P_{1}(v_{1}\partial_{x}M)=\frac{1}{\epsilon}L_{M}G+\frac{1}{\epsilon}Q(G,G).
\end{align}
Here the linearized operator $L_{M}$ is defined in \eqref{def.LM}. Recall that the null space $\mathcal{N}$ of $L_{M}$ is spanned by
$\chi_{i}~(i=0,1,2,3,4)$.
It follows by \eqref{1.14} that
\begin{equation}
\label{1.15}
G=\epsilon L^{-1}_{M}[P_{1}(v_{1}\partial_{x}M)]+L^{-1}_{M}\Theta,
\quad \Theta:=\epsilon \partial_{t}G+\epsilon P_{1}(v_{1}\partial_{x}G)-Q(G,G).
\end{equation}
Substituting \eqref{1.15} into \eqref{1.13},
we obtain the following fluid-type system
\begin{align}
\label{1.16}
\begin{cases}
\rho_{t}+(\rho u_{1})_{x}=0,
\\
(\rho u_{1})_{t}+(\rho u_{1}^{2})_{x}+p_{x}=\frac{4}{3}\epsilon(\mu(\theta)u_{1x})_{x}-(\int_{\mathbb{R}^{3}} v^{2}_{1}L^{-1}_{M}\Theta \,dv)_{x},
\\
(\rho u_{i})_{t}+(\rho u_{1}u_{i})_{x}=\epsilon(\mu(\theta)u_{ix})_{x}-(\int_{\mathbb{R}^{3}} v_{1}v_{i}L^{-1}_{M}\Theta \,dv)_{x}, ~~i=2,3,
\\
\big\{\rho (\theta+\frac{|u|^{2}}{2})\big\}_{t}+\big\{\rho u_{1}(\theta+\frac{|u|^{2}}{2})+pu_{1}\big\}_{x}=\epsilon(\kappa(\theta)\theta_{x})_{x}+\frac{4}{3}\epsilon(\mu(\theta)u_{1}u_{1x})_{x}
\\
\ \ \ +\epsilon(\mu(\theta)u_{2}u_{2x})_{x}+\epsilon(\mu(\theta)u_{3}u_{3x})_{x}-\frac{1}{2}(\int_{\mathbb{R}^{3}}v_{1}|v|^{2}L^{-1}_{M}\Theta \,dv)_{x}.
\end{cases}
\end{align}
Here the viscosity coefficient $\mu(\theta)>0$ and the heat conductivity coefficient $\kappa(\theta)>0$, both are smooth functions depending only on $\theta$.
The explicit formulas of  $\mu(\theta)$ and $\kappa(\theta)$ are defined by \eqref{5.3}.

\subsection{Rarefaction wave and its smooth approximation}
Now we turn to define the rarefaction wave profile to the system \eqref{1.1} as in \cite{LiuXin,LTP3,M1}. Consider the Euler system \eqref{1.4} with the state equation
$p=\frac{2}{3}\rho\theta=k_0\rho^{5/3}\exp(S)$, where $k_0=\frac{1}{2\pi e}$ and $S$ is the macroscopic entropy,
supplemented with the following Riemann initial data
\begin{equation}
\label{1.17}
(\rho,u,\theta)(t,x)|_{t=0}=(\rho^{R}_{0},u^{R}_{0},\theta^{R}_{0})(x)
=\begin{cases}
(\rho_{+},u_{+},\theta_{+}),\quad x>0,
\\
(\rho_{-},u_{-},\theta_{-}),\quad x<0.
\end{cases}
\end{equation}
Here $\rho_{\pm}>0$, $u_{\pm}=(u_{1\pm},0,0)$ and $\theta_{\pm}>0$
are assumed to be constant. It is well known that the Euler system \eqref{1.4} for $(\rho,u_{1},S)$ has three distinct eigenvalues
$$
\lambda_{i}(\rho,u_{1},S)=u_{1}+(-1)^{\frac{i+1}{2}}\sqrt{p_{\rho}(\rho,S)},~i=1,3,
\quad \lambda_{2}(\rho,u_{1},S)=u_{1},
$$
where $p_{\rho}(\rho,S)=\frac{5}{3}k_0\rho^{\frac{2}{3}}e^{S}>0$. In terms of the two Riemann invariants
of the third eigenvalue $\lambda_{3}(\rho,u_{1},S)$, we define the 3-rarefaction wave curve for the given
left constant state $(\rho_{-},u_{1-},\theta_{-})$ with $\rho_{-}>0$ and  $\theta_{-}>0$ as below (cf.~\cite{LP,SM})
\begin{align}
\label{1.18a}
R_{3}(\rho_{-},u_{1-},\theta_{-})=\{&(\rho,u_{1},\theta)\in\mathbb{R}_{+}\times\mathbb{R}\times\mathbb{R}_{+}\mid
S=S_{*},
\nonumber\\
&u_{1}-\sqrt{15k_0}e^{\frac{S}{2}}\rho^{\frac{1}{3}}=u_{1-}-\sqrt{15k_0}e^{\frac{S_{*}}{2}}\rho_{-}^{\frac{1}{3}},
~~\rho>\rho_{-},~~u_{1}>u_{1-}\}.
\end{align}
Here and to the end, $S_{*}:=S_{-}=-\frac{2}{3}\ln\rho_{-}+\ln(\frac{4}{3}\pi\theta_{-})+1$ is a constant.

Without  loss of generality, we consider only the simple 3-rarefaction wave in this paper, and  the case for 1-rarefaction wave can be treated similarly.
The 3-rarefaction wave to the Euler system \eqref{1.4} with \eqref{1.17} can be expressed explicitly by the
Riemann solution to the inviscid Burgers equation
\begin{equation}
\label{1.19}
\begin{cases}
\omega_{t}+\omega\omega_{x}=0,
\\
\omega(0,x)=
\begin{cases}
\omega_{-},\quad x<0,
\\
\omega_{+},\quad x>0.
\end{cases}
\end{cases}
\end{equation}
If two constants $\omega_{-}<\omega_{+}$ are chosen, then \eqref{1.19} admits a centered rarefaction wave solution
$\omega^{R}(x,t)=\omega^{R}(\frac{x}{t})$ connecting $\omega_{-}$ and $\omega_{+}$ (cf.~\cite{M1})
in the form of
\begin{equation*}
\omega^{R}(\frac{x}{t})=
\begin{cases}
\omega_{-},\quad \frac{x}{t}\leq\omega_{-},
\\
\frac{x}{t},\quad\omega_{-}< \frac{x}{t}\leq\omega_{+},
\\
\omega_{+},\quad \frac{x}{t}>\omega_{+}.
\end{cases}
\end{equation*}
For $(\rho_{+},u_{1+},\theta_{+})\in R_{3}(\rho_{-},u_{1-},\theta_{-})$, the 3-rarefaction wave $(\rho^{R},u^{R},\theta^{R})(\frac{x}{t})$ with
$u^{R}(\frac{x}{t})=(u_{1}^{R},u_{2}^{R},u_{3}^{R})(\frac{x}{t})$
to the Riemann problem \eqref{1.4} with \eqref{1.17} can be defined explicitly by
\begin{equation}
\label{1.20}
\begin{cases}
\lambda_{3}(\rho^{R}(\frac{x}{t}),u^{R}_{1}(\frac{x}{t}),S_{*})=
\begin{cases}
\lambda_{3}(\rho_{-},u_{1-},S_{*}),\quad \frac{x}{t}\leq \lambda_{3}(\rho_{-},u_{1-},S_{*}),
\\
\frac{x}{t}, \quad  \lambda_{3}(\rho_{-},u_{1-},S_{*})<\frac{x}{t}\leq \lambda_{3}(\rho_{+},u_{1+},S_{*}),
\\
\lambda_{3}(\rho_{+},u_{1+},S_{*}),\quad \frac{x}{t}>\lambda_{3}(\rho_{+},u_{1+},S_{*}),
\end{cases}
\\
u^{R}_{1}(\frac{x}{t})-\sqrt{15k_0}e^{\frac{S_{*}}{2}}(\rho^{R})^{\frac{1}{3}}(\frac{x}{t})
=u_{1-}-\sqrt{15k_0}e^{\frac{S_{*}}{2}}\rho_{-}^{\frac{1}{3}},\quad u^{R}_{2}=u^{R}_{3}=0,
\\
\theta^{R}(\frac{x}{t})
=\frac{3}{2}k_0e^{S_{*}}(\rho^{R})^{\frac{2}{3}}(\frac{x}{t}).
\end{cases}
\end{equation}
\par
Since the above 3-rarefaction wave is only Lipschitz continuous, we shall construct an approximate smooth rarefaction wave
to the 3-rarefaction wave defined in \eqref{1.20}. Motivated by \cite{M1,Xin}, the approximate smooth rarefaction wave can be constructed
by the Burgers equation
\begin{equation}
\label{1.21}
\begin{cases}
\overline{\omega}_{t}+\overline{\omega}\,\overline{\omega}_{x}=0,
\\
\overline{\omega}(0,x)=\overline{\omega}_{\delta}(x)=\overline{\omega}(\frac{x}{\delta})=\frac{\omega_{+}+\omega_{-}}{2}+\frac{\omega_{+}
-\omega_{-}}{2}\tanh(\frac{x}{\delta}),
\end{cases}
\end{equation}
where $\delta>0$  is a small constant depending on the Knudsen number $\epsilon$. In fact, as given in \eqref{3.3} later on, we will choose $\delta=\frac{1}{k}\epsilon^{\frac{3}{5}-\frac{2}{5}a}$ for a suitably small constant $k>0$ independent of $\epsilon$.
By the method of characteristic curves,
the solution $\overline{\omega}_{\delta}(t,x)$ to the problem \eqref{1.21} can be given  by
\begin{equation*}
\overline{\omega}_{\delta}(t,x)=\overline{\omega}_{\delta}(x_{0}(t,x)),\quad
x=x_{0}(t,x)+\overline{\omega}_{\delta}(x_{0}(t,x))t.
\end{equation*}
The properties of $\overline{\omega}_{\delta}(t,x)$ are given by Lemma \ref{lem5.1} in Section \ref{sec.5}.

Correspondingly, the approximate smooth 3-rarefaction wave  $(\bar{\rho}_{\delta},\bar{u}_{\delta},\bar{\theta}_{\delta})(t,x)$ to \eqref{1.20} for the
Euler system \eqref{1.4} and \eqref{1.17} can be defined by
\begin{equation}
\label{1.23}
\begin{cases}
\overline{\omega}_{\delta}(t,x)=\lambda_{3}(\bar{\rho}_{\delta}(t,x),\bar{u}_{1\delta}(t,x),S_{*}), \quad
\omega_{\pm}=\lambda_{3}(\rho_{\pm},u_{1\pm},S_{*}),
\\
\bar{u}_{1\delta}(t,x)-\sqrt{15k_0}e^{\frac{S_{*}}{2}}\bar{\rho}_{\delta}^{\frac{1}{3}}(t,x)
=u_{1-}-\sqrt{15k_0}e^{\frac{S_{*}}{2}}\rho_{-}^{\frac{1}{3}},\quad \bar{u}_{2\delta}=\bar{u}_{3\delta}=0,\\
\bar{\theta}_{\delta}(t,x)
=\frac{3}{2}k_0e^{S_{*}}\bar{\rho}_{\delta}^{\frac{2}{3}}(t,x),
\\
\lim\limits_{x\to \pm\infty}(\bar{\rho}_{\delta},\bar{u}_{1\delta},\bar{\theta}_{\delta})(t,x)=(\rho_{\pm},u_{1\pm},\theta_{\pm}),
\quad (\rho_{+},u_{1+},\theta_{+})\in R_{3}(\rho_{-},u_{1-},\theta_{-}),
\end{cases}
\end{equation}
where $\overline{\omega}_{\delta}(t,x)$ is the solution of Burger equation \eqref{1.21}.
From now on, we shall omit the explicit dependence of $(\bar{\rho}_{\delta},\bar{u}_{\delta},\bar{\theta}_{\delta})(t,x)$ on $\delta$ and denote it by $(\bar{\rho},\bar{u},\bar{\theta})(t,x)$ for simplicity. Then the approximate smooth 3-rarefaction wave $(\bar{\rho},\bar{u},\bar{\theta})(t,x)$
satisfies the following Euler system
\begin{equation}
\label{1.24}
\begin{cases}
\bar{\rho}_{t}+(\bar{\rho}\bar{u}_{1})_{x}=0,
\\
(\bar{\rho} \bar{u}_{1})_{t}+(\bar{\rho}\bar{u}_{1}^{2})_{x}+\bar{p}_{x}=0,
\\
(\bar{\rho} \bar{u}_{i})_{t}+(\bar{\rho}\bar{u}_{1}\bar{u}_{i})_{x}=0,\quad {i=2,3},
\\
(\bar{\rho} \bar{\theta})_{t}+(\bar{\rho}\bar{u}_{1}\bar{\theta})_{x}+\bar{p}\bar{u}_{1x}=0,
\end{cases}
\end{equation}
where $\bar{p}=R\bar{\rho}\bar{\theta}$. Properties of $(\bar{\rho},\bar{u},\bar{\theta})(t,x)$ are given in Lemma
\ref{lem5.2} in Section \ref{sec.5}.
In terms of the approximate rarefaction wave $(\bar{\rho},\bar{u},\bar{\theta})(t,x)$, we denote
\begin{equation*}
\overline{M}=M_{[\bar{\rho},\bar{u},\bar{\theta}](t,x)}(v)
=\frac{\bar{\rho}(t,x)}{(2\pi R\bar{\theta}(t,x))^{3/2}}\exp\Big(-\frac{|v-\bar{u}(t,x)|^{2}}{2R\bar{\theta}(t,x)}\Big).
\end{equation*}
For the technical reason as in \cite{XinZeng}, we choose the far-field data $(\rho_+,u_+,\theta_+)$ and $(\rho_-,u_-,\theta_-)$
in \eqref{1.17} to be close enough to the constant state $(1,0,\frac{3}{2})$ such that the approximate smooth rarefaction wave further satisfies that
\begin{equation}
\label{1.26a}
\left\{\begin{aligned}
 &  \eta_{0}:=\sup_{t\geq 0,x\in \R}\{|\bar{\rho}(t,x)-1|+|\bar{u}(t,x)|+|\bar{\theta}(t,x)-\frac{3}{2}|\}\ \text{is small},\\
& \frac{1}{2}\sup_{t\geq0,x\in\mathbb{R}}\bar{\theta}(t,x)<\frac{3}{2}<\inf_{t\geq0,x\in\mathbb{R}}\bar{\theta}(t,x).
\end{aligned}\right.
\end{equation}
As in \cite{LTP3}, associated with the constant state $(1,0,\frac{3}{2})$, we will use throughout the paper a global Maxwellian
$$
\mu=M_{[1,0,\frac{3}{2}]}(v)=(2\pi)^{-\frac{3}{2}}\exp\{-|v|^{2}/2\}.
$$

\subsection{Main result}
With all the above preparations, the main result of the paper can be stated as follows.

\begin{theorem}\label{thm1.1}
Let $-3\leq \gamma<-2$ in \eqref{1.3}. Assume that the far-field data $(\rho_\pm,u_\pm,\theta_\pm)$ satisfy $u_{2\pm}=u_{3\pm}=0$ and $(\rho_+,u_{1+},\theta_+)\in R_3(\rho_-,u_{1-},\theta_-)$ in \eqref{1.18a}, and $\delta_r:=|\rho_{+}-\rho_{-}|+|u_{+}-u_{-}|+|\theta_{+}-\theta_{-}|$ is the wave strength. Let $(\rho^{R},u^{R},\theta^{R})(\frac{x}{t})$ be the Riemann solution \eqref{1.20} of the Euler system \eqref{1.4} and \eqref{1.17}, and $(\bar{\rho},\bar{u},\bar{\theta})(t,x)$ be the corresponding approximate smooth profile satisfying \eqref{1.23}, \eqref{1.24} and \eqref{1.26a} induced by the Burgers equation \eqref{1.21} with $\delta=\frac{1}{k}\epsilon^{\frac{3}{5}-\frac{2}{5}a}$ for $\frac{2}{3}\leq a\leq 1$ and $k>0$. Then, there are small constants $\epsilon_0>0$, $\overline{\eta}_0>0$ and $k>0$ such that for any $\epsilon\in (0,\epsilon_0)$ and any $\delta_r>0$ and any $\eta_{0}>0$ with $\delta_r+\eta_0<\overline{\eta}_0$, there exists a global-in-time solution  $F(t,x,v)\geq 0$ to the Landau equation \eqref{1.1} with initial data
\begin{equation}
\label{thm.conid}
F(0,x,v)=M_{[\bar{\rho},\bar{u},\bar{\theta}](0,\frac{x}{\epsilon^a})}(v)
\end{equation}
such that the following things hold true:
\begin{itemize}
  \item[(a)] Under the scaling transformation $(\tau,y)=(\epsilon^{-a}t,\epsilon^{-a}x)$ as in \eqref{2.1}, there are an energy functional $\mathcal{E}_{2}(\tau)$ and a corresponding energy dissipation functional $\mathcal{D}_{2}(\tau)$, given by \eqref{2.18} and \eqref{2.19} in terms of $(\tau,y)$ coordinates, respectively, such that
\begin{align}\label{thm.enineq}
\sup_{\tau\geq 0}\mathcal{E}_{2}(\tau)
+\int^{+\infty}_{0}\mathcal{D}_{2}(\tau)\,d\tau\leq Ck^{\frac{1}{6}}\epsilon^{\frac{6}{5}-\frac{4}{5}a}.
\end{align}
  \item[(b)] For any $l>0$, there is a constant $C_{l,k}>0$, independent of $\epsilon$, such that
\begin{equation}
\label{1.27}
\sup_{t\geq l}\|\frac{F(t,x,v)-M_{[\rho^{R},u^{R},\theta^{R}](x/t)}(v)}{\sqrt{\mu}}\|_{L_{x}^{\infty}L_{v}^{2}}\leq C_{l,k}\epsilon^{\frac{3}{5}-\frac{2}{5}a}|\ln\epsilon|.
\end{equation}
\end{itemize}
\end{theorem}

\begin{remark}
To the best of our knowledge, Theorem \ref{thm1.1} seems to provide the first result regarding the hydrodynamic limit with rarefaction waves for the Landau equation. It remains open to obtain similar results in case of other kinds of basic wave patterns such as shock wave and contact discontinuity. Moreover, we expect that the current work may shed a little light on the study of the same topic on the non-cutoff Boltzmann equation for which the grazing collision effect plays a role similar to the Landau equation.
\end{remark}

\begin{remark}
Estimate \eqref{1.27} shows that under the condition \eqref{thm.conid} on initial data, the uniform convergence rate in small Knudsen number $\epsilon>0$ can be variable with respect to the scaling parameter $a\in[\frac{2}{3},1]$. In particular, choosing $a=\frac{2}{3}$ can give the fastest convergence rate $\epsilon^{\frac{1}{3}}|\ln\epsilon|$.
\end{remark}

\begin{remark}
It should be pointed out that the scaling argument was first used in  Xin \cite{Xin} to study under the transformation $(\tau,y)=(\epsilon^{-\frac{3}{4}}t,\epsilon^{-\frac{3}{4}}x)$ the vanishing viscosity limit to rarefaction waves for the one-dimensional compressible Navier-Stokes system, where the convergence rate is $\epsilon^{\frac{1}{4}}|\ln\epsilon|$. Later, Xin-Zeng \cite{XinZeng} justified the hydrodynamic limit with rarefaction waves of the Boltzmann equation for the hard sphere model with the convergence rate
$\epsilon^{\frac{1}{5}}|\ln\epsilon|$ through the scaling transformation $(\tau,y)=(\epsilon^{-1}t,\epsilon^{-1}x)$;
this convergence rate was later improved by Li \cite{LX} to be $\epsilon^{\frac{1}{3}}|\ln\epsilon|^{2}$ under the scaling $(\tau,y)=(\epsilon^{-\frac{2}{3}}t,\epsilon^{-\frac{2}{3}}x)$.
\end{remark}

\subsection{Relevant literature}
Let's review some works related to the study in this paper. Mathematically it is an important and challenging problem to rigorously justify the hydrodynamic limit of kinetic equations in a general setting. Great contributions have been made into different topics of the Boltzmann equation with cutoff. We only refer readers to \cite{Golse,Gra,SR} mentioned before, as well as two recent progresses \cite{EGKM,EGM} and reference therein, and also refer to \cite{Sone} for numerical investigations. Thus, we mainly focus on those known results on the limit of the Boltzmann equation to the compressible Euler system admitting solutions of basic wave patterns, such as rarefaction waves, contact discontinuities and shock waves. Particularly, Yu \cite{Yu1} first established the validity of hydrodynamic limit of the Boltzmann equation for the hard-sphere model when the solution of the Euler system contains only the non-interacting shocks. Precisely, he showed that the Boltzmann solution converges to a local Maxwellian defined by the solution of the Euler system uniformly away from the shock in any fixed time interval. Later, Huang-Wang-Yang \cite{HWY} proved the hydrodynamic limit to a single contact discontinuity wave, and Xin-Zeng \cite{XinZeng} showed the hydrodynamic limit to the Euler system with non-interacting rarefaction waves. As mentioned before, the convergence rate in \cite{XinZeng} was improved by Li \cite{LX} through a different scaling transformation.
Furthermore, Huang-Wang-Wang-Yang \cite{HWWY} obtained the hydrodynamic limit in the general setting of Riemann solutions that contains the superposition of shock, rarefaction wave and contact discontinuity.

Although the hydrodynamic limit from the Boltzmann equation for the hard-sphere model to the Euler system with basic wave patterns has been greatly studied  as mentioned above, to the best of our knowledge there are few results on the same topic for the Landau equation or the non-cutoff Boltzmann equation when grazing collisions of particles are dominated.
Notice that the cutoff Boltzmann operator is an integral one without angular singularity while the Landau operator or the non-cutoff Boltzmann operator features the velocity diffusion,
so it is formally much harder to treat the latter case for the hydrodynamic limit to the non-trivial profiles with space variations.

In this paper, we prove the existence of global-in-time solutions to
the one-dimensional Landau equation with suitable initial data as Knudsen number $\epsilon>0$
is sufficiently small. And the solution of the Landau equation converges to the local Maxwellian defined by the
rarefaction wave of the Euler system uniformly away from $t=0$ as  $\epsilon\to 0$.
Moreover, we obtain the uniform convergence rate $\epsilon^{\frac{3}{5}-\frac{2}{5}a}|\ln\epsilon|$ with
$a\in[\frac{2}{3},1]$ by using the scaling transformation $y=\epsilon^{-a}x$ and $\tau=\epsilon^{-a}t$.
It should be pointed out that the energy estimates in the current work are performed in the Eulerian coordinates instead of the Lagrangian coordinates as used in \cite{HWY,HWWY,LX,XinZeng}.

\subsection{Main strategy of the proof}
In what follows we present a few key points on the proof of the main result. In fact, the main strategy is based on a scaling transformation of the independent variables and
the decomposition of the solution for the Landau equation with respect to the local Maxwellian that was initiated by Liu-Yu \cite{LTP1} and developed by Liu-Yang-Yu \cite{LTP2} in the Boltzmann theory. We thus can make use of the macro-micro
decomposition to rewrite the Landau equation as the form of the compressible Navier-Stokes-type system so that
the analysis in the context of the viscous conservation laws can be applied to capture the dissipation of the fluid part around wave patterns. Since we are concerned with the convergence of
the solution of Landau equation to the local Maxwellian defined by rarefaction waves of the Euler system, it gives rise to more analytic difficulties than the study of convergence to a global Maxwellian as $\epsilon\to 0$.
Similar for showing the large time asympotics in \cite{DuanY}, the term $\|(\bar{u}_{x},\bar{\theta}_x)\|^{2}$ is not integrable with respect to time $t$. Hence we need to consider the subtraction of $G(t,x,v)$ by $\overline{G}(t,x,v)$ as \eqref{2.3}
to cancel the slow time decay terms. However, unlike the Boltzmann equation with cutoff potentials, the inverse of the linearized operator $L^{-1}_{M}$ defined as \eqref{1.15}
is unbounded in $L^{2}(\mathbb{R}^{3})$, which leads to considerable difficulties in the analysis.
In order to handle the terms involving $L^{-1}_{M}$, we will
apply the Burnett functions and the fast decay properties about the velocity of the Burnett functions.
As in \cite{G1} for the dynamical stability of global Maxwellians, we need use the weight function $w=\langle v\rangle^{\gamma+2}$ as \eqref{2.14} to overcome the dissipation deficiency in case of Coulomb potentials and handle
velocity derivatives of the free transport term $v_{1}\partial_{x}f$.
Furthermore, we use the decompositions $F=M+\overline{G}+\sqrt{\mu}f$ as in \cite{DuanY} to
improve the decompositions in \cite{LX,XinZeng} such that we can simplify the related energy estimates, and some basic estimates developed by Guo \cite{G1} around global Maxwellians can be applied to treat the derivatives estimates conveniently.

The new scaling transformation of the independent variables takes the form of $y=\epsilon^{-a}x$, $\tau=\epsilon^{-a}t$ involving with a free parameter $a\in[\frac{2}{3},1]$. The result in Theorem \ref{thm1.1} shows that
the solution converges to the local Maxwellian defined by the rarefaction wave of the Euler system at a rate
$\epsilon^{\frac{3}{5}-\frac{2}{5}a}|\ln\epsilon|$. In particular, we can obtain the fastest convergence rate $\epsilon^{\frac{1}{3}}|\ln\epsilon|$
if one takes $a=\frac{2}{3}$. As mentioned before, the scaling argument through the change of variables $y=\epsilon^{-\frac{3}{4}}x$ and $\tau=\epsilon^{-\frac{3}{4}}t$  was first used in \cite{Xin} to study the zero dissipation limit to the rarefaction wave for the one-dimensional compressible Navier-Stokes system, where the obtained convergence rate is $\epsilon^{\frac{1}{4}}|\ln\epsilon|$. Note that if one takes $a=\frac{3}{4}$ in terms of Theorem \ref{thm1.1}, then
we obtain the convergence rate $\epsilon^{\frac{3}{10}}|\ln\epsilon|$ which is still a little sharper than the one in \cite{Xin}. This is basically due to the detailed energy analysis  such that both the lower and higher order estimates possess the same convergence rate of the Knudsen number $\epsilon$. We also notice that under the scaling transformations
$y=\epsilon^{-1}x$, $\tau=\epsilon^{-1}t$ in \cite{XinZeng} and $y=\epsilon^{-\frac{2}{3}}x$, $\tau=\epsilon^{-\frac{2}{3}}t$ in \cite{LX}
for the Boltzmann equation with hard sphere model, the convergence rates $\epsilon^{\frac{1}{5}}|\ln\epsilon|$  and $\epsilon^{\frac{1}{3}}|\ln\epsilon|^{2}$ are obtained respectively.
Therefore, Theorem \ref{thm1.1} also implies that the results of \cite{Xin,XinZeng,LX} can be improved to give rise to the faster rate $\epsilon^{\frac{1}{3}}|\ln\epsilon|$ by choosing $a=\frac{2}{3}$. At this moment,  we would remark that we expect that such scaling argument also could be applied to study convergence to
basic wave patterns (i.e., rarefaction waves, contact waves, and shock waves) for the Boltzmann equation, Navier-Stokes system, radiative hydrodynamic equations and many other related models.
Once we use the scaling transformation \eqref{2.1} for an arbitrary parameter $a\in [\frac{2}{3},1]$, we have to deal with some difficulties caused by the higher order derivatives estimates such as \eqref{3.85}. For this purpose, we design the energy functional $\mathcal{E}_{2}(\tau)$ and the corresponding dissipation functional $\mathcal{D}_{2}(\tau)$  involving the Knudsen number $\epsilon$, see \eqref{2.18} and \eqref{2.19}. The desired goal is to obtain the uniform a priori estimate \eqref{thm.enineq} and then derive the convergence rate \eqref{1.27}.

\subsection{Organization of the paper}
The rest of this paper is arranged as follows. In Section \ref{sec.2}, we will
reformulate the system \eqref{1.1} and introduce a scaling for the new independent variable and the perturbation.
In Section \ref{sec.3}, we will establish the a priori estimates including low order energy estimates, high order energy estimates and the weighted energy
estimates. In Section \ref{sec.4}, we will establish the existence of global-in-time solutions as well as the convergence to the local Maxwellian defined by the rarefaction wave of the Euler system uniformly away from $t=0$ as $\epsilon\to 0$. In the appendix Section \ref{sec.5}, we will give some basic estimates frequently used in the previous sections.

\section{Reformulation of the problem}\label{sec.2}

\subsection{Scaling and reformultion}
In this section, we will reformulate the system and introduce a scaling for the independent variable
and the perturbation. Firstly, we define the scaled independent variables
\begin{equation}
\label{2.1}
y=\frac{x}{\epsilon^{a}}, \quad \tau=\frac{t}{\epsilon^{a}}, \quad \mbox{for} \quad a\in[\frac{2}{3},1].
\end{equation}
Here the range of the constant $a$ is determined by \eqref{3.92} and \eqref{4.2}.
\par
Correspondingly, we set the scaled perturbation $(\widetilde{\rho},\widetilde{u},\widetilde{\theta})=
(\widetilde{\rho},\widetilde{u},\widetilde{\theta})(\tau,y)$
and $\widetilde{G}=\widetilde{G}(\tau,y,v)$ as
\begin{equation}
\label{2.2}
\begin{cases}
\widetilde{\rho}=\rho(t,x)-\bar{\rho}(t,x),
\\
\widetilde{u}=u(t,x)-\bar{u}(t,x),
\\
\widetilde{\theta}=\theta(t,x)-\bar{\theta}(t,x),
\\
\widetilde{G}=G(t,x,v)-\overline{G}(t,x,v),\quad\widetilde{G}=\sqrt{\mu}f(\tau,y,v).
\end{cases}
\end{equation}
Here the term $\overline{G}=\overline{G}(t,x,v)$ is defined as
\begin{equation}
\label{2.3}
\overline{G}=\epsilon^{1-a} L_{M}^{-1}P_{1}v_{1}M\big\{\frac{|v-u|^{2}
\bar{\theta}_{y}}{2R\theta^{2}}+\frac{(v-u)\cdot\bar{u}_{y}}{R\theta}\big\}.
\end{equation}
We remark that in \eqref{2.2}, $\overline{G}$ is subtracted from $G$ because the time decay of
$\|(\bar{u}_{y},\bar{\theta}_{y})\|^{2}$  is $\epsilon^{a}(\delta+\epsilon^{a}\tau)^{-1}$ by Lemma \ref{lem5.3}, which is not integrable about the time $\tau$.
\par
Subtracting \eqref{1.24} from system \eqref{1.13} and using the scaling \eqref{2.1}, we can obtain
\begin{equation}
\label{2.4}
\begin{cases}
\widetilde{\rho}_{\tau}+\bar{\rho} \widetilde{u}_{1y}+\bar{\rho}_{y} \widetilde{u}_{1}=-J_{1},
\\
\widetilde{u}_{1\tau}+\bar{u}_{1}\widetilde{u}_{1y}
+\frac{2}{3}\widetilde{\theta}_{y}+\frac{2\bar{\theta}}{3\bar{\rho}}\widetilde{\rho}_{y}
=-J_{2}-\frac{1}{\rho}\int_{\mathbb{R}^3}  v^{2}_{1}G_{y}dv,
\\
\widetilde{u}_{i\tau}+\widetilde{u}_{1}\widetilde{u}_{iy}+\bar{u}_{1}\widetilde{u}_{iy}
=-\frac{1}{\rho}\int_{\mathbb{R}^3}  v_{1}v_{i}G_{y}dv, ~~i=2,3,
\\
\widetilde{\theta}_{\tau}+\frac{2}{3}\bar{\theta}
\widetilde{u}_{1y}+\bar{u}_{1}\widetilde{\theta}_{y}=-J_{3}-\frac{1}{\rho}\int_{\mathbb{R}^3} \frac{1}{2}v_{1}v\cdot(v-2u)G_{y}dv,
\end{cases}
\end{equation}
where
\begin{equation}
\label{2.5}
\begin{cases}
J_{1}=(\widetilde{\rho}\widetilde{u}_{1})_{y}+\bar{u}_{1}\widetilde{\rho}_{y}+\bar{u}_{1y}\widetilde{\rho},
\\
J_{2}=\widetilde{u}_{1}\bar{u}_{1y}+\widetilde{u}_{1}
\widetilde{u}_{1y}+\frac{2}{3}\rho_{y}\frac{\bar{\rho}\widetilde{\theta}-\widetilde{\rho}\bar{\theta}}
{\rho\bar{\rho}},
\\
J_{3}=\frac{2}{3}(\widetilde{\theta}\bar{u}_{1y}+\widetilde{\theta}\widetilde{u}_{1y})
+(\widetilde{\theta}_{y}\widetilde{u}_{1}+\bar{\theta}_{y}\widetilde{u}_{1}).
\end{cases}
\end{equation}
Moreover, we also get from \eqref{1.16}, \eqref{1.24} and \eqref{2.1} that
\begin{equation}
\label{2.6}
\begin{cases}
\widetilde{\rho}_{\tau}+\bar{\rho} \widetilde{u}_{1y}+\bar{\rho}_{y} \widetilde{u}_{1}=-J_{1},
\\
\widetilde{u}_{1\tau}+\bar{u}_{1}\widetilde{u}_{1y}
+\frac{2}{3}\widetilde{\theta}_{y}+\frac{2\bar{\theta}}{3\bar{\rho}}\widetilde{\rho}_{y}
=-J_{2}+\epsilon^{1-a}\frac{4}{3\rho}(\mu(\theta)u_{1y})_{y}-\frac{1}{\rho}(\int_{\mathbb{R}^{3}} v^{2}_{1}L^{-1}_{M}\Theta dv)_{y},
\\
\widetilde{u}_{i\tau}+\widetilde{u}_{1}\widetilde{u}_{iy}+\bar{u}_{1}\widetilde{u}_{iy}
=\epsilon^{1-a}\frac{1}{\rho}(\mu(\theta)u_{iy})_{y}-\frac{1}{\rho}(\int_{\mathbb{R}^3} v_{1}v_{i}L^{-1}_{M}\Theta dv)_{y}, ~~i=2,3,
\\
\widetilde{\theta}_{\tau}+\frac{2}{3}\bar{\theta}
\widetilde{u}_{1y}+\bar{u}_{1}\widetilde{\theta}_{y}=-J_{3}+\epsilon^{1-a}\frac{1}{\rho}(\kappa(\theta)\theta_{y})_{y}
+\epsilon^{1-a}\frac{4}{3\rho}\mu(\theta)u^{2}_{1y}
\\
\hspace{0.5cm}+\epsilon^{1-a}\frac{1}{\rho}\mu(\theta)(u^{2}_{2y}+u^{2}_{3y})-\frac{1}{\rho}
(\int_{\mathbb{R}^3}v_{1}\frac{|v|^{2}}{2}L^{-1}_{M}\Theta dv)_{y}+\frac{1}{\rho}u(\int_{\mathbb{R}^3} v_{1}v L^{-1}_{M}\Theta dv)_{y}.
\end{cases}
\end{equation}
Here $J_{1},J_{2}$ and $J_{3}$ are defined in \eqref{2.5} and $\Theta$ can be rewritten as
\begin{equation*}
\Theta=\epsilon^{1-a}\partial_{\tau}G+\epsilon^{1-a}P_{1}(v_{1}\partial_{y}G)-Q(G,G).
\end{equation*}
On the other hand, we need to derive the equation of the microscopic component $f$ as in \eqref{2.2}. For this, we first denote
\begin{equation}
\label{2.8}
\Gamma(h,g)=\frac{1}{\sqrt{\mu}}Q(\sqrt{\mu}h,\sqrt{\mu}g),
\quad \mathcal{L}h=\Gamma(h,\sqrt{\mu})+\Gamma(\sqrt{\mu},h).
\end{equation}
This together with the definition of $L_{M}$ in \eqref{1.14} imply that
\begin{equation}
\label{2.9}
\frac{1}{\sqrt{\mu}}L_{M}(\sqrt{\mu}f)=
\frac{1}{\sqrt{\mu}}\{Q(M,\sqrt{\mu}f)+Q(\sqrt{\mu}f,M)\}
=\mathcal{L}f+\Gamma(f,\frac{M-\mu}{\sqrt{\mu}})+
\Gamma(\frac{M-\mu}{\sqrt{\mu}},f).
\end{equation}
From \eqref{1.14} and \eqref{2.1}, one can easily see that
\begin{equation}
\label{2.10}
\partial_{\tau}G+P_{1}(v_{1}\partial_{y}G)+P_{1}(v_{1}\partial_{y}M)=\epsilon^{a-1}L_{M}G+\epsilon^{a-1}Q(G,G).
\end{equation}
By using \eqref{2.9}, $P_{1}=\mathbf{I}-P_{0}$ and $G=\overline{G}+\sqrt{\mu}f$,
we can rewrite equation \eqref{2.10} as
\begin{align}
\label{2.11}
\partial_{\tau}f+v_{1}\partial_{y}f
-\epsilon^{a-1}\mathcal{L}f&=\epsilon^{a-1}\Gamma(f,\frac{M-\mu}{\sqrt{\mu}})+
\epsilon^{a-1}\Gamma(\frac{M-\mu}{\sqrt{\mu}},f)
+\epsilon^{a-1}\Gamma(\frac{G}{\sqrt{\mu}},\frac{G}{\sqrt{\mu}})
\nonumber\\
&\quad+\frac{P_{0}(v_{1}\sqrt{\mu}\partial_{y}f)}{\sqrt{\mu}}
-\frac{1}{\sqrt{\mu}}P_{1}v_{1}M\big\{\frac{|v-u|^{2}
\widetilde{\theta}_{y}}{2R\theta^{2}}+\frac{(v-u)\cdot\widetilde{u}_{y}}{R\theta}\big\}
\nonumber\\
&\quad-\frac{P_{1}(v_{1}\partial_{y}\overline{G})}{\sqrt{\mu}}
-\frac{\partial_{\tau}\overline{G}}{\sqrt{\mu}}.
\end{align}
Here we have used the fact that
\begin{equation*}
P_{1}(v_{1}\partial_{y}M)=P_{1}v_{1}M\big\{\frac{|v-u|^{2}
\widetilde{\theta}_{y}}{2R\theta^{2}}+\frac{(v-u)\cdot\widetilde{u}_{y}}{R\theta}\big\}
+\epsilon^{a-1}L_{M}\overline{G}.
\end{equation*}
Finally, we obtain by \eqref{1.1} and the scaling \eqref{2.1} that
\begin{equation}
\label{2.13}
\partial_{\tau}F+v_{1}\partial_{y}F=\epsilon^{a-1}Q(F,F).
\end{equation}

\subsection{Notations and norms}
The following notations are needed in the energy analysis for convenience of presentation.
We shall use $\langle \cdot , \cdot \rangle$  to denote the standard $L^{2}$ inner product in $\mathbb{R}_{v}^{3}$
with its corresponding $L^{2}$ norm $|\cdot|_2$. We also use $( \cdot , \cdot )$ to denote $L^{2}$ inner product in
$\mathbb{R}_{y}$ or $\mathbb{R}_{y}\times \mathbb{R}_{v}^{3}$  with its corresponding $L^{2}$ norm $\|\cdot\|$.
Let $\alpha$ and $\beta$ be multi indices $\alpha=[\alpha_{1},\alpha_{2}]$ and $\beta=[\beta_{1},\beta_{2},\beta_{3}]$,
respectively. Denote a high order derivative
\begin{align*}
\partial_{\beta}^{\alpha}=\partial_{\tau}^{\alpha_{1}}\partial_{y}^{\alpha_{2}}
\partial_{v_{1}}^{\beta_{1}}\partial_{v_{2}}^{\beta_{2}}\partial_{v_{3}}^{\beta_{3}}.
\end{align*}
If each component of $\beta$ is not greater than the corresponding one  of
$\overline{\beta}$, we use the standard notation
$\beta\leq\overline{\beta}$. And $\beta<\overline{\beta}$ means that
$\beta\leq\overline{\beta}$ and $|\beta|<|\overline{\beta}|$.
$C^{\bar\beta}_{\beta}$ is the usual  binomial coefficient.
Throughout the paper, generic positive constants are denoted by
$C$ (generally large) and  $c$ (generally small)  which are
independent of time $\tau$ and $\epsilon$ unless otherwise stated.
The notation  $A\approx B$ is used to denote that there exists $c_{0}>1$
such that $c_{0}^{-1}B\leq A\leq c_{0}B$.
Motivated by \cite{G1}, we introduce the following velocity weight function
\begin{equation}
\label{2.14}
w=w(v)\equiv\langle v\rangle^{\gamma+2},\quad
\langle v\rangle=\sqrt{1+|v|^{2}}.
\end{equation}
Denote weighted $L^{2}$ norms as
\begin{equation*}
|g|_{2,\ell}^{2}\equiv\int_{\mathbb{R}^{3}}w^{2\ell}|g|^{2}\,dv, \quad  \|g\|_{2,\ell}^{2}\equiv\int_{\mathbb{R}}|g|_{2,\ell}^{2}\,dy.
\end{equation*}
The Landau collision frequency is
\begin{equation}
\label{2.15}
\sigma^{ij}(v)=\phi^{ij}\ast\mu=\int_{\mathbb{R}^{3}}\phi^{ij}(v-v_{*})\mu(v_{*})\,dv_{*},\quad 1\leq i,j\leq 3,
\end{equation}
where $\phi^{ij}$ is given in \eqref{1.3}.
We remark that $[\sigma^{ij}(v)]_{1\leq i,j\leq 3}$ is a positive-definite self-adjoint matrix. With \eqref{2.15}, we define the weighted dissipation norms as
\begin{equation*}
|g|_{\sigma,\ell}^{2}\equiv\sum^{3}_{i,j=1}\int_{\mathbb{R}^{3}}w^{2\ell}\{\sigma^{ij}\partial_{i}g\partial_{j}g+
\sigma^{ij}\frac{v_{i}}{2}\frac{v_{j}}{2}|g|^{2}\}\,dv, \quad  \|g\|_{\sigma,\ell}^{2}\equiv\int_{\mathbb{R}}|g|_{\sigma,\ell}^{2}\,dy.
\end{equation*}
And let $|g|_{\sigma}=|g|_{\sigma,0}$ and $\|g\|_{\sigma}=\|g\|_{\sigma,0}$.
From \cite[Lemma 5, p.315]{SG}, one has
\begin{equation}
\label{2.17}
|g|_{\sigma}\approx\Big|\langle v\rangle^{\frac{\gamma+2}{2}}g\Big|_{2}+\Big|\langle v\rangle^{\frac{\gamma}{2}}\nabla_{v}g\cdot\frac{v}{|v|}\Big|_{2}
+\Big|\langle v\rangle^{\frac{\gamma+2}{2}}\nabla_{v}g\times\frac{v}{|v|}\Big|_{2}.
\end{equation}
\par
Now, we define the instant energy functional $\mathcal{E}_{2}(\tau)$ by
\begin{align}
\label{2.18}
\mathcal{E}_{2}(\tau)&=\sum_{|\alpha|\leq1}\|\partial^{\alpha}(\widetilde{\rho},\widetilde{u},\widetilde{\theta})(\tau)\|^{2}
+\epsilon^{2-2a}\sum_{|\alpha|=2}\{\|\partial^{\alpha}(\widetilde{\rho},\widetilde{u},\widetilde{\theta})(\tau)\|^{2}+\|\partial^{\alpha}f(\tau)\|^{2}\}
\nonumber\\
&\quad+\sum_{|\alpha|\leq 1}\|\partial^{\alpha}f(\tau)\|^{2}
+\sum_{|\alpha|+|\beta|\leq 2,|\beta|\geq1}\|\partial^{\alpha}_{\beta}f(\tau)\|_{2,|\beta|}^{2}.
\end{align}
As usual, the instant energy functional $\mathcal{E}_{2}(\tau)$ is assumed to
be small enough a priori. And this will be closed by the energy
estimates in the end.
The corresponding dissipation rate $\mathcal{D}_{2}(\tau)$ is given by
\begin{align}
\label{2.19}
\mathcal{D}_{2}(\tau)&=\epsilon^{1-a}\sum_{1\leq|\alpha|\leq 2}\|\partial^{\alpha}(\widetilde{\rho},\widetilde{u},\widetilde{\theta})(\tau)\|^{2}
+\epsilon^{1-a}\sum_{|\alpha|=2}\|\partial^{\alpha}f(\tau)\|_{\sigma}^{2}
\nonumber\\
&\quad+\epsilon^{a-1}\sum_{|\alpha|\leq 1}\|\partial^{\alpha}f(\tau)\|_{\sigma}^{2}+\epsilon^{a-1}\sum_{|\alpha|+|\beta|\leq 2,|\beta|\geq1}
\|\partial^{\alpha}_{\beta}f(\tau)\|_{\sigma,|\beta|}^{2}.
\end{align}

\section{A priori estimates}\label{sec.3}
This section is devoted to deducing the a priori estimates for the equation \eqref{2.13} around the smooth rarefaction wave. To this end, we first
choose the initial value of the  equation \eqref{2.13} as
\begin{equation}
\label{3.1}
F(0,y,v)\equiv F_{0}(y,v)=M_{[\bar{\rho},\bar{u},\bar{\theta}]}(0,y,v),
\end{equation}
such that $(\widetilde{\rho},\widetilde{u},\widetilde{\theta})(0,y)$ and $f(0,y,v)$ satisfying
\begin{equation}
\label{3.2}
\mathcal{E}_{2}(0)\leq \epsilon.
\end{equation}
Here $\mathcal{E}_{2}(\tau)$ is defined by \eqref{2.18} and $(\bar{\rho},\bar{u},\bar{\theta})$ with $\bar{u}=(\bar{u}_{1},0,0)$ is the smooth approximation rarefaction wave
given by \eqref{1.23}. Due to the smoothness of $(\bar{\rho},\bar{u},\bar{\theta})$, the
local existence of the unique solution to the Cauchy problem \eqref{2.13} and \eqref{3.1} can be obtained by a straightforward modification
of the arguments in \cite{G1} .  To obtain the global-in-time existence of solution,
it suffices to get uniform a priori estimates on solutions.

Throughout this section, we need to find out small positive constants $k$, $\delta_{0}$ and $\epsilon_{0}$ with $0<\epsilon_{0}\ll k\ll1$ and  $0<\delta_{0}\ll 1$, independent of $\epsilon$, $\delta$ and $\tau$, then we choose
\begin{equation}
\label{3.3}
\delta=\frac{1}{k}\epsilon^{\frac{3}{5}-\frac{2}{5}a}
\end{equation}
for a parameter $a\in[\frac{2}{3},1]$, where $\delta$ is given in \eqref{1.21}, and we further let $\epsilon$ be arbitrarily chosen such that $0<\epsilon<\epsilon_{0}$ and $0<\delta<\delta_{0}$.
We now make the a priori assumption:
\begin{equation}
\label{3.4}
\sup_{0\leq\tau\leq\tau_{1}}\mathcal{E}_{2}(\tau)\leq k^{\frac{1}{6}}\epsilon^{\frac{6}{5}-\frac{4}{5}a}
\end{equation}
for $\tau_{1}\in(0,+\infty)$, where $\mathcal{E}_{2}(\tau)$ is defined by \eqref{2.18}.
In what follows we formally explain why one has to choose such $\delta$ as \eqref{3.3} and the a priori assumption as \eqref{3.4}. Indeed, if one assumes that
$$
\sup_{0\leq\tau\leq\tau_{1}}\mathcal{E}_{2}(\tau)\leq O(1)\epsilon^{q}
$$
with a constant $q>0$, then it follows from this assumption and Lemma \ref{lem5.2} that
\begin{align*}
&\|(\rho,u,\theta)(t,x)-(\rho^{R},u^{R},\theta^{R})(\frac{x}{t})\|_{L^{\infty}}\\
&\leq\|(\widetilde{\rho},\widetilde{u},\widetilde{\theta})(\tau,y)\|_{L^{\infty}}
+\|(\bar{\rho},\bar{u},\bar{\theta})(t,x)-(\rho^{R},u^{R},\theta^{R})(\frac{x}{t})\|_{L^{\infty}}
\\
&\leq C\epsilon^{\frac{q}{2}}+Ct^{-1}\delta(\ln(1+t)+|\ln\delta|),
\end{align*}
for any $t>0$. Hence the above estimate in vanishing Knudsen number $\epsilon>0$ is optimal by taking $\delta=O(1)\epsilon^{\frac{q}{2}}$.
On the other hand, we have to deal with the slow time decay of the term in \eqref{3.16} in the way that
\begin{align*}
\epsilon^{1-a}\int^{\infty}_{0}\|\widetilde{\theta}\|^{\frac{2}{3}}\|\bar{\theta}_{yy}\|^{\frac{4}{3}}_{L^{1}}\,d\tau
&\leq C\epsilon^{1-a}\int^{\infty}_{0}\epsilon^{\frac{q}{3}}\epsilon^{\frac{4}{3}a}(\delta+\epsilon^{a}\tau)^{-\frac{4}{3}}\,d\tau
\\
&\leq C\epsilon^{1-a+\frac{1}{3}q+\frac{1}{3}a}\delta^{-\frac{1}{3}}=O(1)\epsilon^{1-\frac{2}{3}a+\frac{1}{6}q},
\end{align*}
where we have replaced $\delta=O(1)\epsilon^{\frac{q}{2}}$ in the last identity.
To close the a priori assumption, we need to require that
\begin{align*}
\epsilon^{1-\frac{2}{3}a+\frac{1}{6}q}\leq \epsilon^{q}, \quad \mbox{that~~~is}\quad q\leq\frac{6}{5}-\frac{4}{5}a.
\end{align*}
Notice that the convergence rate is the fastest  by choosing $q=\frac{6}{5}-\frac{4}{5}a$.
Hence, we can obtain the sharp convergence rate under the condition of \eqref{3.3} and \eqref{3.4}.
\begin{remark}
By the a priori assumption \eqref{3.4} and in view of \eqref{1.26a}, we have
\begin{equation}
\label{1.26}
\left\{\begin{aligned}
 &\sup_{t\geq 0,x\in \R}\{|\rho(t,x)-1|+|u(t,x)|+|\theta(t,x)-\frac{3}{2}|\}<2\eta_{0},\\
& \frac{1}{2}\sup_{t\geq0,x\in\mathbb{R}}\theta(t,x)<\frac{3}{2}<\inf_{t\geq0,x\in\mathbb{R}}\theta(t,x).
\end{aligned}\right.
\end{equation}
due to the smallness of $\epsilon$ and $k$, where $\eta_{0}$ is the small constant given in \eqref{1.26a}.
We point out that \eqref{1.26} will be frequently used in the later energy estimates.
\end{remark}
\par
From now on, we will focus on the reformulated system \eqref{2.4}, \eqref{2.6}, \eqref{2.11} and \eqref{2.13}
with initial data \eqref{3.1}.
We will first derive the lower order energy estimates for the macroscopic
component $(\widetilde{\rho},\widetilde{u},\widetilde{\theta})$  and the microscopic component $f$ in Subsection \ref{sec3.1}.
Then, Subsection \ref{sec3.2} is devoted to obtaining the high order energy estimates
of $(\widetilde{\rho},\widetilde{u},\widetilde{\theta})$  and $f$.
The weighted energy estimates of $f$ will be given in Subsection \ref{sec3.3}.

\subsection{Lower order energy estimates}\label{sec3.1}
Now we will derive the lower order estimates for $(\widetilde{\rho},\widetilde{u},\widetilde{\theta})$  by the entropy and entropy flux.
As in \cite{LTP2,LTP3}, the following macroscopic entropy $S$ will be estimated for the lower order energy estimates. Set
$$
-\frac{3}{2}\rho S=\int_{\mathbb{R}^{3}}M\ln Mdv.
$$
Multiplying \eqref{2.13} by $\ln M$, integrating over $v$ and making a direct calculation, it holds that
\begin{equation*}
(-\frac{3}{2}\rho S)_{\tau}+(-\frac{3}{2}\rho u_{1}S)_{y}
+(\int_{\mathbb{R}^{3}} v_{1}G\ln Mdv)_{y}-\int_{\mathbb{R}^{3}} v_{1}G(\ln M)_{y}dv=0,
\end{equation*}
where
\begin{equation}
\label{3.6}
S=-\frac{2}{3}\ln\rho+\ln(\frac{4\pi}{3}\theta)+1,\quad
p=\frac{2}{3}\rho\theta=\frac{1}{2\pi e}\rho^{\frac{5}{3}}\exp (S).
\end{equation}
In terms of the scaling transformation \eqref{2.1}, we can rewrite the conservation laws \eqref{1.16} as $X_{\tau}+Y_{y}=$
\begin{equation*}
\begin{pmatrix}
0
\\
\frac{4}{3}\epsilon^{1-a}(\mu(\theta)u_{1y})_{y}-(\int v^{2}_{1}L^{-1}_{M}\Theta dv)_{y}
\\
\epsilon^{1-a}(\mu(\theta)u_{2y})_{y}-(\int v_{1}v_{2}L^{-1}_{M}\Theta dv)_{y}
\\
\epsilon^{1-a}(\mu(\theta)u_{3y})_{y}-(\int v_{1}v_{3}L^{-1}_{M}\Theta dv)_{y}
\\
\epsilon^{1-a}\big\{(\kappa(\theta)\theta_{y})_{y}+\frac{4}{3}(\mu(\theta)u_{1}u_{1y})_{y}
+\sum\limits_{i=2}^{3}(\mu(\theta)u_{i}u_{iy})_{y}\big\}-\frac{1}{2}(\int v_{1}|v|^{2}L^{-1}_{M}\Theta dv)_{y}
\end{pmatrix}.
\end{equation*}
Here
\begin{align*}
X=&(X_{0},X_{1},X_{2},X_{3},X_{4})^{t}=(\rho,\rho u_{1},\rho u_{2},\rho u_{3},\rho(\theta+\frac{|u|^{2}}{2}))^{t},
\\
Y=&(Y_{0},Y_{1},Y_{2},Y_{3},Y_{4})^{t}=(\rho u_{1},\rho u^{2}_{1}+p,\rho u_{1}u_{2},\rho u_{1} u_{3},\rho u_{1}(\theta+\frac{|u|^{2}}{2})+pu_{1})^{t},
\end{align*}
where $(\cdot,\cdot,\cdot)^{t}$ is the transpose of the vector $(\cdot,\cdot,\cdot)$.
We define an entropy-entropy flux pair $(\eta,q)$ around a Maxwellian $\overline{M}=M_{[\bar{\rho},\bar{u},\bar{S}]}$
$(\bar{u}_{2}=\bar{u}_{3}=0)$ as
\begin{equation}
\begin{cases}
\label{3.8}
\eta(\tau,y)=\bar{\theta}\{-\frac{3}{2}\rho S+\frac{3}{2}\bar{\rho}\bar{S}+\frac{3}{2}\nabla_{X}(\rho S)|_{X=\bar{X}}\cdot(X-\bar{X})\},
\\
q(\tau,y)=\bar{\theta}\{-\frac{3}{2}\rho u_{1}S+\frac{3}{2}\bar{\rho}\bar{u}_{1}\bar{S}+\frac{3}{2}\nabla_{X}(\rho S)|_{X=\bar{X}}\cdot(Y-\bar{Y})\}.
\end{cases}
\end{equation}
Hence, by using \eqref{3.8} and \eqref{3.6}, we can obtain
\begin{equation*}
(\rho S)_{X_{0}}=S+\frac{|u|^{2}}{2\theta}-\frac{5}{3}, \quad
(\rho S)_{X_{i}}=-\frac{u_{i}}{\theta}, i=1,2,3, \quad (\rho S)_{X_{4}}=\frac{1}{\theta},
\end{equation*}
and
\begin{equation}
\label{3.10}
\begin{cases}
\eta(\tau,y)=\frac{3}{2}\{\rho\theta-\bar{\theta}\rho S+\rho[(\bar{S}-\frac{5}{3})\bar{\theta}
+\frac{|u-\bar{u}|^{2}}{2}]+\frac{2}{3}\bar{\rho}\bar{\theta}\}
\\
\ \ \ \ \ \ \ \ \ =\rho\bar{\theta}\Phi(\frac{\bar{\rho}}{\rho})+\frac{3}{2}\rho\bar{\theta}\Phi(\frac{\theta}{\bar{\theta}})
+\frac{3}{4}\rho|u-\bar{u}|^{2},
\\
q(\tau,y)=u_{1}\eta(\tau,y)+(u_{1}-\bar{u}_{1})(\rho\theta-\bar{\rho}\bar{\theta}),
\end{cases}
\end{equation}
where the convex function $\Phi(s)$ is defined as $\Phi(s)=s-\ln s-1$. From \eqref{3.10}, there exists a constant
$c_{1}>1$ such that
\begin{equation}
\label{3.11}
c_{1}^{-1}\|(\widetilde{\rho},\widetilde{u},\widetilde{\theta})\|^{2}\leq \eta(\tau,y)\leq c_{1}\|(\widetilde{\rho},\widetilde{u},\widetilde{\theta})\|^{2}.
\end{equation}
In view of the definition of \eqref{3.8}, we have by a direct computation that
\begin{align*}
&\eta(\tau,y)_{\tau}+q(\tau,y)_{y}-\nabla_{[\bar{\rho},\bar{u},\bar{S}]}\eta(\tau,y)\cdot(\bar{\rho},\bar{u},\bar{S})_{\tau}
-\nabla_{[\bar{\rho},\bar{u},\bar{S}]}q(\tau,y)\cdot(\bar{\rho},\bar{u},\bar{S})_{y}
\nonumber\\
&=\bar{\theta}\{(-\frac{3}{2}\rho S)_{\tau}+(-\frac{3}{2}\rho u_{1}S)_{y}\}+\frac{3}{2}\bar{\theta}
\{\nabla_{X}(\rho S)|_{X=\bar{X}}(X_{\tau}+Y_{y})\}.
\end{align*}
A direct but tedious computation shows that
\begin{align}
\label{3.13}
\eta&(\tau,y)_{\tau}+q(\tau,y)_{y}+\epsilon^{1-a}\frac{2\bar{\theta}}{\theta}\mu(\theta)\widetilde{u}^{2}_{1y}
+\epsilon^{1-a}\frac{3\bar{\theta}}{2\theta}\sum_{i=2}^{3}\mu(\theta)\widetilde{u}^{2}_{iy}
+\epsilon^{1-a}\frac{3\bar{\theta}}{2\theta^{2}}\kappa(\theta)\widetilde{\theta}^{2}_{y}
\nonumber\\
&-\{\nabla_{[\bar{\rho},\bar{u},\bar{S}]}\eta(\tau,y)\cdot(\bar{\rho},\bar{u},\bar{S})_{\tau}
+\nabla_{[\bar{\rho},\bar{u},\bar{S}]}q(\tau,y)\cdot(\bar{\rho},\bar{u},\bar{S})_{y}\}-(\cdots)_{y}=\sum_{i=1}^{4}H_{i},
\end{align}
where we have denoted
\begin{align*}
\sum_{i=1}^{4}H_{i}
&=\epsilon^{1-a}\Big\{\frac{3\kappa(\theta)}{2\theta^{2}}
(\bar{\theta}_{y}+\widetilde{\theta}_{y})\bar{\theta}_{y}\widetilde{\theta}
-\frac{3\bar{\theta}}{2\theta^{2}}\kappa(\theta)\widetilde{\theta}_{y}\bar{\theta}_{y}\Big\}
\nonumber\\
&\quad+\epsilon^{1-a}\Big\{\frac{2\mu(\theta)}{\theta}(\bar{u}_{1y}+\widetilde{u}_{1y})
\bar{u}_{1y}\widetilde{\theta}-\frac{2\bar{\theta}}{\theta}\mu(\theta)\widetilde{u}_{1y}\bar{u}_{1y}\Big\}
\nonumber\\
&\quad+\Big\{\frac{3}{2}(\frac{\widetilde{\theta}}{\theta})_{y}
\int_{\mathbb{R}^{3}} (\frac{1}{2}v_{1}|v|^{2}-\sum^{3}_{i=1}u_{i} v_{1}v_{i})L^{-1}_{M}\Theta dv\Big\}
\nonumber\\
&\quad+\Big\{(\frac{3}{2}\sum^{3}_{i=1}\widetilde{u}_{iy}-\frac{3}{2}\frac{\widetilde{\theta}}{\theta}\sum^{3}_{i=1}u_{iy})
\int_{\mathbb{R}^{3}} v_{1}v_{i}L^{-1}_{M}\Theta dv\Big\}.
\end{align*}
Here the notation $(\cdots)_{y}$ represents the term in the conservative form so that
it vanishes after integration. In the following energy analysis, we shall assume a priori estimates that $\mathcal{E}_{2}(\tau)$ is small enough due to \eqref{3.4}, and we have from this and \eqref{1.26} that $(\rho, u, \theta)$ and $(\bar{\rho}, \bar{u},  \bar{\theta})$ are close enough to the state $(1,0,\frac{3}{2})$.

By the similar arguments as \cite{LH1},
there exists $c_{2}>0$ such that
\begin{align}
\label{3.14}
&-\{\nabla_{[\bar{\rho},\bar{u},\bar{S}]}\eta(\tau,y)\cdot(\bar{\rho},\bar{u},\bar{S})_{\tau}
+\nabla_{[\bar{\rho},\bar{u},\bar{S}]}q(\tau,y)\cdot(\bar{\rho},\bar{u},\bar{S})_{y}\}
\nonumber\\
&=\frac{3}{2}\rho\bar{u}_{1y}(u_{1}-\bar{u}_{1})^{2}+\frac{2}{3}\rho\bar{\theta}\bar{u}_{1y}\Phi(\frac{\bar{\rho}}{\rho})
+\rho\bar{\theta}\bar{u}_{1y}\Phi(\frac{\theta}{\bar{\theta}})+\frac{3}{2}\rho\bar{\theta}_{y}(u_{1}-\bar{u}_{1})(\frac{2}{3}\ln\frac{\bar{\rho}}{\rho}+\ln\frac{\theta}{\bar{\theta}})
\nonumber\\
&\geq c_{2}\bar{u}_{1y}(\widetilde{\rho}^{2}+\widetilde{u}_{1}^{2}+\widetilde{\theta}^{2}).
\end{align}
Since both $\mu(\theta)$ and $\kappa(\theta)$ are smooth functions of $\theta$, there exists a constant $c_{3}>1$ such that $\mu(\theta),\kappa(\theta)\in[c^{-1}_{3},c_{3}]$. Plugging \eqref{3.14} into \eqref{3.13} and integrating the resulting equation with respect to $y$,
we can obtain
\begin{equation}
\label{3.15}
\frac{d}{d\tau}\int_{\mathbb{R}}\eta(\tau,y) dy+c\epsilon^{1-a}\|(\widetilde{u}_{y},\widetilde{\theta}_{y})\|^{2}
+c_{2}\|\sqrt{\bar{u}_{1y}}(\widetilde{\rho},\widetilde{u}_{1},\widetilde{\theta})\|^{2}
\leq \sum_{i=1}^{4}\int_{\mathbb{R}} H_{i}\,dy.
\end{equation}
We will estimate the terms of \eqref{3.15} involving $H_{i}$. By the integration by parts, the
a priori assumption \eqref{3.4}, Lemma \ref{lem5.3} and  the Cauchy-Schwarz inequality, one gets that
\begin{align*}
\int_{\mathbb{R}}H_{1}dy
&=\epsilon^{1-a}\int_{\mathbb{R}}\Big\{\frac{3\kappa(\theta)}{2\theta^{2}}
(\bar{\theta}_{y}+\widetilde{\theta}_{y})\bar{\theta}_{y}\widetilde{\theta}
-\frac{3\bar{\theta}}{2\theta^{2}}\kappa(\theta)\widetilde{\theta}_{y}\bar{\theta}_{y}\Big\}dy
\nonumber\\
&\leq C\epsilon^{1-a}\int_{\mathbb{R}}|
\frac{3\kappa(\theta)}{2\theta^{2}}(\widetilde{\theta}_{y}+\bar{\theta}_{y})
\bar{\theta}_{y}\widetilde{\theta}+(\frac{3\bar{\theta}\kappa(\theta)}{2\theta^{2}})_{y}\widetilde{\theta}\bar{\theta}_{y}+
\frac{3\bar{\theta}\kappa(\theta)}{2\theta^{2}}\widetilde{\theta}\bar{\theta}_{yy}|dy
\nonumber\\
&\leq C\epsilon^{1-a}\int_{\mathbb{R}}|\widetilde{\theta}|\{|\bar{\theta}_{y}||\widetilde{\theta}_{y}|
+|\bar{\theta}_{y}|^{2}+|\bar{\theta}_{y}||\theta_{y}|+|\bar{\theta}_{yy}|\}dy
\nonumber\\
&\leq C\epsilon^{1-a}\|\widetilde{\theta}\|_{L^{\infty}}\{\|\bar{\theta}_{yy}\|_{L^{1}}+\|\bar{\theta}_{y}\|^{2}
+\|\widetilde{\theta}_{y}\|^2\},
\end{align*}
which further implies that
\begin{align}
\label{3.16}
\int_{\mathbb{R}}H_{1}dy
&\leq \epsilon^{1-a}\big\{\eta\|\widetilde{\theta}_{y}\|^{2}
+C_{\eta}\|\widetilde{\theta}\|^{\frac{2}{3}}\|\bar{\theta}_{yy}\|^{\frac{4}{3}}_{L^{1}}
+C_{\eta}\|\widetilde{\theta}\|^{\frac{2}{3}}\|\bar{\theta}_{y}\|^{\frac{8}{3}}
+ C\|\widetilde{\theta}\|^{\frac{1}{2}}\|\widetilde{\theta}_{y}\|^{\frac{1}{2}}\|\widetilde{\theta}_{y}\|^2\big\}
\nonumber\\
&\leq \eta\epsilon^{1-a}\|\widetilde{\theta}_{y}\|^{2}
+C_{\eta}\epsilon^{1-a}\epsilon^{\frac{1}{3}(\frac{6}{5}-\frac{4}{5}a)}\epsilon^{\frac{4}{3}a}
(\delta+\epsilon^{a}\tau)^{-\frac{4}{3}}+C\sqrt{\mathcal{E}_{2}(\tau)}\mathcal{D}_{2}(\tau)
\nonumber\\
&\leq \eta\epsilon^{1-a}\|\widetilde{\theta}_{y}\|^{2}
+C_{\eta}\epsilon^{\frac{7}{5}+\frac{1}{15}a}(\delta+\epsilon^{a}\tau)^{-\frac{4}{3}}+
Ck^{\frac{1}{12}}\epsilon^{\frac{3}{5}-\frac{2}{5}a}\mathcal{D}_{2}(\tau).
\end{align}
Here we have used the following one-dimensional Sobolev imbedding theorem
\begin{align*}
\|g(y)\|_{L^{\infty}}\leq \sqrt{2}\|g(y)\|^{\frac{1}{2}}\|\partial_{y}g(y)\|^{\frac{1}{2}},
\quad \mbox{for}\quad g(y)\in H^{1}(\mathbb{R})\subset L^{\infty}(\mathbb{R}).
\end{align*}
Following the same method used as \eqref{3.16}, it holds that
\begin{align*}
\int_{\mathbb{R}} H_{2}dy&=\epsilon^{1-a}\int_{\mathbb{R}}\Big\{\frac{2\mu(\theta)}{\theta}(\bar{u}_{1y}+\widetilde{u}_{1y})
\bar{u}_{1y}\widetilde{\theta}-\frac{2\bar{\theta}}{\theta}\mu(\theta)\widetilde{u}_{1y}\bar{u}_{1y}\Big\}dy
\nonumber\\
&\leq\eta\epsilon^{1-a}\|(\widetilde{u}_{1y},\widetilde{\theta}_{y})\|^{2}
+C_{\eta}\epsilon^{\frac{7}{5}+\frac{1}{15}a}(\delta+\epsilon^{a}\tau)^{-\frac{4}{3}}+Ck^{\frac{1}{12}}\epsilon^{\frac{3}{5}-\frac{2}{5}a}
\mathcal{D}_{2}(\tau).
\end{align*}
By using the self-adjoint property of $L^{-1}_{M}$, \eqref{5.1} and \eqref{5.2}, one can show that
\begin{align}
\label{3.18}
\int_{\mathbb{R}^{3}}& (\frac{1}{2}v_{1}|v|^{2}-\sum^{3}_{i=1}v_{1}u_{i}v_{i})L^{-1}_{M}\Theta dv=
\int_{\mathbb{R}^{3}} L^{-1}_{M}\{P_{1}(\frac{1}{2}v_{1}|v|^{2}-\sum^{3}_{i=1}v_{1}u_{i}v_{i})M\}\frac{\Theta}{M} dv
\nonumber\\
=&\int_{\mathbb{R}^{3}} L^{-1}_{M}\{(R\theta)^{\frac{3}{2}}\hat{A}_{1}(\frac{v-u}{\sqrt{R\theta}})M\}\frac{\Theta}{M} dv
=(R\theta)^{\frac{3}{2}}\int_{\mathbb{R}^{3}}A_{1}(\frac{v-u}{\sqrt{R\theta}})\frac{\Theta}{M} dv,
\end{align}
and
\begin{align}
\label{3.19}
\int_{\mathbb{R}^{3}}v_{1}v_{i}L^{-1}_{M}\Theta dv=&
\int_{\mathbb{R}^{3}} L^{-1}_{M}\{P_{1}( v_{1}v_{i}M)\}\frac{\Theta}{M} dv
\nonumber\\
=&\int_{\mathbb{R}^{3}} L^{-1}_{M}\{R\theta\hat{B}_{1i}(\frac{v-u}{\sqrt{R\theta}})M\}\frac{\Theta}{M} dv
=R\theta\int_{\mathbb{R}^{3}}B_{1i}(\frac{v-u}{\sqrt{R\theta}})\frac{\Theta}{M} dv.
\end{align}
Both \eqref{3.18} and the expression of $H_{3}$ in \eqref{3.13} imply
\begin{align}
\label{3.20}
\int_{\mathbb{R}} H_{3}dy&=\int_{\mathbb{R}^{3}}\Big\{\frac{3}{2}(\frac{\widetilde{\theta}}{\theta})_{y}
\int_{\mathbb{R}^{3}} (\frac{1}{2}v_{1}|v|^{2}-\sum^{3}_{i=1}u_{i} v_{1}v_{i})L^{-1}_{M}\Theta dv\Big\}dy
\nonumber\\
&=\int_{\mathbb{R}^{3}}\Big\{\frac{3}{2}(\frac{\widetilde{\theta}}{\theta})_{y}
(R\theta)^{\frac{3}{2}}\int_{\mathbb{R}^{3}}A_{1}(\frac{v-u}{\sqrt{R\theta}})\frac{\Theta}{M} dv\Big\}dy.
\end{align}
Notice that for any multi-index $\beta$ and $m\geq 0$, we have by using the fast decay of the Burnett functions \eqref{5.4} and \eqref{1.26}
that
\begin{equation}
\label{3.21}
\int_{\mathbb{R}^{3}}\frac{|\langle v\rangle^{m}\sqrt{\mu}\partial_{\beta}A_{1}(\frac{v-u}{\sqrt{R\theta}})|^{2}}{M^{2}}dv\leq C.
\end{equation}
We now turn to compute the term with $\Theta$ in \eqref{3.20}. Recalling that
\begin{align}
\label{3.22}
\Theta=\epsilon^{1-a}\partial_{\tau}G+\epsilon^{1-a}P_{1}(v_{1}\partial_{y}G)-Q(G,G).
\end{align}
For the first term on the right-hand side of \eqref{3.22}. Recalling that $G=\overline{G}+\sqrt{\mu}f$, applying the similar arguments as \eqref{5.24},
and using $\sqrt{\mathcal{E}_{2}(\tau)}\leq k^{\frac{1}{12}}\epsilon^{\frac{3}{5}-\frac{2}{5}a}$ and $\epsilon^{a}\delta^{-1}\leq k^{\frac{1}{12}}\epsilon^{\frac{3}{5}-\frac{2}{5}a}$ due to \eqref{3.4} and \eqref{3.3}, one has from \eqref{3.21}, the Cauchy-Schwarz inequality and Lemma \ref{lem5.3} that
\begin{align}
\label{3.23}
&\int_{\mathbb{R}^{3}}\Big\{\frac{3}{2}(\frac{\widetilde{\theta}}{\theta})_{y}
(R\theta)^{\frac{3}{2}}\int_{\mathbb{R}^{3}}A_{1}(\frac{v-u}{\sqrt{R\theta}})\frac{\epsilon^{1-a}\partial_{\tau}\overline{G}}{M} dv\Big\}dy
\nonumber\\
&\leq C\epsilon^{1-a}(\|\widetilde{\theta}_{y}\|+\|\widetilde{\theta}\theta_{y}\|)
\times\big(\int_{\mathbb{R}}\int_{\mathbb{R}^{3}}|\frac{\partial_{\tau}\overline{G}}{\sqrt{\mu}}|^{2}dvdy\big)^{\frac{1}{2}}
\nonumber\\
&\leq C\epsilon^{2(1-a)}(\|\widetilde{\theta}_{y}\|+\|\widetilde{\theta}\theta_{y}\|)
\times (\|(\bar{u}_{1y\tau},\bar{\theta}_{y\tau})\|+\|(\bar{u}_{1y},\bar{\theta}_{y})\cdot(u_{\tau},\theta_{\tau})\|)
\nonumber\\
&\leq C\eta\epsilon^{1-a}\|\widetilde{\theta}_{y}\|^{2}
+C_{\eta}\epsilon^{\frac{7}{5}+\frac{1}{15}a}(\delta+\epsilon^{a}\tau)^{-\frac{4}{3}}
+C_{\eta}(\epsilon^{a}\delta^{-1}+\sqrt{\mathcal{E}_{2}(\tau)})\mathcal{D}_{2}(\tau)
\nonumber\\
&\leq C\eta\epsilon^{1-a}\|\widetilde{\theta}_{y}\|^{2}
+C_{\eta}\epsilon^{\frac{7}{5}+\frac{1}{15}a}(\delta+\epsilon^{a}\tau)^{-\frac{4}{3}}
+C_{\eta}k^{\frac{1}{12}}\epsilon^{\frac{3}{5}-\frac{2}{5}a}\mathcal{D}_{2}(\tau).
\end{align}
Similarly, it holds that
\begin{align}
\label{3.24}
&\int_{\mathbb{R}^{3}}\Big\{\frac{3}{2}(\frac{\widetilde{\theta}}{\theta})_{y}
(R\theta)^{\frac{3}{2}}\int_{\mathbb{R}^{3}}A_{1}(\frac{v-u}{\sqrt{R\theta}})\frac{\epsilon^{1-a}\sqrt{\mu}\partial_{\tau}f}{M} dv\Big\}dy
\leq C\epsilon^{(1-a)}\{\|\widetilde{\theta}_{y}\|+\|\widetilde{\theta}\theta_{y}\|\}
\times\|\langle v\rangle^{-\frac{1}{2}}\partial_{\tau}f\|
\nonumber\\
&\leq C\eta\epsilon^{1-a}\|\widetilde{\theta}_{y}\|^{2}+C_{\eta}\epsilon^{1-a}\|\partial_{\tau}f\|^{2}_{\sigma}
+C\epsilon^{\frac{7}{5}+\frac{1}{15}a}(\delta+\epsilon^{a}\tau)^{-\frac{4}{3}}
+Ck^{\frac{1}{12}}\epsilon^{\frac{3}{5}-\frac{2}{5}a}\mathcal{D}_{2}(\tau).
\end{align}
It follows from \eqref{3.23} and \eqref{3.24} that
\begin{align}
\label{3.25}
&\int_{\mathbb{R}^{3}}\Big\{\frac{3}{2}(\frac{\widetilde{\theta}}{\theta})_{y}
(R\theta)^{\frac{3}{2}}\int_{\mathbb{R}^{3}}A_{1}(\frac{v-u}{\sqrt{R\theta}})\frac{\epsilon^{1-a}\partial_{\tau}G}{M} dv\Big\}dy
\nonumber\\
&\leq C\eta\epsilon^{1-a}\|\widetilde{\theta}_{y}\|^{2}+C_{\eta}\epsilon^{1-a}\|\partial_{\tau}f\|^{2}_{\sigma}
+C_{\eta}\epsilon^{\frac{7}{5}+\frac{1}{15}a}(\delta+\epsilon^{a}\tau)^{-\frac{4}{3}}
+C_{\eta}k^{\frac{1}{12}}\epsilon^{\frac{3}{5}-\frac{2}{5}a}\mathcal{D}_{2}(\tau).
\end{align}
For the second term on the right-hand side of \eqref{3.22}. Similar arguments as \eqref{3.25} imply
\begin{align*}
&\int_{\mathbb{R}^{3}}\Big\{\frac{3}{2}(\frac{\widetilde{\theta}}{\theta})_{y}
(R\theta)^{\frac{3}{2}}\int_{\mathbb{R}^{3}}A_{1}(\frac{v-u}{\sqrt{R\theta}})\frac{\epsilon^{1-a}P_{1}(v_{1}\partial_{y}G)}{M} dv\Big\}dy
\nonumber\\
&\leq C\eta\epsilon^{1-a}\|\widetilde{\theta}_{y}\|^{2}+C_{\eta}\epsilon^{1-a}\|\partial_{y}f\|^{2}_{\sigma}
+C_{\eta}\epsilon^{\frac{7}{5}+\frac{1}{15}a}(\delta+\epsilon^{a}\tau)^{-\frac{4}{3}}
+C_{\eta}k^{\frac{1}{12}}\epsilon^{\frac{3}{5}-\frac{2}{5}a}\mathcal{D}_{2}(\tau).
\end{align*}
For the last term of \eqref{3.22}, by using \eqref{2.8}, \eqref{3.21} and the similar arguments as \eqref{5.20}, we get
\begin{align}
\label{3.27}
&\int_{\mathbb{R}^{3}}\Big\{\frac{3}{2}(\frac{\widetilde{\theta}}{\theta})_{y}
(R\theta)^{\frac{3}{2}}\int_{\mathbb{R}^{3}}A_{1}(\frac{v-u}{\sqrt{R\theta}})\frac{Q(G,G)}{M} dv\Big\}dy
\nonumber\\
&=\int_{\mathbb{R}^{3}}\Big\{\frac{3}{2}(\frac{\widetilde{\theta}}{\theta})_{y}
(R\theta)^{\frac{3}{2}}\int_{\mathbb{R}^{3}}\frac{\sqrt{\mu}A_{1}(\frac{v-u}{\sqrt{R\theta}})}{M}\Gamma(\frac{G}{\sqrt{\mu}},\frac{G}{\sqrt{\mu}}) dv\Big\}dy
\nonumber\\
&\leq C\eta\epsilon^{1-a}\|\widetilde{\theta}_{y}\|^{2}
+C_{\eta}\epsilon^{\frac{7}{5}+\frac{1}{15}a}(\delta+\epsilon^{a}\tau)^{-\frac{4}{3}}
+C_{\eta}k^{\frac{1}{12}}\epsilon^{\frac{3}{5}-\frac{2}{5}a}\mathcal{D}_{2}(\tau).
\end{align}
By the estimates from \eqref{3.25} to \eqref{3.27}, we have from \eqref{3.20} that
\begin{align}
\label{3.28}
\int_{\mathbb{R}} H_{3}dy
\leq C\eta\epsilon^{1-a}\|\widetilde{\theta}_{y}\|^{2}+C_{\eta}\epsilon^{1-a}\sum_{|\alpha|=1}\|\partial^{\alpha}f\|^{2}_{\sigma}
+C_{\eta}\epsilon^{\frac{7}{5}+\frac{1}{15}a}(\delta+\epsilon^{a}\tau)^{-\frac{4}{3}}
+C_{\eta}k^{\frac{1}{12}}\epsilon^{\frac{3}{5}-\frac{2}{5}a}\mathcal{D}_{2}(\tau).
\end{align}
Following the same strategies used in the estimates of \eqref{3.18}, then similar arguments as \eqref{3.28} imply
\begin{align*}
\int_{\mathbb{R}} H_{4}dy&=\int_{\mathbb{R}}\Big\{(\frac{3}{2}\sum^{3}_{i=1}\widetilde{u}_{iy}-\frac{3}{2}\frac{\widetilde{\theta}}{\theta}\sum^{3}_{i=1}u_{iy})
\int_{\mathbb{R}^{3}} v_{1}v_{i}L^{-1}_{M}\Theta dv\Big\}dy
\nonumber\\
&=\int_{\mathbb{R}}\Big\{(\frac{3}{2}\sum^{3}_{i=1}\widetilde{u}_{iy}-\frac{3}{2}\frac{\widetilde{\theta}}{\theta}\sum^{3}_{i=1}u_{iy})
R\theta\int_{\mathbb{R}^{3}}B_{1i}(\frac{v-u}{\sqrt{R\theta}})\frac{\Theta}{M} dv\Big\}dy
\nonumber\\
\leq C\eta\epsilon^{1-a}&\|\widetilde{u}_{y}\|^{2}+C_{\eta}\epsilon^{1-a}\sum_{|\alpha|=1}\|\partial^{\alpha}f\|^{2}_{\sigma}
+C_{\eta}\epsilon^{\frac{7}{5}+\frac{1}{15}a}(\delta+\epsilon^{a}\tau)^{-\frac{4}{3}}
+C_{\eta}k^{\frac{1}{12}}\epsilon^{\frac{3}{5}-\frac{2}{5}a}\mathcal{D}_{2}(\tau).
\end{align*}
Here we have used the fact that for any multi-index $\beta$ and $m\geq 0$,
\begin{equation*}
\int_{\mathbb{R}^{3}}\frac{|\langle v\rangle^{m}\sqrt{\mu}\partial_{\beta}B_{1i}(\frac{v-u}{\sqrt{R\theta}})|^{2}}{M^{2}}dv\leq C.
\end{equation*}
Therefore, substituting the estimates of $H_{1}-H_{4}$ into \eqref{3.15} and taking $\eta>0$ small enough gives
\begin{align}
\label{3.31}
\frac{d}{d\tau}&\int_{\mathbb{R}}\eta(\tau,y) dy+c\epsilon^{1-a}\|(\widetilde{u}_{y},\widetilde{\theta}_{y})\|^{2}
+c_{2}\|\sqrt{\bar{u}_{1y}}(\widetilde{\rho},\widetilde{u}_{1},\widetilde{\theta})\|^{2}
\nonumber\\
&\leq C\epsilon^{1-a}\sum_{|\alpha|=1}\|\partial^{\alpha}f\|^{2}_{\sigma}
+C\epsilon^{\frac{7}{5}+\frac{1}{15}a}(\delta+\epsilon^{a}\tau)^{-\frac{4}{3}}
+Ck^{\frac{1}{12}}\epsilon^{\frac{3}{5}-\frac{2}{5}a}\mathcal{D}_{2}(\tau).
\end{align}
\par
Since there is no dissipation for density function and the temporal derivatives for $(\widetilde{\rho},\widetilde{u},\widetilde{\theta})$ in \eqref{3.31}, to get the estimation of $\|\widetilde{\rho}_{y}\|^{2}$ and $\|(\widetilde{\rho}_{\tau},\widetilde{u}_{\tau},\widetilde{\theta}_{\tau})\|^{2}$,
we first take the inner product of  \eqref{2.4}$_2$  with $\widetilde{\rho}_{y}$ over $\mathbb{R}$ to get
\begin{align}
\label{3.32}
\epsilon^{1-a}(\frac{2\bar{\theta}}{3\bar{\rho}}\widetilde{\rho}_{y},\widetilde{\rho}_{y})=\epsilon^{1-a}(-\widetilde{u}_{1\tau}-\bar{u}_{1}\widetilde{u}_{1y}
-\frac{2}{3}\widetilde{\theta}_{y}-J_{2}-\frac{1}{\rho}\int_{\mathbb{R}^3} v^{2}_{1}G_{y}dv,\widetilde{\rho}_{y}).
\end{align}
By using \eqref{2.4}$_1$, the integration by parts, the Cauchy inequality and Lemma \ref{lem5.3}, one has
\begin{align*}
-\epsilon^{1-a}(\widetilde{u}_{1\tau},\widetilde{\rho}_{y})=&
-\epsilon^{1-a}(\widetilde{u}_{1},\widetilde{\rho}_{y})_{\tau}
-\epsilon^{1-a}(\widetilde{u}_{1y},\widetilde{\rho}_{\tau})
\nonumber\\
=&-\epsilon^{1-a}(\widetilde{u}_{1},\widetilde{\rho}_{y})_{\tau}
-\epsilon^{1-a}(\widetilde{u}_{1y},\{-\bar{\rho} \widetilde{u}_{1y}-\bar{\rho}_{y}\widetilde{u}_{1}-
(\widetilde{\rho}\widetilde{u}_{1})_{y}-\bar{u}_{1}\widetilde{\rho}_{y}-\bar{u}_{1y}\widetilde{\rho}\})
\nonumber\\
\leq& -\epsilon^{1-a}(\widetilde{u}_{1},\widetilde{\rho}_{y})_{\tau}+
C\eta\epsilon^{1-a}\|\widetilde{\rho}_{y}\|^{2}+C_{\eta}\epsilon^{1-a}\|\widetilde{u}_{1y}\|^{2}
\nonumber\\
&+C\epsilon^{\frac{7}{5}+\frac{1}{15}a}(\delta+\epsilon^{a}\tau)^{-\frac{4}{3}}
+C\sqrt{\mathcal{E}_{2}(\tau)}\mathcal{D}_{2}(\tau),
\end{align*}
where in the last inequality, we have dealt with the typical terms as follows
\begin{align*}
\epsilon^{1-a}|(\widetilde{u}_{1y},\bar{u}_{1y}\widetilde{\rho})|&\leq C\epsilon^{1-a}\|\widetilde{u}_{1y}\|^{2}
+C\epsilon^{1-a}\|\bar{u}_{1y}\|^{2}_{L^{\infty}}\|\widetilde{\rho}\|^{2}
\\
&\leq C\epsilon^{1-a}\|\widetilde{u}_{1y}\|^{2}
+C\epsilon^{1-a}\epsilon^{2a}(\delta+\epsilon^{a}\tau)^{-2}\epsilon^{\frac{6}{5}-\frac{4}{5}a}
\\
&\leq C\epsilon^{1-a}\|\widetilde{u}_{1y}\|^{2}
+C\epsilon^{\frac{7}{5}+\frac{1}{15}a}(\delta+\epsilon^{a}\tau)^{-\frac{4}{3}}.
\end{align*}
The Cauchy inequality implies
\begin{align*}
\epsilon^{1-a}|(\bar{u}_{1}\widetilde{u}_{1y}+\frac{2}{3}\widetilde{\theta}_{y},\widetilde{\rho}_{y})|
\leq\eta\epsilon^{1-a}\|\widetilde{\rho}_{y}\|^{2}
+C_{\eta}\epsilon^{1-a}\|(\widetilde{u}_{1y},\widetilde{\theta}_{y})\|^{2}.
\end{align*}
Notice that the term involving $J_{2}$ can be controlled by
\begin{align*}
\epsilon^{1-a}|(J_{2},\widetilde{\rho}_{y})|
&=\epsilon^{1-a}|(\{\widetilde{u}_{1}\bar{u}_{1y}+\widetilde{u}_{1}
\widetilde{u}_{1y}+\frac{2}{3}\rho_{y}\frac{\bar{\rho}\widetilde{\theta}-\widetilde{\rho}\bar{\theta}}
{\rho\bar{\rho}}\},\widetilde{\rho}_{y})|
\\
&\leq \eta\epsilon^{1-a}\|\widetilde{\rho}_{y}\|^{2}
+C_{\eta}\epsilon^{\frac{7}{5}+\frac{1}{15}a}(\delta+\epsilon^{a}\tau)^{-\frac{4}{3}}
+C_{\eta}\sqrt{\mathcal{E}_{2}(\tau)}\mathcal{D}_{2}(\tau),
\end{align*}
according to the Cauchy inequality, Lemma \ref{lem5.3}, \eqref{3.4} and \eqref{3.3}.
Recall $G=\overline{G}+\sqrt{\mu}f$, we have from the H\"older inequality, Sobolev imbedding theorem, \eqref{5.26} and \eqref{2.17} that
\begin{align*}
&\epsilon^{1-a}\Big|\int_{\mathbb{R}}\int_{\mathbb{R}^3} v^{2}_{1}\frac{G_{y}}{\rho}\widetilde{\rho}_{y}dvdy\Big|
\leq \epsilon^{1-a}\|\widetilde{\rho}_{y}\|(\|\langle v\rangle^{-\frac{1}{2}}f_{y}\|+\|\frac{\overline{G}_{y}}{\sqrt{\mu}}\|)
\\
&\leq\eta\epsilon^{1-a}\|\widetilde{\rho}_{y}\|^{2}
+C_{\eta}\epsilon^{1-a}\|f_{y}\|_{\sigma}^{2}+C_{\eta}\epsilon^{\frac{7}{5}+\frac{1}{15}a}(\delta+\epsilon^{a}\tau)^{-\frac{4}{3}}
+C_{\eta}\sqrt{\mathcal{E}_{2}(\tau)}\mathcal{D}_{2}(\tau).
\end{align*}
Hence, plugging the above related estimates into \eqref{3.32} and using $\sqrt{\mathcal{E}_{2}(\tau)}\leq k^{\frac{1}{12}}\epsilon^{\frac{3}{5}-\frac{2}{5}a}$,
we have by choosing $\eta>0$ small enough that
\begin{align}
\label{3.33}
\epsilon^{1-a}\|\widetilde{\rho}_{y}\|^{2}\leq&
 -C\epsilon^{1-a}(\widetilde{u}_{1},\widetilde{\rho}_{y})_{\tau}+
C\epsilon^{1-a}\|(\widetilde{u}_{1y},\widetilde{\theta}_{y})\|^{2}
+C\epsilon^{1-a}\|f_{y}\|_{\sigma}^{2}
\nonumber\\
&+C\epsilon^{\frac{7}{5}+\frac{1}{15}a}(\delta+\epsilon^{a}\tau)^{-\frac{4}{3}}
+Ck^{\frac{1}{12}}\epsilon^{\frac{3}{5}-\frac{2}{5}a}\mathcal{D}_{2}(\tau).
\end{align}
On the other hand, by using the system \eqref{2.4} again, we can arrive at
\begin{align}
&\epsilon^{1-a}\|(\widetilde{\rho}_{\tau},\widetilde{u}_{\tau},\widetilde{\theta}_{\tau})\|^{2}\notag\\
&\leq C\epsilon^{1-a}\{\|(\widetilde{\rho}_{y},\widetilde{u}_{y},\widetilde{\theta}_{y})\|^{2}+\|f_{y}\|_{\sigma}^{2}\}
+C\epsilon^{\frac{7}{5}+\frac{1}{15}a}(\delta+\epsilon^{a}\tau)^{-\frac{4}{3}}
+Ck^{\frac{1}{12}}\epsilon^{\frac{3}{5}-\frac{2}{5}a}\mathcal{D}_{2}(\tau).\label{3.34}
\end{align}
For some suitably large constant $\widetilde{C}_{0}>0$, a suitable linear combination of \eqref{3.34}, \eqref{3.33} and \eqref{3.31} yields
\begin{align}
\label{3.35}
&\frac{d}{d\tau}\Big(\widetilde{C}_{0}\int_{\mathbb{R}} \eta(\tau,y) dy+C\int_{\mathbb{R}}\epsilon^{1-a}\widetilde{u}_{1}\widetilde{\rho}_{y} dy\Big)
+c\|\sqrt{\bar{u}_{1y}}(\widetilde{\rho},\widetilde{u}_{1},\widetilde{\theta})\|^{2}
+c\epsilon^{1-a}\sum_{|\alpha|=1}\|\partial^{\alpha}(\widetilde{\rho},\widetilde{u},\widetilde{\theta})\|^{2}
\nonumber\\
&\hspace{1cm}\leq C\epsilon^{1-a}\sum_{|\alpha|=1}\|\partial^{\alpha}f\|^{2}_{\sigma}
+C\epsilon^{\frac{7}{5}+\frac{1}{15}a}(\delta+\epsilon^{a}\tau)^{-\frac{4}{3}}
+Ck^{\frac{1}{12}}\epsilon^{\frac{3}{5}-\frac{2}{5}a}\mathcal{D}_{2}(\tau).
\end{align}
This completes the proof of lower order energy estimates for the macroscopic component $(\widetilde{\rho},\widetilde{u},\widetilde{\theta})$.
\par
Next, we turn to prove lower order energy estimates for the microscopic component $f$.
Taking the inner product of \eqref{2.11} with $f$  over $\mathbb{R}\times\mathbb{R}^{3}$ gives
\begin{align}
\label{3.36}
(&\partial_{\tau}f+v_{1}\partial_{y}f
-\epsilon^{a-1}\mathcal{L}f,f)=\epsilon^{a-1}(\Gamma(f,\frac{M-\mu}{\sqrt{\mu}})+
\Gamma(\frac{M-\mu}{\sqrt{\mu}},f)
+\Gamma(\frac{G}{\sqrt{\mu}},\frac{G}{\sqrt{\mu}}),f)
\nonumber\\
&+(\frac{P_{0}(v_{1}\sqrt{\mu}\partial_{y}f)}{\sqrt{\mu}}-\frac{1}{\sqrt{\mu}}P_{1}v_{1}M\big\{\frac{|v-u|^{2}
\widetilde{\theta}_{y}}{2R\theta^{2}}+\frac{(v-u)\cdot\widetilde{u}_{y}}{R\theta}\big\}
-\frac{P_{1}(v_{1}\partial_{y}\overline{G})}{\sqrt{\mu}}
-\frac{\partial_{\tau}\overline{G}}{\sqrt{\mu}},f).
\end{align}
We will estimate each term for \eqref{3.36}. First of all, we have from the integration by parts and \eqref{5.5} that
$$
(\partial_{\tau}f+v_{1}\partial_{y}f
-\epsilon^{a-1}\mathcal{L}f,f)\geq\frac{1}{2}\frac{d}{d\tau}\|f\|^{2}+\sigma_{1}\epsilon^{a-1}\|f\|_{\sigma}^{2}.
$$
By using \eqref{5.7} and \eqref{5.12}, we can obtain
$$
\epsilon^{a-1}|(\Gamma(f,\frac{M-\mu}{\sqrt{\mu}}),f)+
(\Gamma(\frac{M-\mu}{\sqrt{\mu}},f),f)|
\leq  C\eta_{0}\epsilon^{a-1}\|f\|^{2}_{\sigma}.
$$
From \eqref{5.20}, one can see easily that
$$
\epsilon^{a-1}|(\Gamma(\frac{G}{\sqrt{\mu}},\frac{G}{\sqrt{\mu}}),f)|
\leq C\eta\epsilon^{a-1}\|f\|^{2}_{\sigma}
+C_{\eta}\epsilon^{\frac{7}{5}+\frac{1}{15}a}(\delta+\epsilon^{a}\tau)^{-\frac{4}{3}}
+C_{\eta}k^{\frac{1}{12}}\epsilon^{\frac{3}{5}-\frac{2}{5}a}\mathcal{D}_{2}(\tau).
$$
In view of the properties of $P_{0}$ in \eqref{1.9} as well as \eqref{2.17}, one can show that
\begin{align*}
|(\frac{P_{0}(v_{1}\sqrt{\mu}\partial_{y}f)}{\sqrt{\mu}},f)|
&\leq C\|\langle v\rangle^{-\frac{1}{2}}f\|\|\langle v\rangle^{\frac{1}{2}}\frac{P_{0}(v_{1}\sqrt{\mu}\partial_{y}f)}{\sqrt{\mu}}\|
\\
&\leq C\|\langle v\rangle^{-\frac{1}{2}}f\|\|\langle v\rangle^{-\frac{1}{2}}\partial_{y}f\|\leq C\eta\epsilon^{a-1}\|f\|_{\sigma}^{2}+C_{\eta}\epsilon^{1-a}\|f_{y}\|_{\sigma}^{2}.
\end{align*}
By using \eqref{5.1}, a direct computation shows that
\begin{align*}
P_{1}v_{1}M\big\{\frac{|v-u|^{2}
\widetilde{\theta}_{y}}{2R\theta^{2}}+\frac{(v-u)\cdot\widetilde{u}_{y}}{R\theta}\big\}
=\frac{\sqrt{R}\widetilde{\theta}_{y}}{\sqrt{\theta}}\hat{A}_{1}(\frac{v-u}{\sqrt{R\theta}})M
+\sum_{j=1}^{3}\widetilde{u}_{jy}\hat{B}_{1j}(\frac{v-u}{\sqrt{R\theta}})M,
\end{align*}
which implies that
$$
|(\frac{1}{\sqrt{\mu}}P_{1}v_{1}M\big\{\frac{|v-u|^{2}
\widetilde{\theta}_{y}}{2R\theta^{2}}+\frac{(v-u)\cdot\widetilde{u}_{y}}{R\theta}\big\},f)|
\leq C\eta\epsilon^{a-1}\|f\|_{\sigma}^{2}+C_{\eta}\epsilon^{1-a}\|(\widetilde{u}_{y},\widetilde{\theta}_{y})\|^{2}.
$$
In addition, we use \eqref{1.9}, \eqref{5.26}, \eqref{3.3}, \eqref{3.4}, \eqref{2.17}, the Sobolev imbedding theorem and Lemma \ref{lem5.3} to obtain
\begin{align*}
&|\big(\frac{P_{1}(v_{1}\partial_{y}\overline{G})}{\sqrt{\mu}}
+\frac{\partial_{\tau}\overline{G}}{\sqrt{\mu}},f\big)|
=|\big(\frac{v_{1}\partial_{y}\overline{G}}{\sqrt{\mu}}-
\frac{P_{0}(v_{1}\partial_{y}\overline{G})}{\sqrt{\mu}}
+\frac{\partial_{\tau}\overline{G}}{\sqrt{\mu}},f\big)|
\nonumber\\
&\leq C\epsilon^{1-a}\{\|(\bar{u}_{1yy},\bar{\theta}_{yy})\|
+\|(\bar{u}_{1y},\bar{\theta}_{y})\cdot(u_{y},\theta_{y})\|
+\|(\bar{u}_{1y\tau},\bar{\theta}_{y\tau})\|+
\|(\bar{u}_{1y},\bar{\theta}_{y})\cdot(u_{\tau},\theta_{\tau})\|\}\|\langle v\rangle^{-\frac{1}{2}}f\|
\nonumber\\
&\leq C\eta\epsilon^{a-1}\|f\|_{\sigma}^{2}
+C_{\eta}\epsilon^{\frac{7}{5}+\frac{1}{15}a}(\delta+\epsilon^{a}\tau)^{-\frac{4}{3}}
+C_{\eta}k^{\frac{1}{12}}\epsilon^{\frac{3}{5}-\frac{2}{5}a}\mathcal{D}_{2}(\tau).
\end{align*}
Plugging the above related estimates into \eqref{3.36}, we get
\begin{align}
\label{3.39}
\frac{1}{2}\frac{d}{d\tau}\|f\|^{2}+c\epsilon^{a-1}\|f\|_{\sigma}^{2}
&\leq C\epsilon^{1-a}\{\|(\widetilde{u}_{y},\widetilde{\theta}_{y})\|^{2}+\|f_{y}\|_{\sigma}^{2}\}
\nonumber\\
&\quad+C\epsilon^{\frac{7}{5}+\frac{1}{15}a}(\delta+\epsilon^{a}\tau)^{-\frac{4}{3}}
+Ck^{\frac{1}{12}}\epsilon^{\frac{3}{5}-\frac{2}{5}a}\mathcal{D}_{2}(\tau),
\end{align}
by choosing suitably small $\eta$ and using the smallness of $\eta_{0}$.
\par
In summary, for some suitably large constant $\widetilde{C}_1>0$, adding $\eqref{3.35}\times\widetilde{C}_1$
to \eqref{3.39} gives
\begin{align}
\label{3.40}
&\hspace{0.5cm}\frac{d}{d\tau}\Big\{\widetilde{C}_{1}\big(\widetilde{C}_{0}\int_{\mathbb{R}} \eta(\tau,y) dy+C\int_{\mathbb{R}}\epsilon^{1-a}\widetilde{u}_{1}\widetilde{\rho}_{y}dy\big)+\frac{1}{2}\|f\|^{2}\Big\}
\nonumber\\
&\hspace{1cm}+c\|\sqrt{\bar{u}_{1y}}(\widetilde{\rho},\widetilde{u}_{1},\widetilde{\theta})\|^{2}
+c\epsilon^{1-a}\sum_{|\alpha|=1}\|\partial^{\alpha}(\widetilde{\rho},\widetilde{u},\widetilde{\theta})\|^{2}
+c\epsilon^{a-1}\|f\|_{\sigma}^{2}
\nonumber\\
&\hspace{0.5cm}\leq C\epsilon^{1-a}\sum_{|\alpha|=1}\|\partial^{\alpha}f\|^{2}_{\sigma}
+C\epsilon^{\frac{7}{5}+\frac{1}{15}a}(\delta+\epsilon^{a}\tau)^{-\frac{4}{3}}
+Ck^{\frac{1}{12}}\epsilon^{\frac{3}{5}-\frac{2}{5}a}\mathcal{D}_{2}(\tau).
\end{align}
Integrating \eqref{3.40} with respect to $\tau$, we have
\begin{align}
&\|(\widetilde{\rho},\widetilde{u},\widetilde{\theta})\|^{2}+\|f\|^{2}
+\int^{\tau}_{0}\|\sqrt{\bar{u}_{1y}}(\widetilde{\rho},\widetilde{u}_{1},\widetilde{\theta})\|^{2}ds\notag\\
&\quad+\epsilon^{1-a}\sum_{|\alpha|=1}\int^{\tau}_{0}\|\partial^{\alpha}(\widetilde{\rho},\widetilde{u},\widetilde{\theta})\|^{2}ds
+\epsilon^{a-1}\int^{\tau}_{0}\|f\|_{\sigma}^{2}ds
\nonumber\\
&\leq C\epsilon^{2(1-a)}\|\widetilde{\rho}_{y}\|^{2}+ C\epsilon^{1-a}\sum_{|\alpha|=1}\int^{\tau}_{0}\|\partial^{\alpha}f\|^{2}_{\sigma}ds
+Ck^{\frac{1}{3}}\epsilon^{\frac{6}{5}-\frac{4}{5}a}
+Ck^{\frac{1}{12}}\epsilon^{\frac{3}{5}-\frac{2}{5}a}\int^{\tau}_{0}\mathcal{D}_{2}(s)ds,\label{3.41}
\end{align}
by using \eqref{3.2} and \eqref{3.3} as well as the fact $\eta(\tau,y)\approx\|(\widetilde{\rho},\widetilde{u},\widetilde{\theta})\|^{2}$ due to \eqref{3.11}.
This completes the proof of the lower order energy estimates.
\subsection{High order energy estimates}\label{sec3.2}
In this subsection, we will derive high order energy estimates on time-spatial derivatives.
We first consider the fluid variables $(\widetilde{\rho},\widetilde{u},\widetilde{\theta})$.
Differentiating  \eqref{2.6}$_1$ with respect to $y$, we then multiply the resulting equation by $\frac{2\bar{\theta}}{3\bar{\rho}^{2}}\widetilde{\rho}_{y}$ and integrate with respect to $y$
to obtain
\begin{align}
\label{3.42}
(\widetilde{\rho}_{\tau y}+\bar{\rho} \widetilde{u}_{1yy}+2\bar{\rho}_{y}\widetilde{u}_{1y}
+\bar{\rho}_{yy}\widetilde{u}_{1},\frac{2\bar{\theta}}{3\bar{\rho}^{2}}\widetilde{\rho}_{y})
=-(J_{1y},\frac{2\bar{\theta}}{3\bar{\rho}^{2}}\widetilde{\rho}_{y}).
\end{align}
This together with the integration by parts lead to
\begin{align}
\label{3.43}
\frac{1}{2}\frac{d}{d\tau}\|(\frac{2\bar{\theta}}{3\bar{\rho}^{2}})^{1/2}\widetilde{\rho}_{y}\|^{2}
+( \widetilde{u}_{1yy},\frac{2\bar{\theta}}{3\bar{\rho}}\widetilde{\rho}_{y})
=(\frac{\widetilde{\rho}_{y}^{2}}{2},
(\frac{2\bar{\theta}}{3\bar{\rho}^{2}})_{\tau})-(2\bar{\rho}_{y}\widetilde{u}_{1y}
+\bar{\rho}_{yy}\widetilde{u}_{1}+J_{1y},\frac{2\bar{\theta}}{3\bar{\rho}^{2}}\widetilde{\rho}_{y}).
\end{align}
We are going to estimate the terms on the right-hand side of \eqref{3.43}.
By the Sobolev imbedding theorem and Lemma \ref{lem5.3}, we get
\begin{align}
\label{3.44}
|(\frac{\widetilde{\rho}_{y}^{2}}{2},
(\frac{2\bar{\theta}}{3\bar{\rho}^{2}})_{\tau})|
+&|(2\bar{\rho}_{y}\widetilde{u}_{1y},\frac{2\bar{\theta}}{3\bar{\rho}^{2}}\widetilde{\rho}_{y})|
\leq C\|(\bar{\rho}_{\tau},\bar{\theta}_{\tau})\|_{L^{\infty}}\|\widetilde{\rho}_{y}\|^{2}
+C\|\bar{\rho}_{y}\|_{L^{\infty}}\|\widetilde{u}_{1y}\|\|\widetilde{\rho}_{y}\|
\nonumber\\
&\leq C\epsilon^{a}\delta^{-1}\|\widetilde{\rho}_{y}\|^{2}+ C\epsilon^{a}\delta^{-1}\|\widetilde{u}_{1y}\|\|\widetilde{\rho}_{y}\|
\nonumber\\
&= C\epsilon^{a}\delta^{-1}\epsilon^{a-1}
\epsilon^{1-a}\|\widetilde{\rho}_{y}\|^{2}+ C\epsilon^{a}\delta^{-1}\epsilon^{a-1}
\epsilon^{1-a}\|\widetilde{u}_{1y}\|\|\widetilde{\rho}_{y}\|
\nonumber\\
&\leq C\epsilon^{2a-1}\delta^{-1}\mathcal{D}_{2}(\tau)
=C\epsilon^{2a-1}k\epsilon^{-(\frac{3}{5}-\frac{2}{5}a)}\mathcal{D}_{2}(\tau)
= Ck\epsilon^{\frac{12}{5}a-\frac{8}{5}}\mathcal{D}_{2}(\tau),
\end{align}
where in the last line, we have used \eqref{3.3} and \eqref{2.19}. It also holds by using the Sobolev imbedding theorem, Lemma \ref{lem5.3},
\eqref{2.19}, \eqref{3.4} and \eqref{3.3} that
\begin{align}
\label{3.45}
|(\bar{\rho}_{yy}\widetilde{u}_{1},\frac{2\bar{\theta}}{3\bar{\rho}^{2}}\widetilde{\rho}_{y})|
&\leq C\|\widetilde{u}_{1}\|_{L^{\infty}}\|\bar{\rho}_{yy}\|\|\widetilde{\rho}_{y}\|
\leq C\|\widetilde{u}_{1}\|^{\frac{1}{2}}\|\widetilde{u}_{1y}\|^{\frac{1}{2}}(\|\bar{\rho}_{yy}\|^{2}+\|\widetilde{\rho}_{y}\|^{2})
\nonumber\\
&\leq Ck^{\frac{1}{12}}\epsilon^{\frac{3}{5}-\frac{2}{5}a}\big\{\epsilon^{3a}\delta^{-1}(\delta+\epsilon^{a}\tau)^{-2}
+\epsilon^{a-1}\mathcal{D}_{2}(\tau)\big\}
\nonumber\\
&\leq Ck^{\frac{1}{12}}\epsilon^{\frac{7}{5}+\frac{1}{15}a}(\delta+\epsilon^{a}\tau)^{-\frac{4}{3}}
+Ck^{\frac{1}{12}}\epsilon^{\frac{3}{5}a-\frac{2}{5}}\mathcal{D}_{2}(\tau).
\end{align}
Now, we estimate the term with $ J_{1y}$. Recalling $J_{1}=\widetilde{\rho}_{y}\bar{u}_{1}+\bar{u}_{1y}\widetilde{\rho}
+(\widetilde{\rho}\,\widetilde{u}_{1})_{y}$, performing calculations similar to \eqref{3.44} and \eqref{3.45}, we can arrive at
\begin{align*}
&|((\widetilde{\rho}_{y}\bar{u}_{1})_{y}+(\bar{u}_{1y}\widetilde{\rho})_y,
\frac{2\bar{\theta}}{3\bar{\rho}^{2}}\widetilde{\rho}_{y})|
\leq|((\frac{2\bar{\theta}\bar{u}_1}{3\bar{\rho}^{2}})_y,\frac{\widetilde{\rho}_{y}^2}{2})|
+|(\bar{u}_{1y}\widetilde{\rho}_{y}+\bar{u}_{1yy}\widetilde{\rho}+\bar{u}_{1y}\widetilde{\rho}_y,
\frac{2\bar{\theta}}{3\bar{\rho}^{2}}\widetilde{\rho}_{y})|
\nonumber\\
&\hspace{2cm}\leq C\epsilon^{\frac{7}{5}+\frac{1}{15}a}(\delta+\epsilon^{a}\tau)^{-\frac{4}{3}}
+ C(k\epsilon^{\frac{12}{5}a-\frac{8}{5}}+k^{\frac{1}{12}}\epsilon^{\frac{3}{5}a-\frac{2}{5}})\mathcal{D}_{2}(\tau).
\end{align*}
On the other hand, we have from the Sobolev imbedding theorem and \eqref{3.4} that
\begin{align*}
|((\widetilde{\rho}\widetilde{u}_{1})_{yy},\frac{2\bar{\theta}}{3\bar{\rho}^{2}}\widetilde{\rho}_{y})|
&=|(\widetilde{\rho}_{yy}\widetilde{u}_{1}+\widetilde{\rho}\widetilde{u}_{1yy}
+2\widetilde{\rho}_{y}\widetilde{u}_{1y},\frac{2\bar{\theta}}{3\bar{\rho}^{2}}\widetilde{\rho}_{y})|
\nonumber\\
&\leq C\epsilon^{a-1}\sqrt{\mathcal{E}_{2}(\tau)}\mathcal{D}_{2}(\tau)
\leq Ck^{\frac{1}{12}}\epsilon^{\frac{3}{5}a-\frac{2}{5}}\mathcal{D}_{2}(\tau).
\end{align*}
With the help of the above two estimates, we get
\begin{equation}
\label{3.46}
|(J_{1y},\frac{2\bar{\theta}}{3\bar{\rho}^{2}}\widetilde{\rho}_{y})|
\leq C\epsilon^{\frac{7}{5}+\frac{1}{15}a}(\delta+\epsilon^{a}\tau)^{-\frac{4}{3}}+ C(k\epsilon^{\frac{12}{5}a-\frac{8}{5}}+k^{\frac{1}{12}}\epsilon^{\frac{3}{5}a-\frac{2}{5}})\mathcal{D}_{2}(\tau).
\end{equation}
Hence, substituting the estimates \eqref{3.44}-\eqref{3.46} into \eqref{3.43}, we can obtain
\begin{equation}
\label{3.48}
\frac{1}{2}\frac{d}{d\tau}\|(\frac{2\bar{\theta}}{3\bar{\rho}^{2}})^{1/2}\widetilde{\rho}_{y}\|^{2}
+( \widetilde{u}_{1yy},\frac{2\bar{\theta}}{3\bar{\rho}}\widetilde{\rho}_{y})
\leq C\epsilon^{\frac{7}{5}+\frac{1}{15}a}(\delta+\epsilon^{a}\tau)^{-\frac{4}{3}}
+ Ck^{\frac{1}{12}}\epsilon^{\frac{3}{5}a-\frac{2}{5}}\mathcal{D}_{2}(\tau),
\end{equation}
due to the fact that
\begin{equation}
\label{3.47}
k\epsilon^{\frac{12}{5}a-\frac{8}{5}}\leq k^{\frac{1}{12}}\epsilon^{\frac{3}{5}a-\frac{2}{5}},
\end{equation}
by the assumption of $\frac{2}{3}\leq a\leq1$ as well as the smallness of $k$.

Similar for deducing \eqref{3.42}, by differentiating the equation \eqref{2.6}$_{2}$ with respect to $y$, we then take the inner product of
the resulting equation with $\widetilde{u}_{1y}$ to obtain
\begin{align}
\label{3.49}
&\frac{1}{2}\frac{d}{d\tau}\|\widetilde{u}_{1y}\|^{2}
+((\bar{u}_{1}\widetilde{u}_{1y})_{y},\widetilde{u}_{1y})
+(\frac{2}{3}\widetilde{\theta}_{yy},\widetilde{u}_{1y})
+((\frac{2\bar{\theta}}{3\bar{\rho}}\widetilde{\rho}_{y})_{y},\widetilde{u}_{1y})
\nonumber\\
&=-(J_{2y},\widetilde{u}_{1y})+(\epsilon^{1-a}[\frac{4}{3\rho}(\mu(\theta)u_{1y})_{y}]_{y},\widetilde{u}_{1y})
-([\frac{1}{\rho}(\int_{\mathbb{R}^{3}} v^{2}_{1}L^{-1}_{M}\Theta dv)_{y}]_{y},\widetilde{u}_{1y}).
\end{align}
By integration by parts and using the Sobolev imbedding theorem, Lemma \ref{lem5.3}, \eqref{3.3} as well as \eqref{3.47}, we get
\begin{align*}
|((\bar{u}_{1}\widetilde{u}_{1y})_{y},\widetilde{u}_{1y})|&=|(\bar{u}_{1y},\frac{\widetilde{u}^{2}_{1y}}{2})|
\leq C\|\bar{u}_{1y}\|_{L^{\infty}}\|\widetilde{u}_{1y}\|^{2}
\leq C\epsilon^{a}\delta^{-1}\epsilon^{a-1}\epsilon^{1-a}\|\widetilde{u}_{1y}\|^{2}
\\
&\leq Ck\epsilon^{\frac{12}{5}a-\frac{8}{5}}\mathcal{D}_{2}(\tau)
\leq Ck^{\frac{1}{12}}\epsilon^{\frac{3}{5}a-\frac{2}{5}}\mathcal{D}_{2}(\tau).
\end{align*}
We will compute the right-hand side of \eqref{3.49} term by term. Performing the similar calculations as \eqref{3.46}, we thereby obtain
\begin{align*}
|(J_{2y},\widetilde{u}_{1y})|
&=|(\partial_{y}\{\widetilde{u}_{1}\bar{u}_{1y}+\widetilde{u}_{1}
\widetilde{u}_{1y}+\frac{2}{3}\rho_{y}\frac{\bar{\rho}\widetilde{\theta}-\widetilde{\rho}\bar{\theta}}
{\rho\bar{\rho}}\},\widetilde{u}_{1y})|
\nonumber\\
&\leq C\epsilon^{\frac{7}{5}+\frac{1}{15}a}(\delta+\epsilon^{a}\tau)^{-\frac{4}{3}}+
Ck^{\frac{1}{12}}\epsilon^{\frac{3}{5}a-\frac{2}{5}}\mathcal{D}_{2}(\tau).
\end{align*}
For the second term on the right-hand side of \eqref{3.49},
we first use an integration by parts about $y$ to obtain
\begin{align*}
([\frac{4}{3\rho}(\mu(\theta)u_{1y})_{y}]_{y},\widetilde{u}_{1y})
=-(\frac{4}{3\rho}(\mu(\theta)\widetilde{u}_{1y})_{y},\widetilde{u}_{1yy})
-(\frac{4}{3\rho}(\mu(\theta)\bar{u}_{1y})_{y},\widetilde{u}_{1yy}).
\end{align*}
In view of the Sobolev imbedding theorem, Lemma \ref{lem5.3}, \eqref{3.3} and \eqref{3.4}, one can show that
\begin{align*}
-\epsilon^{1-a}(\frac{4}{3\rho}(\mu(\theta)\widetilde{u}_{1y})_{y},\widetilde{u}_{1yy})
&=-\epsilon^{1-a}(\frac{4}{3\rho}\mu(\theta)\widetilde{u}_{1yy},\widetilde{u}_{1yy})
-\epsilon^{1-a}(\frac{4}{3\rho}\mu'(\theta)\theta_{y}\widetilde{u}_{1y},\widetilde{u}_{1yy})
\\
&\leq -c_{5}\epsilon^{1-a}\|\widetilde{u}_{1yy}\|^{2}
+C\epsilon^{1-a}\|\theta_{y}\|_{L^{\infty}}\|\widetilde{u}_{1y}\|\|\widetilde{u}_{1yy}\|
\\
&\leq -c_{5}\epsilon^{1-a}\|\widetilde{u}_{1yy}\|^{2}+Ck^{\frac{1}{12}}\epsilon^{\frac{3}{5}a-\frac{2}{5}}\mathcal{D}_{2}(\tau),
\end{align*}
for some constant $c_{5}>0$. Similarly, it holds that
\begin{align*}
\epsilon^{1-a}|(\frac{4}{3\rho}(\mu(\theta)\bar{u}_{1y})_{y},\widetilde{u}_{1yy})|
\leq C\eta\epsilon^{1-a}\|\widetilde{u}_{1yy}\|^{2}
+C_{\eta}\epsilon^{\frac{7}{5}+\frac{1}{15}a}(\delta+\epsilon^{a}\tau)^{-\frac{4}{3}}
+ C_{\eta}k^{\frac{1}{12}}\epsilon^{\frac{3}{5}a-\frac{2}{5}}\mathcal{D}_{2}(\tau).
\end{align*}
Hence, by taking $\eta>0$  small enough, there exists a  constant $c_{6}>0$ such that
\begin{align}
\label{3.50}
\epsilon^{1-a}([\frac{4}{3\rho}(\mu(\theta)u_{1y})_{y}]_{y},\widetilde{u}_{1y})
\leq-c_{6}\epsilon^{1-a}\|\widetilde{u}_{1yy}\|^{2}
+C\epsilon^{\frac{7}{5}+\frac{1}{15}a}(\delta+\epsilon^{a}\tau)^{-\frac{4}{3}}+
Ck^{\frac{1}{12}}\epsilon^{\frac{3}{5}a-\frac{2}{5}}\mathcal{D}_{2}(\tau).
\end{align}
The estimations for the last term of \eqref{3.49} is more complicated. The integration by parts and \eqref{3.19} give that
\begin{align}
\label{3.51}
&-([\frac{1}{\rho}(\int_{\mathbb{R}^{3}} v^{2}_{1}L^{-1}_{M}\Theta dv)_{y}]_{y},\widetilde{u}_{1y})
=(\frac{1}{\rho}[R\theta\int_{\mathbb{R}^{3}}B_{11}(\frac{v-u}{\sqrt{R\theta}})\frac{\Theta}{M} dv]_{y},\widetilde{u}_{1yy})
\nonumber\\
&=(\frac{1}{\rho}\int_{\mathbb{R}^{3}}[R\theta B_{11}(\frac{v-u}{\sqrt{R\theta}})\frac{1}{M}]_{y}\Theta dv,\widetilde{u}_{1yy})
+(\frac{1}{\rho}\int_{\mathbb{R}^{3}}R\theta B_{11}(\frac{v-u}{\sqrt{R\theta}})\frac{\partial_{y}\Theta}{M} dv,\widetilde{u}_{1yy}).
\end{align}
Notice that the first term on the right hand side of \eqref{3.51} is the higher nonlinear term compared with \eqref{3.20} and is easier to estimate.
Therefore, we can follow the similar method as used \eqref{3.21}-\eqref{3.27} to deal with this term,
then we can arrive at
\begin{align*}
&|(\frac{1}{\rho}\int_{\mathbb{R}^{3}}[R\theta B_{11}(\frac{v-u}{\sqrt{R\theta}})\frac{1}{M}]_{y}\Theta dv,\widetilde{u}_{1yy})|\\
&\leq \eta\epsilon^{1-a}\|\widetilde{u}_{1yy}\|^{2}
+C_{\eta}\epsilon^{\frac{7}{5}+\frac{1}{15}a}(\delta+\epsilon^{a}\tau)^{-\frac{4}{3}}
+C_{\eta}k^{\frac{1}{12}}\epsilon^{\frac{3}{5}-\frac{2}{5}a}\mathcal{D}_{2}(\tau).
\end{align*}
The second term on the right hand side of \eqref{3.51} is similar to \eqref{3.20}. We thus have from
the similar arguments as \eqref{3.20} that
\begin{align*}
|(\frac{1}{\rho}\int_{\mathbb{R}^{3}}R\theta B_{11}(\frac{v-u}{\sqrt{R\theta}})\frac{\partial_{y}\Theta}{M} dv,\widetilde{u}_{1yy})|
&\leq \eta\epsilon^{1-a}\|\widetilde{u}_{1yy}\|^{2}
+C_{\eta}\epsilon^{1-a}\sum_{|\alpha|=2}\|\partial^{\alpha}f\|^{2}_{\sigma}
\nonumber\\
&\quad+C_{\eta}\epsilon^{\frac{7}{5}+\frac{1}{15}a}(\delta+\epsilon^{a}\tau)^{-\frac{4}{3}}
+C_{\eta}k^{\frac{1}{12}}\epsilon^{\frac{3}{5}-\frac{2}{5}a}\mathcal{D}_{2}(\tau).
\end{align*}
It follows from the above three estimates that
\begin{align}
\label{3.52}
-([\frac{1}{\rho}(\int_{\mathbb{R}^{3}} v^{2}_{1}L^{-1}_{M}\Theta dv)_{y}]_{y},\widetilde{u}_{1y})
&\leq C\eta\epsilon^{1-a}\|\widetilde{u}_{1yy}\|^{2}+C_{\eta}\epsilon^{1-a}\sum_{|\alpha|=2}\|\partial^{\alpha}f\|^{2}_{\sigma}
\nonumber\\
&\quad+C_{\eta}\epsilon^{\frac{7}{5}+\frac{1}{15}a}(\delta+\epsilon^{a}\tau)^{-\frac{4}{3}}
+C_{\eta}k^{\frac{1}{12}}\epsilon^{\frac{3}{5}-\frac{2}{5}a}\mathcal{D}_{2}(\tau).
\end{align}
Substituting the above related estimates into \eqref{3.49} and  taking $\eta>0$ small enough, we get
\begin{align}
\label{3.53}
&\frac{1}{2}\frac{d}{d\tau}\|\widetilde{u}_{1y}\|^{2}+(\frac{2}{3}\widetilde{\theta}_{yy},\widetilde{u}_{1y})
+((\frac{2\bar{\theta}}{3\bar{\rho}}\widetilde{\rho}_{y})_{y},\widetilde{u}_{1y})+c\epsilon^{1-a}\|\widetilde{u}_{1yy}\|^{2}
\nonumber\\
&\leq C\epsilon^{1-a}\sum_{|\alpha|=2}\|\partial^{\alpha}f\|^{2}_{\sigma}+C\epsilon^{\frac{7}{5}+\frac{1}{15}a}(\delta+\epsilon^{a}\tau)^{-\frac{4}{3}}
+C(k^{\frac{1}{12}}\epsilon^{\frac{3}{5}-\frac{2}{5}a}+k^{\frac{1}{12}}\epsilon^{\frac{3}{5}a-\frac{2}{5}})\mathcal{D}_{2}(\tau).
\end{align}
Similar to \eqref{3.49}, by differentiating \eqref{2.6}$_{3}$ with respect to $y$ and taking the inner product of
the resulting equation with $\widetilde{u}_{iy}$  $(i=2,3)$, the similar arguments as \eqref{3.53} imply that
\begin{align}
\label{3.54}
\frac{1}{2}\frac{d}{d\tau}\|\widetilde{u}_{iy}\|^{2}+c\epsilon^{1-a}\|\widetilde{u}_{iyy}\|^{2}
&\leq  C\epsilon^{1-a}\sum_{|\alpha|=2}\|\partial^{\alpha}f\|^{2}_{\sigma}
+C\epsilon^{\frac{7}{5}+\frac{1}{15}a}(\delta+\epsilon^{a}\tau)^{-\frac{4}{3}}
\nonumber\\
&\quad+C(k^{\frac{1}{12}}\epsilon^{\frac{3}{5}-\frac{2}{5}a}
+k^{\frac{1}{12}}\epsilon^{\frac{3}{5}a-\frac{2}{5}})\mathcal{D}_{2}(\tau).
\end{align}
\par
Finally, we still deal with \eqref{2.6}$_4$. Differentiating  \eqref{2.6}$_4$ with respect
to $y$, we then take the inner product of the resulting equation with $\frac{1}{\bar{\theta}}\widetilde{\theta}_{y}$  to get
\begin{align}
\label{3.55}
&(\widetilde{\theta}_{\tau y},\frac{1}{\bar{\theta}}\widetilde{\theta}_{y})
+(\frac{2}{3}\bar{\theta}\widetilde{u}_{1yy},\frac{1}{\bar{\theta}}\widetilde{\theta}_{y})+
(\frac{2}{3}\bar{\theta}_{y}\widetilde{u}_{1y},\frac{1}{\bar{\theta}}\widetilde{\theta}_{y})
+((\bar{u}_{1}\widetilde{\theta}_{y})_{y},\frac{1}{\bar{\theta}}\widetilde{\theta}_{y})
\nonumber\\
&=-(J_{3y},\frac{1}{\bar{\theta}}\widetilde{\theta}_{y})+\epsilon^{1-a}([\frac{1}{\rho}(\kappa(\theta)\theta_{y})_{y}]_{y},\frac{1}{\bar{\theta}}\widetilde{\theta}_{y})
+\epsilon^{1-a}([\frac{4}{3\rho}\mu(\theta)u^{2}_{1y}]_{y}
+[\frac{1}{\rho}\sum_{i=2}^{3}\mu(\theta)u^{2}_{iy}]_{y},\frac{1}{\bar{\theta}}\widetilde{\theta}_{y})
\nonumber\\
&\quad+([\frac{1}{\rho}u(\int_{\mathbb{R}^3} v_{1}v L^{-1}_{M}\Theta dv)_{y}]_{y},\frac{1}{\bar{\theta}}\widetilde{\theta}_{y})-([\frac{1}{\rho}
(\int_{\mathbb{R}^3}v_{1}\frac{|v|^{2}}{2}L^{-1}_{M}\Theta dv)_{y}]_{y},\frac{1}{\bar{\theta}}\widetilde{\theta}_{y}).
\end{align}
We will estimate \eqref{3.55} term by term. First of all,
one has
\begin{align*}
|(\frac{2}{3}\bar{\theta}_{y}\widetilde{u}_{1y},\frac{1}{\bar{\theta}}\widetilde{\theta}_{y})|
\leq C\|\bar{\theta}_{y}\|_{L^{\infty}}\|\widetilde{u}_{1y}\|\|\widetilde{\theta}_{y}\|
\leq Ck^{\frac{1}{12}}\epsilon^{\frac{3}{5}a-\frac{2}{5}}\mathcal{D}_{2}(\tau),
\end{align*}
according to the Sobolev imbedding theorem, Lemma \ref{lem5.3} and \eqref{3.3}.
Similarly, it holds that
\begin{align*}
|((\bar{u}_{1}\widetilde{\theta}_{y})_{y},\frac{1}{\bar{\theta}}\widetilde{\theta}_{y})|
&\leq|(\bar{u}_{1y}\widetilde{\theta}_{y},\frac{1}{\bar{\theta}}\widetilde{\theta}_{y})|
+|(\bar{u}_{1}\widetilde{\theta}_{yy},\frac{1}{\bar{\theta}}\widetilde{\theta}_{y})|
\\
&\leq|(\bar{u}_{1y}\widetilde{\theta}_{y},\frac{1}{\bar{\theta}}\widetilde{\theta}_{y})|
+|((\frac{\bar{u}_{1}}{\bar{\theta}})_{y},\frac{1}{2}\widetilde{\theta}^{2}_{y})|
\leq Ck^{\frac{1}{12}}\epsilon^{\frac{3}{5}a-\frac{2}{5}}\mathcal{D}_{2}(\tau).
\end{align*}
Recalling $J_{3}$ defined by \eqref{2.5} and following the similar method used in \eqref{3.46}, we have from \eqref{3.47} that
\begin{align*}
|(J_{3y},\frac{1}{\bar{\theta}}\widetilde{\theta}_{y})|
&=|(\{\frac{2}{3}(\widetilde{\theta}\bar{u}_{1y}+\widetilde{\theta}\widetilde{u}_{1y})
+(\widetilde{\theta}_{y}\widetilde{u}_{1}+\bar{\theta}_{y}\widetilde{u}_{1})\}_{y},\frac{1}{\bar{\theta}}\widetilde{\theta}_{y})|
\\
&\leq C\epsilon^{\frac{7}{5}+\frac{1}{15}a}(\delta+\epsilon^{a}\tau)^{-\frac{4}{3}}+ Ck^{\frac{1}{12}}\epsilon^{\frac{3}{5}a-\frac{2}{5}}\mathcal{D}_{2}(\tau).
\end{align*}
Similar arguments as \eqref{3.50} imply that
\begin{align*}
&\epsilon^{1-a}([\frac{1}{\rho}(\kappa(\theta)\theta_{y})_{y}]_{y},\frac{1}{\bar{\theta}}\partial_{y}\widetilde{\theta})
+\epsilon^{1-a}|([\frac{4}{3\rho}\mu(\theta)u^{2}_{1y}]_{y}
+[\frac{1}{\rho}\sum_{i=2}^{3}\mu(\theta)u^{2}_{iy}]_{y},\frac{1}{\bar{\theta}}\widetilde{\theta}_{y})|
\\
&\hspace{1cm}\leq-c\epsilon^{1-a}\|\widetilde{\theta}_{yy}\|^{2}
+C\epsilon^{\frac{7}{5}+\frac{1}{15}a}(\delta+\epsilon^{a}\tau)^{-\frac{4}{3}}+
Ck^{\frac{1}{12}}\epsilon^{\frac{3}{5}a-\frac{2}{5}}\mathcal{D}_{2}(\tau).
\end{align*}
On the other hand, we use the integration by parts, \eqref{3.18} and \eqref{3.19} to obtain
\begin{align*}
&([\frac{1}{\rho}u(\int_{\mathbb{R}^3} v_{1}v L^{-1}_{M}\Theta dv)_{y}]_{y},\frac{1}{\bar{\theta}}\widetilde{\theta}_{y})-([\frac{1}{\rho}
(\int_{\mathbb{R}^3}v_{1}\frac{|v|^{2}}{2}L^{-1}_{M}\Theta dv)_{y}]_{y},\frac{1}{\bar{\theta}}\widetilde{\theta}_{y})
\\
&=([\int_{\mathbb{R}^3}(v_{1}\frac{|v|^{2}}{2}-v_{1}uv) L^{-1}_{M}\Theta dv]_{y},\frac{1}{\rho}(\frac{1}{\bar{\theta}}\widetilde{\theta}_{y})_{y})+
(u_{y}\int_{\mathbb{R}^3} v_{1}v L^{-1}_{M}\Theta dv,\frac{1}{\rho}(\frac{1}{\bar{\theta}}\widetilde{\theta}_{y})_{y})
\\
&=([(R\theta)^{\frac{3}{2}}\int_{\mathbb{R}^{3}}A_{1}(\frac{v-u}{\sqrt{R\theta}})\frac{\Theta}{M} dv]_{y},\frac{1}{\rho}(\frac{1}{\bar{\theta}}\widetilde{\theta}_{y})_{y})+\sum_{i=1}^{3}
(u_{iy}R\theta\int_{\mathbb{R}^{3}}B_{1i}(\frac{v-u}{\sqrt{R\theta}})\frac{\Theta}{M} dv,\frac{1}{\rho}(\frac{1}{\bar{\theta}}\widetilde{\theta}_{y})_{y})\notag\\
&:=H_{5}.
\end{align*}
Notice that the above terms are similar to \eqref{3.51} and then the similar arguments as \eqref{3.52} imply that
\begin{align*}
H_{5}\leq \eta\epsilon^{1-a}\|\widetilde{\theta}_{yy}\|^{2}+C_{\eta}\epsilon^{1-a}\sum_{|\alpha|=2}\|\partial^{\alpha}f\|^{2}_{\sigma}
+C_{\eta}\epsilon^{\frac{7}{5}+\frac{1}{15}a}(\delta+\epsilon^{a}\tau)^{-\frac{4}{3}}
+C_{\eta}k^{\frac{1}{12}}\epsilon^{\frac{3}{5}-\frac{2}{5}a}\mathcal{D}_{2}(\tau).
\end{align*}
Substituting the above related estimates into \eqref{3.55} and choosing $\eta>0$ small enough, we can arrive at
\begin{align}
\label{3.56}
&\frac{1}{2}\frac{d}{d\tau}\|(\frac{1}{\bar{\theta}})^{1/2}\widetilde{\theta}_{y}\|^{2}
+(\frac{2}{3}\widetilde{u}_{1yy},\widetilde{\theta}_{y})+c\epsilon^{1-a}\|\widetilde{\theta}_{yy}\|^{2}
\nonumber\\
&\leq C\epsilon^{1-a}\sum_{|\alpha|=2}\|\partial^{\alpha}f\|^{2}_{\sigma}
+C\epsilon^{\frac{7}{5}+\frac{1}{15}a}(\delta+\epsilon^{a}\tau)^{-\frac{4}{3}}+C(k^{\frac{1}{12}}\epsilon^{\frac{3}{5}-\frac{2}{5}a}
+k^{\frac{1}{12}}\epsilon^{\frac{3}{5}a-\frac{2}{5}})\mathcal{D}_{2}(\tau).
\end{align}
As a consequence, the sum of \eqref{3.48}, \eqref{3.53}, \eqref{3.54} and \eqref{3.56} give
\begin{align}
\label{3.57}
&\frac{d}{d\tau}\Big\{\|(\frac{2\bar{\theta}}{3\bar{\rho}^{2}})^{1/2}\widetilde{\rho}_{y}\|^{2}+
\|\widetilde{u}_{y}\|^{2}+
\|(\frac{1}{\bar{\theta}})^{1/2}\widetilde{\theta}_{y}\|^{2}\Big\}+c\epsilon^{1-a}\|(\widetilde{u}_{yy},\widetilde{\theta}_{yy})\|^{2}
\nonumber\\
&\leq C\epsilon^{1-a}\sum_{|\alpha|=2}\|\partial^{\alpha}f\|^{2}_{\sigma}
+C\epsilon^{\frac{7}{5}+\frac{1}{15}a}(\delta+\epsilon^{a}\tau)^{-\frac{4}{3}}
+C(k^{\frac{1}{12}}\epsilon^{\frac{3}{5}-\frac{2}{5}a}+k^{\frac{1}{12}}\epsilon^{\frac{3}{5}a-\frac{2}{5}})\mathcal{D}_{2}(\tau).
\end{align}
Differentiating the equations \eqref{2.6} with respect to $\tau$ and multiplying the resulting equations
by $\frac{2\bar{\theta}}{3\bar{\rho}^{2}}\widetilde{\rho}_{\tau}$, $\widetilde{u}_{1\tau}$, $\widetilde{u}_{i\tau}$ with $i=2,3$ and
$\frac{1}{\bar{\theta}}\widetilde{\theta}_{\tau}$ respectively, then adding them together and integrating with respect to $y$ over $\mathbb{R}$, we have by the similar arguments as \eqref{3.57} such that
\begin{align}
\label{3.58}
&\frac{d}{d\tau}\big\{\|(\frac{2\bar{\theta}}{3\bar{\rho}^{2}})^{1/2}\widetilde{\rho}_{\tau}\|^{2}+
\|\widetilde{u}_{\tau}\|^{2}+
\|(\frac{1}{\bar{\theta}})^{1/2}\widetilde{\theta}_{\tau}\|^{2}\big\}
+c\epsilon^{1-a}\|(\widetilde{u}_{\tau y},\widetilde{\theta}_{\tau y})\|^{2}
\nonumber\\
&\leq C\epsilon^{1-a}\sum_{|\alpha|=2}\|\partial^{\alpha}f\|^{2}_{\sigma}
+C\epsilon^{\frac{7}{5}+\frac{1}{15}a}(\delta+\epsilon^{a}\tau)^{-\frac{4}{3}}
+C(k^{\frac{1}{12}}\epsilon^{\frac{3}{5}-\frac{2}{5}a}
+k^{\frac{1}{12}}\epsilon^{\frac{3}{5}a-\frac{2}{5}})\mathcal{D}_{2}(\tau).
\end{align}
\par
The dissipative terms $\|\widetilde{\rho}_{yy}\|^{2}$ and $\|\widetilde{\rho}_{\tau y}\|^{2}$ are not included
in the dissipation of \eqref{3.57} and \eqref{3.58}.
For this, we use the system \eqref{2.4} again.
Differentiating the first and second equation of \eqref{2.4} with respect to $y$ yields
\begin{align}
\label{3.59}
\begin{cases}
\widetilde{\rho}_{\tau y}+(\bar{\rho} \widetilde{u}_{1y})_{y}+(\bar{\rho}_{y} \widetilde{u}_{1})_{y}=-J_{1y},
\\
\widetilde{u}_{1\tau y}+(\bar{u}_{1}\widetilde{u}_{1y})_{y}
+\frac{2}{3}\widetilde{\theta}_{yy}+(\frac{2\bar{\theta}}{3\bar{\rho}}\widetilde{\rho}_{y})_{y}
=-(\frac{1}{\rho}\int v^{2}_{1}G_{y}dv)_{y}-J_{2y}.
\end{cases}
\end{align}
Taking the inner product of \eqref{3.59}$_2$ with $\widetilde{\rho}_{yy}$ and performing  the similar calculations as \eqref{3.32} yield that
\begin{align}
\label{3.60}
 \epsilon^{1-a}\|\widetilde{\rho}_{yy}\|^{2}\leq&-C\epsilon^{1-a}(\widetilde{u}_{1y},\widetilde{\rho}_{yy})_{\tau}
+C\epsilon^{1-a}\|(\widetilde{u}_{yy},\widetilde{\theta}_{yy})\|^{2}+C\epsilon^{1-a}\|f_{yy}\|_{\sigma}^{2}
\nonumber\\
&+C\epsilon^{\frac{7}{5}+\frac{1}{15}a}(\delta+\epsilon^{a}\tau)^{-\frac{4}{3}}
+Ck^{\frac{1}{12}}\epsilon^{\frac{3}{5}-\frac{2}{5}a}\mathcal{D}_{2}(\tau).
\end{align}
Similarly, the following estimate holds
\begin{align}
\label{3.61}
 \epsilon^{1-a}\|\widetilde{\rho}_{\tau y}\|^{2}
\leq C\epsilon^{1-a}\|(\widetilde{\rho}_{yy},\widetilde{u}_{yy})\|^{2}
+C\epsilon^{\frac{7}{5}+\frac{1}{15}a}(\delta+\epsilon^{a}\tau)^{-\frac{4}{3}}
+Ck^{\frac{1}{12}}\epsilon^{\frac{3}{5}-\frac{2}{5}a}\mathcal{D}_{2}(\tau).
\end{align}
It remains to estimate the dissipative term $\|(\widetilde{\rho}_{\tau\tau},\widetilde{u}_{\tau\tau},\widetilde{\theta}_{\tau\tau})\|^{2}$.
Differentiating the equations \eqref{2.4} with respect to $\tau$ and multiplying the resulting equations
by $\widetilde{\rho}_{\tau\tau}$, $\widetilde{u}_{1\tau\tau}$, $\widetilde{u}_{i\tau\tau}$ with $i=2,3$ and
$\widetilde{\theta}_{\tau\tau}$ respectively, then adding them together and integrating with respect to $y$ over $\mathbb{R}$,
one can arrive at
\begin{align}
\label{3.62}
\epsilon^{1-a}\|(\widetilde{\rho}_{\tau\tau},\widetilde{u}_{\tau\tau},\widetilde{\theta}_{\tau\tau})\|^{2}
&\leq C \epsilon^{1-a}\|(\widetilde{\rho}_{\tau y},\widetilde{u}_{\tau y},
\widetilde{\theta}_{\tau y})\|^{2}+ C\epsilon^{1-a}\sum_{|\alpha|=2}\|\partial^{\alpha}f\|^{2}_{\sigma}
\nonumber\\
&\quad+C\epsilon^{\frac{7}{5}+\frac{1}{15}a}(\delta+\epsilon^{a}\tau)^{-\frac{4}{3}}
+C(k^{\frac{1}{12}}\epsilon^{\frac{3}{5}-\frac{2}{5}a}+k^{\frac{1}{12}}\epsilon^{\frac{3}{5}a-\frac{2}{5}})\mathcal{D}_{2}(\tau).
\end{align}
\par
In summary, for some suitably large constants $\overline{C}_{1}\gg C_{1}>0$, we have from
a suitable linear combination of \eqref{3.57}, \eqref{3.58}, \eqref{3.60}, \eqref{3.61} and  \eqref{3.62} that
\begin{align}
\label{3.63}
&\frac{d}{d\tau}\sum_{|\alpha|=1}\Big\{\overline{C}_{1}
(\|(\frac{2\bar{\theta}}{3\bar{\rho}^{2}})^{1/2}\partial^{\alpha}\widetilde{\rho}\|^{2}+
\|\partial^{\alpha}\widetilde{u}\|^{2}+\|(\frac{1}{\bar{\theta}})^{1/2}\partial^{\alpha}\widetilde{\theta}\|^{2})
+C_{1}\epsilon^{1-a}(\widetilde{u}_{1y},\widetilde{\rho}_{yy})\Big\}
\nonumber\\
&\hspace{2cm}+c\epsilon^{1-a}\sum_{|\alpha|=2}\|\partial^{\alpha}(\widetilde{\rho},\widetilde{u},
\widetilde{\theta})\|^{2}
\nonumber\\
&\leq C\epsilon^{1-a}\sum_{|\alpha|=2}\|\partial^{\alpha}f\|^{2}_{\sigma}
+C\epsilon^{\frac{7}{5}+\frac{1}{15}a}(\delta+\epsilon^{a}\tau)^{-\frac{4}{3}}
+C(k^{\frac{1}{12}}\epsilon^{\frac{3}{5}-\frac{2}{5}a}
+k^{\frac{1}{12}}\epsilon^{\frac{3}{5}a-\frac{2}{5}})\mathcal{D}_{2}(\tau).
\end{align}
Integrating \eqref{3.63} with respect to $\tau$ and using \eqref{3.2} with \eqref{3.3}, we can obtain
\begin{align}
&\sum_{|\alpha|=1}\|\partial^{\alpha}(\widetilde{\rho},\widetilde{u},\widetilde{\theta})\|^{2}
+\epsilon^{1-a}\sum_{|\alpha|=2}\int^{\tau}_{0}\|\partial^{\alpha}(\widetilde{\rho},\widetilde{u},
\widetilde{\theta})\|^{2}ds
\nonumber\\
&\leq C\epsilon^{2(1-a)}\|\widetilde{\rho}_{yy}\|^{2}+C\epsilon^{1-a}\sum_{|\alpha|=2}\int^{\tau}_{0}\|\partial^{\alpha}f\|^{2}_{\sigma}ds
+Ck^{\frac{1}{3}}\epsilon^{\frac{6}{5}-\frac{4}{5}a}\notag\\
&\quad+C(k^{\frac{1}{12}}\epsilon^{\frac{3}{5}-\frac{2}{5}a}+k^{\frac{1}{12}}\epsilon^{\frac{3}{5}a-\frac{2}{5}})\int^{\tau}_{0}\mathcal{D}_{2}(s)ds.
\label{3.64}
\end{align}

Next, we will derive the derivative estimates for the microscopic component $f$.
Taking the derivative $\partial^{\alpha}$ of \eqref{2.11} with $|\alpha|=1$ and taking the inner product with $\partial^{\alpha}f$,
we can obtain
\begin{align}
\label{3.66}
&(\partial_{\tau}\partial^{\alpha}f+v_{1}\partial_{y}\partial^{\alpha}f
-\epsilon^{a-1}\mathcal{L}\partial^{\alpha}f,\partial^{\alpha}f)
-\epsilon^{a-1}(\partial^{\alpha}\Gamma(f,\frac{M-\mu}{\sqrt{\mu}})+\partial^{\alpha}\Gamma(\frac{M-\mu}{\sqrt{\mu}},f),\partial^{\alpha}f)
\nonumber\\
&=\epsilon^{a-1}(\partial^{\alpha}\Gamma(\frac{G}{\sqrt{\mu}},\frac{G}{\sqrt{\mu}}),\partial^{\alpha}f)
+(\frac{\partial^{\alpha}P_{0}(v_{1}\sqrt{\mu}\partial_{y}f)}{\sqrt{\mu}},\partial^{\alpha}f)
-(\frac{\partial^{\alpha}P_{1}(v_{1}\partial_{y}\overline{G})}{\sqrt{\mu}},\partial^{\alpha}f)
\nonumber\\
&\quad-(\frac{1}{\sqrt{\mu}}\partial^{\alpha}P_{1}v_{1}M\big\{\frac{|v-u|^{2}
\widetilde{\theta}_{y}}{2R\theta^{2}}+\frac{(v-u)\cdot\widetilde{u}_{y}}{R\theta}\big\},\partial^{\alpha}f)
-(\frac{\partial_{\tau}\partial^{\alpha}\overline{G}}{\sqrt{\mu}},\partial^{\alpha}f).
\end{align}
We will compute each term for \eqref{3.66}. First of all, we have from the integration by parts and \eqref{5.5} that
$$
(\partial_{\tau}\partial^{\alpha}f+v_{1}\partial_{y}\partial^{\alpha}f
-\epsilon^{a-1}\mathcal{L}\partial^{\alpha}f,\partial^{\alpha}f)\geq\frac{1}{2}\frac{d}{d\tau}\|\partial^{\alpha}f\|^{2}
+\sigma_{1}\epsilon^{a-1}\|\partial^{\alpha}f\|_{\sigma}^{2}.
$$
From \eqref{5.10} and \eqref{5.20}, one can see that
\begin{align*}
&\epsilon^{a-1}|(\partial^{\alpha}\Gamma(f,\frac{M-\mu}{\sqrt{\mu}})
+\partial^{\alpha}\Gamma(\frac{M-\mu}{\sqrt{\mu}},f),\partial^{\alpha}f)|
\\
&\leq C\eta\epsilon^{a-1}\|\partial^{\alpha}f\|^{2}_{\sigma}
+C_{\eta}(\eta_{0}+k^{\frac{1}{12}}\epsilon^{\frac{3}{5}-\frac{2}{5}a})\mathcal{D}_{2}(\tau),
\end{align*}
and
\begin{align*}
\epsilon^{a-1}|(\partial^{\alpha}\Gamma(\frac{G}{\sqrt{\mu}},\frac{G}{\sqrt{\mu}}),\partial^{\alpha}f)|
\leq  C\eta\epsilon^{a-1}\|\partial^{\alpha}f\|^{2}_{\sigma}
+C_{\eta}\epsilon^{\frac{7}{5}+\frac{1}{15}a}(\delta+\epsilon^{a}\tau)^{-\frac{4}{3}}
+C_{\eta}k^{\frac{1}{12}}\epsilon^{\frac{3}{5}-\frac{2}{5}a}\mathcal{D}_{2}(\tau).
\end{align*}
By \eqref{1.9}, \eqref{1.8}, \eqref{1.26}, \eqref{2.17}, the Sobolev imbedding theorem, Lemma \ref{lem5.3}, \eqref{3.3} and \eqref{3.4}, one has
\begin{align*}
|&(\frac{\partial^{\alpha}P_{0}(v_{1}\sqrt{\mu}\partial_{y}f)}{\sqrt{\mu}},\partial^{\alpha}f)|
\leq \|\sum_{i=0}^{4}\langle v\rangle^{\frac{1}{2}}\mu^{-\frac{1}{2}}\partial^{\alpha}(\langle v_{1}\sqrt{\mu}\partial_{y}f,\frac{\chi_{i}}{M}\rangle\chi_{i})\|
\|\langle v\rangle^{-\frac{1}{2}}\partial^{\alpha}f\|
\\
&\leq \eta\epsilon^{a-1}\|\partial^{\alpha}f\|^{2}_{\sigma}+C_{\eta}\epsilon^{1-a}\|\partial^{\alpha}\partial_{y}f\|^{2}_{\sigma}
+C_{\eta}\epsilon^{\frac{7}{5}+\frac{1}{15}a}(\delta+\epsilon^{a}\tau)^{-\frac{4}{3}}
+C_{\eta}k^{\frac{1}{12}}\epsilon^{\frac{3}{5}-\frac{2}{5}a}\mathcal{D}_{2}(\tau),
\end{align*}
and
\begin{align*}
&|(\frac{1}{\sqrt{\mu}}\partial^{\alpha}P_{1}v_{1}M\big\{\frac{|v-u|^{2}
\widetilde{\theta}_{y}}{2R\theta^{2}}+\frac{(v-u)\cdot\widetilde{u}_{y}}{R\theta}\big\},\partial^{\alpha}f)|
\\
&\leq C\|\langle v\rangle^{\frac{1}{2}}\mu^{-\frac{1}{2}}\partial^{\alpha}P_{1}v_{1}M\big\{\frac{|v-u|^{2}
\widetilde{\theta}_{y}}{2R\theta^{2}}+\frac{(v-u)\cdot\widetilde{u}_{y}}{R\theta}\big\}\|\|\langle v\rangle^{-\frac{1}{2}}\partial^{\alpha}f\|
\\
&\leq \eta\epsilon^{a-1}\|\partial^{\alpha}f\|^{2}_{\sigma}
+C_{\eta}\epsilon^{1-a}\|(\partial^{\alpha}\widetilde{u}_{y},\partial^{\alpha}\widetilde{\theta}_{y})\|^{2}
+C_{\eta}\epsilon^{\frac{7}{5}+\frac{1}{15}a}(\delta+\epsilon^{a}\tau)^{-\frac{4}{3}}
+C_{\eta}k^{\frac{1}{12}}\epsilon^{\frac{3}{5}-\frac{2}{5}a}\mathcal{D}_{2}(\tau).
\end{align*}
Here we have used the fact that $|\langle v\rangle^{m}\mu^{-\frac{1}{2}}M|^{2}_{\sigma}\leq C$
for any $m\geq0$  by \eqref{1.26}.
The terms involving $\overline{G}$ are dominated by
\begin{align*}
&|(\frac{\partial^{\alpha}P_{1}(v_{1}\partial_{y}\overline{G})}{\sqrt{\mu}},\partial^{\alpha}f)|
+|(\frac{\partial_{\tau}\partial^{\alpha}\overline{G}}{\sqrt{\mu}},\partial^{\alpha}f)|
\\
& \leq C\eta\epsilon^{a-1}\|\partial^{\alpha}f\|^{2}_{\sigma}
+C_{\eta}\epsilon^{\frac{7}{5}+\frac{1}{15}a}(\delta+\epsilon^{a}\tau)^{-\frac{4}{3}}
+C_{\eta}k^{\frac{1}{12}}\epsilon^{\frac{3}{5}-\frac{2}{5}a}\mathcal{D}_{2}(\tau),
\end{align*}
according to \eqref{5.26}, \eqref{1.9}, the Sobolev imbedding theorem, Lemma \ref{lem5.3}, \eqref{3.3} and \eqref{3.4}.
Hence, by taking $\eta>0$ small enough, we
deduce from \eqref{3.66} and the above related estimates that
\begin{align}
\label{3.67}
&\sum_{|\alpha|=1}\frac{1}{2}\frac{d}{d\tau}\|\partial^{\alpha}f\|^{2}
+c\sum_{|\alpha|=1}\epsilon^{a-1}\|\partial^{\alpha}f\|_{\sigma}^{2}
\nonumber\\
&\hspace{1cm}\leq C\epsilon^{1-a}\sum_{|\alpha|=1}\Big\{\|\partial^{\alpha}(\widetilde{\rho}_{y},\widetilde{u}_{y},\widetilde{\theta}_{y})\|^{2}+
\|\partial^{\alpha}f_y\|_{\sigma}^{2}\Big\}
\nonumber\\
&\hspace{2cm}+C\epsilon^{\frac{7}{5}+\frac{1}{15}a}(\delta+\epsilon^{a}\tau)^{-\frac{4}{3}}
+C(\eta_{0}+k^{\frac{1}{12}}\epsilon^{\frac{3}{5}-\frac{2}{5}a})\mathcal{D}_{2}(\tau).
\end{align}
Integrating \eqref{3.67} with respect to $\tau$ and using \eqref{3.2} with \eqref{3.3}, then by a suitable linear combination of the resulting equation and \eqref{3.64},
we get
\begin{align}
\label{3.68}
&\sum_{|\alpha|=1}(\|\partial^{\alpha}(\widetilde{\rho},\widetilde{u},\widetilde{\theta})\|^{2}+\|\partial^{\alpha}f\|^{2})
+\epsilon^{1-a}\sum_{|\alpha|=2}\int^{\tau}_{0}\|\partial^{\alpha}(\widetilde{\rho},\widetilde{u},
\widetilde{\theta})\|^{2}ds+\epsilon^{a-1}\sum_{|\alpha|=1}\int^{\tau}_{0}\|\partial^{\alpha}f\|_{\sigma}^{2}ds
\nonumber\\
&\leq C\epsilon^{2(1-a)}\|\widetilde{\rho}_{yy}\|^{2}+C\epsilon^{1-a}\sum_{|\alpha|=2}\int^{\tau}_{0}\|\partial^{\alpha}f\|^{2}_{\sigma}ds
\nonumber\\
&\quad+Ck^{\frac{1}{3}}\epsilon^{\frac{6}{5}-\frac{4}{5}a}+C(\eta_{0}+k^{\frac{1}{12}}\epsilon^{\frac{3}{5}-\frac{2}{5}a}
+k^{\frac{1}{12}}\epsilon^{\frac{3}{5}a-\frac{2}{5}})\int^{\tau}_{0}\mathcal{D}_{2}(s)ds.
\end{align}
\par
Finally, we  derive the higher order derivative
estimates of the microscopic component $f$. In terms of \eqref{2.8},
\eqref{2.9} and \eqref{2.13}, one has
\begin{align}
\label{3.69}
\partial_{\tau}(\frac{F}{\sqrt{\mu}})+v_{1}\partial_{y}(\frac{F}{\sqrt{\mu}})
-\epsilon^{a-1}\mathcal{L}f&=\epsilon^{a-1}\Gamma(f,\frac{M-\mu}{\sqrt{\mu}})+
\epsilon^{a-1}\Gamma(\frac{M-\mu}{\sqrt{\mu}},f)
+\epsilon^{a-1}\Gamma(\frac{G}{\sqrt{\mu}},\frac{G}{\sqrt{\mu}})
\nonumber\\
&\quad+\frac{1}{\sqrt{\mu}}P_{1}v_{1}M\big\{\frac{|v-u|^{2}
\overline{\theta}_{y}}{2R\theta^{2}}+\frac{(v-u)\cdot\bar{u}_{y}}{R\theta}\big\}.
\end{align}
Taking the derivative $\partial^{\alpha}$ of \eqref{3.69} with $|\alpha|=2$
and then  taking the inner product of the resulting equation with $\frac{\partial^{\alpha}F}{\sqrt{\mu}}$, we obtain
\begin{align}
\label{3.70}
&\frac{1}{2}\frac{d}{d\tau}\|\frac{\partial^{\alpha}F}{\sqrt{\mu}}\|^{2}
-\epsilon^{a-1}(\mathcal{L}\partial^{\alpha}f,\frac{\partial^{\alpha}F}{\sqrt{\mu}})
=\epsilon^{a-1}(\partial^{\alpha}\Gamma(f,\frac{M-\mu}{\sqrt{\mu}})+
\partial^{\alpha}\Gamma(\frac{M-\mu}{\sqrt{\mu}},f),\frac{\partial^{\alpha}F}{\sqrt{\mu}})
\nonumber\\
&+\epsilon^{a-1}(\partial^{\alpha}\Gamma(\frac{G}{\sqrt{\mu}},\frac{G}{\sqrt{\mu}}),\frac{\partial^{\alpha}F}{\sqrt{\mu}})
+(\frac{1}{\sqrt{\mu}}\partial^{\alpha}P_{1}v_{1}M\big\{\frac{|v-u|^{2}
\overline{\theta}_{y}}{2R\theta^{2}}+\frac{(v-u)\cdot\bar{u}_{y}}{R\theta}\big\},\frac{\partial^{\alpha}F}{\sqrt{\mu}}).
\end{align}
Recall $F=M+\overline{G}+\sqrt{\mu}f$, we first have from \eqref{5.5} that
$$
-\epsilon^{a-1}(\mathcal{L}\partial^{\alpha}f,\partial^{\alpha}f)
\geq\sigma_{1}\epsilon^{a-1}\|\partial^{\alpha}f\|_{\sigma}^{2}.
$$
Due to $\mathcal{L}f=\Gamma(\sqrt{\mu},f)+\Gamma(f,\sqrt{\mu})$, we get from \eqref{5.7} that
\begin{align*}
&\epsilon^{a-1}|(\mathcal{L}\partial^{\alpha}f,\frac{\partial^{\alpha}\overline{G}}{\sqrt{\mu}})|
\leq C\epsilon^{a-1}\|\partial^{\alpha}f\|_{\sigma}\|\frac{\partial^{\alpha}\overline{G}}{\sqrt{\mu}}\|_{\sigma}
\nonumber\\
&\leq \eta\epsilon^{a-1}\|\partial^{\alpha}f\|^{2}_{\sigma}
+C_{\eta}\big\{\|\partial^{\alpha}(\bar{u}_{1y},\bar{\theta}_{y})\|^{2}+\cdot\cdot\cdot
+\|(\bar{u}_{1y},\bar{\theta}_{y})\partial^{\alpha}(u,\theta)\|^{2}\big\}
\nonumber\\
&\leq \eta\epsilon^{a-1}\|\partial^{\alpha}f\|^{2}_{\sigma}
+C_{\eta}\epsilon^{\frac{7}{5}+\frac{1}{15}a}(\delta+\epsilon^{a}\tau)^{-\frac{4}{3}}
+C_{\eta}k^{\frac{1}{12}}\epsilon^{\frac{3}{5}-\frac{2}{5}a}\mathcal{D}_{2}(\tau),
\end{align*}
according to \eqref{5.26}, the Sobolev imbedding theorem, Lemma \ref{lem5.3} and \eqref{3.3}.
For $\bar{\alpha}\geq 2$, we can see that
\begin{align}
\label{3.72}
\partial^{\bar{\alpha}}M&=M\Big(\frac{\partial^{\bar{\alpha}}\rho}{\rho}-\frac{3\partial^{\bar{\alpha}}\theta}{2\theta} +\frac{(v-u)^{2}\partial^{\bar{\alpha}}\theta}{2R\theta^{2}}+\sum^{3}_{i=1}
\frac{\partial^{\bar{\alpha}}u_{i}(v_{i}-u_{i})}{R\theta}\Big)+\cdot\cdot\cdot
\nonumber\\
&=\Big(\mu+(M-\mu)\Big)\Big(\frac{\partial^{\bar{\alpha}}\rho}{\rho}-\frac{3\partial^{\bar{\alpha}}\theta}{2\theta} +\frac{(v-u)^{2}\partial^{\bar{\alpha}}\theta}{2R\theta^{2}}+\sum^{3}_{i=1}
\frac{\partial^{\bar{\alpha}}u_{i}(v_{i}-u_{i})}{R\theta}\Big)+\cdot\cdot\cdot\notag\\
&=J^{\bar{\alpha}}_{1}
+J^{\bar{\alpha}}_{2}+J^{\bar{\alpha}}_{3}.
\end{align}
Here the terms $J_{1}$ and $J_{2}$ are the high order derivatives of $(\rho,u,\theta)$
with $\mu$ and $M-\mu$ and $J_{3}$ is the low order derivatives with $M$.
Since $\frac{J^{\alpha}_{1}}{\sqrt{\mu}}\in\ker{\mathcal{L}}$, it follows that
$(\mathcal{L}f,\frac{J^{\alpha}_{1}}{\sqrt{\mu}})=0$. For the terms $\frac{J^{\alpha}_{2}}{\sqrt{\mu}}$ and
$\frac{J^{\alpha}_{3}}{\sqrt{\mu}}$, we use \eqref{5.7}, Lemma \ref{lem5.3} and \eqref{3.3} to get
\begin{align*}
\epsilon^{a-1}|(\mathcal{L}\partial^{\alpha}f,\frac{J^{\alpha}_{2}}{\sqrt{\mu}})|
&\leq C\epsilon^{a-1}\|\partial^{\alpha}f\|_{\sigma}\|\frac{J^{\alpha}_{2}}{\sqrt{\mu}}\|_{\sigma}
\\
&\leq C\eta_{0}\epsilon^{a-1}(\|\partial^{\alpha}f\|_{\sigma}^{2}+\|\partial^{\alpha}(\widetilde{\rho},\widetilde{u},\widetilde{\theta})\|^{2})
+C\epsilon^{a-1}\epsilon^{\frac{7}{5}+\frac{1}{15}a}(\delta+\epsilon^{a}\tau)^{-\frac{4}{3}},
\end{align*}
and
\begin{align*}
&\epsilon^{a-1}|(\mathcal{L}\partial^{\alpha}f,\frac{J^{\alpha}_{3}}{\sqrt{\mu}})|
\leq C\epsilon^{a-1}\|\partial^{\alpha}f\|_{\sigma}\|\frac{J^{\alpha}_{3}}{\sqrt{\mu}}\|_{\sigma}
\\
&\leq \eta\epsilon^{a-1}\|\partial^{\alpha}f\|_{\sigma}^{2}
+C_{\eta}\epsilon^{a-1}\epsilon^{\frac{7}{5}+\frac{1}{15}a}(\delta+\epsilon^{a}\tau)^{-\frac{4}{3}}
+C_{\eta}\epsilon^{a-1}k^{\frac{1}{12}}\epsilon^{\frac{3}{5}-\frac{2}{5}a}\mathcal{D}_{2}(\tau),
\end{align*}
where we have used $|\langle v\rangle^{m}\mu^{-\frac{1}{2}}M|^{2}_{\sigma}\leq C$
for any $m\geq0$  by \eqref{1.26}. Owing to these, we thereby obtain
\begin{align*}
\epsilon^{a-1}|(\mathcal{L}\partial^{\alpha}f,\frac{\partial^{\alpha}M}{\sqrt{\mu}})|
&\leq  C(\eta+\eta_{0})\epsilon^{a-1}\|\partial^{\alpha}f\|_{\sigma}^{2}
+C\eta_{0}\epsilon^{a-1}\|\partial^{\alpha}(\widetilde{\rho},\widetilde{u},\widetilde{\theta})\|^{2}
\nonumber\\
&\quad+C_{\eta}\epsilon^{a-1}\epsilon^{\frac{7}{5}+\frac{1}{15}a}(\delta+\epsilon^{a}\tau)^{-\frac{4}{3}}
+C_{\eta}\epsilon^{a-1}k^{\frac{1}{12}}\epsilon^{\frac{3}{5}-\frac{2}{5}a}\mathcal{D}_{2}(\tau).
\end{align*}
As a consequence, the second term on the left-hand side of \eqref{3.70} is controlled by
\begin{align*}
\epsilon^{a-1}(\mathcal{L}\partial^{\alpha}f,\frac{\partial^{\alpha}F}{\sqrt{\mu}})
\leq&-\sigma_{1}\epsilon^{a-1}\|\partial^{\alpha}f\|_{\sigma}^{2}+
C(\eta+\eta_{0})\epsilon^{a-1}\|\partial^{\alpha}f\|_{\sigma}^{2}
+C\eta_{0}\epsilon^{a-1}\|\partial^{\alpha}(\widetilde{\rho},\widetilde{u},\widetilde{\theta})\|^{2}
\nonumber\\
&+C_{\eta}\epsilon^{a-1}\epsilon^{\frac{7}{5}+\frac{1}{15}a}(\delta+\epsilon^{a}\tau)^{-\frac{4}{3}}
+C_{\eta}\epsilon^{a-1}k^{\frac{1}{12}}\epsilon^{\frac{3}{5}-\frac{2}{5}a}\mathcal{D}_{2}(\tau).
\end{align*}
For $\bar{\alpha}=2$, recalling $F=M+\overline{G}+\sqrt{\mu}f$,
we get from \eqref{5.26}, \eqref{3.72}, the Sobolev imbedding theorem, Lemma \ref{lem5.3}, \eqref{3.3} and \eqref{3.4} that
\begin{align}
\label{3.75}
\|\frac{\partial^{\bar{\alpha}}F}{\sqrt{\mu}}\|_{\sigma}^{2}
&\leq \|\frac{\partial^{\bar{\alpha}}\sqrt{\mu}f}{\sqrt{\mu}}\|_{\sigma}^{2}
+\|\frac{\partial^{\bar{\alpha}}\overline{G}}{\sqrt{\mu}}\|_{\sigma}^{2}
+\|\frac{\partial^{\bar{\alpha}}M}{\sqrt{\mu}}\|_{\sigma}^{2}
\nonumber\\
&\leq C(\|\partial^{\bar{\alpha}}f\|^{2}_{\sigma}
+\|\partial^{\bar{\alpha}}(\widetilde{\rho},\widetilde{u},\widetilde{\theta})\|^{2})+C\epsilon^{\frac{7}{5}+\frac{1}{15}a}(\delta+\epsilon^{a}\tau)^{-\frac{4}{3}}
+Ck^{\frac{1}{12}}\epsilon^{\frac{3}{5}-\frac{2}{5}a}\mathcal{D}_{2}(\tau).
\end{align}
For the first term on the right-hand side of \eqref{3.70}, we directly apply \eqref{5.7} to get
\begin{equation}
\label{3.76}
\epsilon^{a-1}|(\partial^{\alpha}\Gamma(\frac{M-\mu}{\sqrt{\mu}},f),\frac{\partial^{\alpha}F}{\sqrt{\mu}})|\leq C\epsilon^{a-1}\sum_{|\alpha_{1}|\leq|\alpha|}
\int_{\mathbb{R}}|\partial^{\alpha_{1}}(\frac{M-\mu}{\sqrt{\mu}})|_{2}|\partial^{\alpha-\alpha_{1}}f|_{\sigma}|\frac{\partial^{\alpha}F}{\sqrt{\mu}}|_{\sigma}dy.
\end{equation}
For $|\alpha_{1}|=0$ in \eqref{3.76}, we use \eqref{5.12} and \eqref{3.75} to obtain
\begin{align}
\label{3.77}
&\epsilon^{a-1}\int_{\mathbb{R}}|\partial^{\alpha_{1}}(\frac{M-\mu}{\sqrt{\mu}})|_{2}
|\partial^{\alpha-\alpha_{1}}f|_{\sigma}|\frac{\partial^{\alpha}F}{\sqrt{\mu}}|_{\sigma}dy
\leq C\eta_{0}\epsilon^{a-1}\|\partial^{\alpha}f\|_{\sigma}\|\frac{\partial^{\alpha}F}{\sqrt{\mu}}\|_{\sigma}
\nonumber\\
&\leq C\eta_{0}\epsilon^{a-1}(\|\partial^{\alpha}f\|^{2}_{\sigma}
+\|\partial^{\alpha}(\widetilde{\rho},\widetilde{u},\widetilde{\theta})\|^{2})\notag\\
&\quad+C\eta_{0}\epsilon^{a-1}\epsilon^{\frac{7}{5}+\frac{1}{15}a}(\delta+\epsilon^{a}\tau)^{-\frac{4}{3}}
+C\eta_{0}\epsilon^{a-1}k^{\frac{1}{12}}\epsilon^{\frac{3}{5}-\frac{2}{5}a}\mathcal{D}_{2}(\tau).
\end{align}
For $1\leq |\alpha_{1}|\leq |\alpha|$ in \eqref{3.76}, it holds that
\begin{align}
\label{3.78}
&\epsilon^{a-1}\int_{\mathbb{R}}|\partial^{\alpha_{1}}(\frac{M-\mu}{\sqrt{\mu}})|_{2}
|\partial^{\alpha-\alpha_{1}}f|_{\sigma}|\frac{\partial^{\alpha}F}{\sqrt{\mu}}|_{\sigma}dy
\nonumber\\
&\leq C\epsilon^{a-1}\big\{\|\partial^{\alpha_{1}}(\rho,u,\theta)\|
+\sum_{|\alpha'|=1}\|\partial^{\alpha'}(\rho,u,\theta)\|^{|\alpha_{1}|}\big\}
\sup_{y\in\mathbb{R}}|\partial^{\alpha-\alpha_{1}}f|_{\sigma}
\|\frac{\partial^{\alpha}F}{\sqrt{\mu}}\|_{\sigma}
\nonumber\\
&\leq \eta\epsilon^{a-1}(\|\partial^{\alpha}f\|^{2}_{\sigma}
+\|\partial^{\alpha}(\widetilde{\rho},\widetilde{u},\widetilde{\theta})\|^{2})
+C_{\eta}\epsilon^{a-1}\epsilon^{\frac{7}{5}+\frac{1}{15}a}(\delta+\epsilon^{a}\tau)^{-\frac{4}{3}}
+C_{\eta}\epsilon^{a-1}k^{\frac{1}{12}}\epsilon^{\frac{3}{5}-\frac{2}{5}a}\mathcal{D}_{2}(\tau).
\end{align}
The term $\epsilon^{a-1}(\partial^{\alpha}\Gamma(f,\frac{M-\mu}{\sqrt{\mu}}),\frac{\partial^{\alpha}F}{\sqrt{\mu}})$
can be treated in the similar way as in \eqref{3.76}. We thus get from \eqref{3.77} and \eqref{3.78} that
\begin{align*}
&\epsilon^{a-1}|(\partial^{\alpha}\Gamma(f,\frac{M-\mu}{\sqrt{\mu}})+
\partial^{\alpha}\Gamma(\frac{M-\mu}{\sqrt{\mu}},f),\frac{\partial^{\alpha}F}{\sqrt{\mu}})|
\nonumber\\
&\leq C(\eta+\eta_{0})\epsilon^{a-1}(\|\partial^{\alpha}f\|^{2}_{\sigma}
+\|\partial^{\alpha}(\widetilde{\rho},\widetilde{u},\widetilde{\theta})\|^{2})
\nonumber\\
&\quad+C_{\eta}\epsilon^{a-1}\epsilon^{\frac{7}{5}+\frac{1}{15}a}(\delta+\epsilon^{a}\tau)^{-\frac{4}{3}}
+C_{\eta}\epsilon^{a-1}k^{\frac{1}{12}}\epsilon^{\frac{3}{5}-\frac{2}{5}a}\mathcal{D}_{2}(\tau).
\end{align*}
Due to $G=\overline{G}+\sqrt{\mu}f$, one has
\begin{equation*}
\Gamma(\frac{G}{\sqrt{\mu}},\frac{G}{\sqrt{\mu}})=\Gamma(\frac{\overline{G}}{\sqrt{\mu}},\frac{\overline{G}}{\sqrt{\mu}})
+\Gamma(\frac{\overline{G}}{\sqrt{\mu}},f)+\Gamma(f,\frac{\overline{G}}{\sqrt{\mu}})
+\Gamma(f,f).
\end{equation*}
For the second term on the right-hand side of \eqref{3.70},
we apply \eqref{3.75} and perform the similar method as \eqref{5.27} and \eqref{5.28} to obtain
\begin{align}
\label{3.80}
\epsilon^{a-1}|(\partial^{\alpha}\Gamma(\frac{\overline{G}}{\sqrt{\mu}},\frac{\overline{G}}{\sqrt{\mu}}),\frac{\partial^{\alpha}F}{\sqrt{\mu}})|
&\leq \eta\epsilon^{a-1}(\|\partial^{\alpha}f\|^{2}_{\sigma}
+\|\partial^{\alpha}(\widetilde{\rho},\widetilde{u},\widetilde{\theta})\|^{2})
\nonumber\\
&\quad +C_{\eta}\epsilon^{a-1}\epsilon^{\frac{7}{5}+\frac{1}{15}a}(\delta+\epsilon^{a}\tau)^{-\frac{4}{3}}
+C_{\eta}\epsilon^{a-1}k^{\frac{1}{12}}\epsilon^{\frac{3}{5}-\frac{2}{5}a}\mathcal{D}_{2}(\tau).
\end{align}
Following the similar method used as \eqref{5.29}, we have by using \eqref{3.75} that
\begin{align*}
&\epsilon^{a-1}|(\partial^{\alpha}\Gamma(\frac{\overline{G}}{\sqrt{\mu}},f),\frac{\partial^{\alpha}F}{\sqrt{\mu}})|
+\epsilon^{a-1}|(\partial^{\alpha}\Gamma(f,\frac{\overline{G}}{\sqrt{\mu}}),\frac{\partial^{\alpha}F}{\sqrt{\mu}})|
\nonumber\\
&\leq C\eta\epsilon^{a-1}(\|\partial^{\alpha}f\|^{2}_{\sigma}
+\|\partial^{\alpha}(\widetilde{\rho},\widetilde{u},\widetilde{\theta})\|^{2})
+C_{\eta}\epsilon^{a-1}\epsilon^{\frac{7}{5}+\frac{1}{15}a}(\delta+\epsilon^{a}\tau)^{-\frac{4}{3}}
+C_{\eta}\epsilon^{a-1}k^{\frac{1}{12}}\epsilon^{\frac{3}{5}-\frac{2}{5}a}\mathcal{D}_{2}(\tau).
\end{align*}
With \eqref{5.7} and the Sobolev imbedding theorem in hand, we get from \eqref{3.3}, \eqref{3.4} and \eqref{3.75}
as well as the Cauchy-Schwarz inequality that
\begin{align}
\label{3.82}
\epsilon^{a-1}|(\partial^{\alpha}\Gamma(f,f),\frac{\partial^{\alpha}F}{\sqrt{\mu}})|
&\leq C\eta\epsilon^{a-1}(\|\partial^{\alpha}f\|^{2}_{\sigma}
+\|\partial^{\alpha}(\widetilde{\rho},\widetilde{u},\widetilde{\theta})\|^{2})
\nonumber\\
&\quad+C_{\eta}\epsilon^{a-1}\epsilon^{\frac{7}{5}+\frac{1}{15}a}(\delta+\epsilon^{a}\tau)^{-\frac{4}{3}}
+C_{\eta}\epsilon^{a-1}k^{\frac{1}{12}}\epsilon^{\frac{3}{5}-\frac{2}{5}a}\mathcal{D}_{2}(\tau).
\end{align}
From \eqref{3.80} to \eqref{3.82}, we can conclude that
\begin{align*}
\epsilon^{a-1}(\partial^{\alpha}\Gamma(\frac{G}{\sqrt{\mu}},\frac{G}{\sqrt{\mu}}),\frac{\partial^{\alpha}F}{\sqrt{\mu}})
&\leq C\eta\epsilon^{a-1}(\|\partial^{\alpha}f\|^{2}_{\sigma}
+\|\partial^{\alpha}(\widetilde{\rho},\widetilde{u},\widetilde{\theta})\|^{2})
\nonumber\\
&\quad+C_{\eta}\epsilon^{a-1}\epsilon^{\frac{7}{5}+\frac{1}{15}a}(\delta+\epsilon^{a}\tau)^{-\frac{4}{3}}
+C_{\eta}\epsilon^{a-1}k^{\frac{1}{12}}\epsilon^{\frac{3}{5}-\frac{2}{5}a}\mathcal{D}_{2}(\tau).
\end{align*}
For the last term on the right-hand side of \eqref{3.70}, one has from \eqref{2.17} and \eqref{3.75} that
\begin{align*}
&|(\frac{1}{\sqrt{\mu}}\partial^{\alpha}P_{1}v_{1}M\big\{\frac{|v-u|^{2}
\overline{\theta}_{y}}{2R\theta^{2}}+\frac{(v-u)\cdot\bar{u}_{y}}{R\theta}\big\},\frac{\partial^{\alpha}F}{\sqrt{\mu}})|
\nonumber\\
&\leq C\|\langle v\rangle^{\frac{1}{2}}\frac{1}{\sqrt{\mu}}\partial^{\alpha}P_{1}v_{1}M\big\{\frac{|v-u|^{2}
\overline{\theta}_{y}}{2R\theta^{2}}+\frac{(v-u)\cdot\bar{u}_{y}}{R\theta}\big\}\|
\|\langle v\rangle^{-\frac{1}{2}}\frac{\partial^{\alpha}F}{\sqrt{\mu}}\|
\nonumber\\
&\leq C\eta\epsilon^{a-1}(\|\partial^{\alpha}f\|^{2}_{\sigma}
+\|\partial^{\alpha}(\widetilde{\rho},\widetilde{u},\widetilde{\theta})\|^{2})
+C_{\eta}\epsilon^{a-1}\epsilon^{\frac{7}{5}+\frac{1}{15}a}(\delta+\epsilon^{a}\tau)^{-\frac{4}{3}}
+C_{\eta}\epsilon^{a-1}k^{\frac{1}{12}}\epsilon^{\frac{3}{5}-\frac{2}{5}a}\mathcal{D}_{2}(\tau).
\end{align*}
Hence, it holds by those above estimates and for any $\eta>0$ and  $\eta_{0}>0$ small enough that
\begin{align}
\label{3.84}
&\frac{1}{2}\frac{d}{d\tau}\sum_{|\alpha|=2}\|\frac{\partial^{\alpha}F}{\sqrt{\mu}}\|^{2}
+c\epsilon^{a-1}\sum_{|\alpha|=2}\|\partial^{\alpha}f\|_{\sigma}^{2}\notag\\
&\leq C(\eta+\eta_{0})\epsilon^{a-1}\sum_{|\alpha|=2}\|\partial^{\alpha}(\widetilde{\rho},\widetilde{u},\widetilde{\theta})\|^{2}+C_{\eta}\epsilon^{a-1}\epsilon^{\frac{7}{5}+\frac{1}{15}a}(\delta+\epsilon^{a}\tau)^{-\frac{4}{3}}
+C_{\eta}\epsilon^{a-1}k^{\frac{1}{12}}\epsilon^{\frac{3}{5}-\frac{2}{5}a}\mathcal{D}_{2}(\tau).
\end{align}
Integrating \eqref{3.84} with respect to $\tau$  and then multiplying the resulting equation by $\epsilon^{2(1-a)}$
with $a\in[\frac{2}{3},1]$, we can obtain
\begin{align}
\label{3.85}
&\epsilon^{2(1-a)}\sum_{|\alpha|=2}(\|\partial^{\alpha}(\widetilde{\rho},\widetilde{u},\widetilde{\theta})\|^{2}+\|\partial^{\alpha}f\|^{2})
+\epsilon^{1-a}\sum_{|\alpha|=2}\int^{t}_{0}\|\partial^{\alpha}f\|_{\sigma}^{2}ds
\nonumber\\
&\leq C(\eta+\eta_{0})\epsilon^{1-a}\sum_{|\alpha|=2}\int^{\tau}_{0}\|\partial^{\alpha}(\widetilde{\rho},\widetilde{u},\widetilde{\theta})\|^{2}ds
+C_{\eta}k^{\frac{1}{3}}\epsilon^{\frac{6}{5}-\frac{4}{5}a}
+C_{\eta}k^{\frac{1}{12}}\epsilon^{\frac{3}{5}-\frac{2}{5}a}\int^{\tau}_{0}\mathcal{D}_{2}(\tau)ds.
\end{align}
Here we used $F=M+\overline{G}+\sqrt{\mu}f$, \eqref{3.2}, \eqref{3.72}, \eqref{3.3}, \eqref{3.4} and the Sobolev imbedding theorem to get
\begin{align*}
\epsilon^{2(1-a)}\sum_{|\alpha|=2}\|\frac{\partial^{\alpha}F(0,y,v)}{\sqrt{\mu}}\|^{2}
\leq Ck^{\frac{1}{3}}\epsilon^{\frac{6}{5}-\frac{4}{5}a}
\end{align*}
and
\begin{align*}
\epsilon^{2(1-a)}\sum_{|\alpha|=2}\|\frac{\partial^{\alpha}F(\tau,y,v)}{\sqrt{\mu}}\|^{2}
\geq c\epsilon^{2(1-a)}(\|\partial^{\alpha}(\widetilde{\rho},\widetilde{u},\widetilde{\theta})\|^{2}
+\|\partial^{\alpha}f\|^{2})
-Ck^{\frac{1}{3}}\epsilon^{\frac{6}{5}-\frac{4}{5}a}.
\end{align*}
By a suitable linear combination of \eqref{3.68} and \eqref{3.85}, we have by choosing $\eta$ and $\eta_{0}$ small enough that
\begin{align}
\label{3.88}
&\sum_{|\alpha|=1}(\|\partial^{\alpha}(\widetilde{\rho},\widetilde{u},\widetilde{\theta})\|^{2}+\|\partial^{\alpha}f\|^{2})
+\epsilon^{2(1-a)}\sum_{|\alpha|=2}(\|\partial^{\alpha}(\widetilde{\rho},\widetilde{u},\widetilde{\theta})\|^{2}+\|\partial^{\alpha}f\|^{2})
\nonumber\\
&\quad+\epsilon^{1-a}\sum_{|\alpha|=2}\int^{\tau}_{0}(\|\partial^{\alpha}(\widetilde{\rho},\widetilde{u},
\widetilde{\theta})\|^{2}+\|\partial^{\alpha}f\|_{\sigma}^{2})ds+\epsilon^{a-1}\sum_{|\alpha|=1}\int^{\tau}_{0}\|\partial^{\alpha}f\|_{\sigma}^{2}ds
\nonumber\\
&\leq Ck^{\frac{1}{3}}\epsilon^{\frac{6}{5}-\frac{4}{5}a}
+C(\eta_{0}+k^{\frac{1}{12}}\epsilon^{\frac{3}{5}-\frac{2}{5}a}+k^{\frac{1}{12}}\epsilon^{\frac{3}{5}a-\frac{2}{5}})\int^{\tau}_{0}\mathcal{D}_{2}(s)ds.
\end{align}
This completes the proof of the high order energy estimates.
\subsection{Weighted energy estimates}\label{sec3.3}
In this subsection, we will derive the weighted mixed derivative estimates of the function $f$
in order to close a priori estimates. To this end, by taking the derivative $\partial_{\beta}^{\alpha}$ to \eqref{2.11}
with $|\alpha|+|\beta|\leq 2$ and $|\beta|\geq1$, for $e_{1}=(1,0,0)$, one has
\begin{align}
\label{3.89}
&\partial_{\tau}\partial^{\alpha}_{\beta}f+v_{1}\partial_{y}\partial^{\alpha}_{\beta}f
+C^{\beta-e_{1}}_{\beta}\partial_{y}\partial_{\beta-e_{1}}^{\alpha}f
-\epsilon^{a-1}\partial^{\alpha}_{\beta}\mathcal{L}f
\nonumber\\
&=\epsilon^{a-1}\partial^{\alpha}_{\beta}\Gamma(f,\frac{M-\mu}{\sqrt{\mu}})+
\epsilon^{a-1}\partial^{\alpha}_{\beta}\Gamma(\frac{M-\mu}{\sqrt{\mu}},f)
+\epsilon^{a-1}\partial^{\alpha}_{\beta}\Gamma(\frac{G}{\sqrt{\mu}},\frac{G}{\sqrt{\mu}})
+\partial^{\alpha}_{\beta}\big\{\frac{P_{0}(v_{1}\sqrt{\mu}\partial_{y}f)}{\sqrt{\mu}}\big\}
\nonumber\\
&\quad-\partial^{\alpha}_{\beta}\big\{\frac{1}{\sqrt{\mu}}P_{1}v_{1}M(\frac{|v-u|^{2}
\widetilde{\theta}_{y}}{2R\theta^{2}}+\frac{(v-u)\cdot\widetilde{u}_{y}}{R\theta})\big\}
-\partial^{\alpha}_{\beta}\big\{\frac{P_{1}(v_{1}\partial_{y}\overline{G})}{\sqrt{\mu}}\big\}
-\partial^{\alpha}_{\beta}\big\{\frac{\partial_{\tau}\overline{G}}{\sqrt{\mu}}\big\}.
\end{align}
We take the inner product of \eqref{3.89} with $w^{2|\beta|}\partial^{\alpha}_{\beta}f$ over $\mathbb{R}_{y}\times\mathbb{R}_{v}^{3}$
and estimate each term. First of all, by integration by parts, we obtain
\begin{align*}
(\partial_{\tau}&\partial^{\alpha}_{\beta}f+v_{1}\partial_{y}\partial^{\alpha}_{\beta}f,w^{2|\beta|}\partial_{\beta}^{\alpha}f)
=\frac{1}{2}\frac{d}{d\tau}\|\partial_{\beta}^{\alpha}f\|^{2}_{2,|\beta|}.
\end{align*}
From the H\"older inequality and Cauchy inequality, it follows that
\begin{align*}
|(\partial_{y}\partial_{\beta-e_{1}}^{\alpha}f,w^{2|\beta|}\partial_{\beta}^{\alpha}f)|
&\leq C\|w^{\frac{1}{2}+(|\beta|-1)}
\partial_{y}\partial_{\beta-e_{1}}^{\alpha}f\|
\|w^{|\beta|+\frac{1}{2}}\partial^{\alpha}_{\beta}f\|
\nonumber\\
&= C\|w^{\frac{1}{2}}w^{|\beta-e_{1}|}
\partial_{y}\partial_{\beta-e_{1}}^{\alpha}f\|
\|w^{\frac{1}{2}}w^{|\beta|}\partial^{\alpha}_{\beta}f\|
\nonumber\\
&\leq \eta\epsilon^{a-1}\|\partial^{\alpha}_{\beta}f\|_{\sigma,|\beta|}^{2}
+ C_{\eta}\epsilon^{1-a}\|\partial_{y}\partial_{\beta-e_{1}}^{\alpha}f\|_{\sigma,|\beta-e_{1}|}^{2}.
\end{align*}
Here  we have used the fact that $|\beta-e_{1}|=|\beta|-1$ and
$\|w^{\frac{1}{2}}w^{|\beta|}\partial^{\alpha}_{\beta}f\|\leq C\|\partial^{\alpha}_{\beta}f\|_{\sigma,|\beta|}$
for $w=\langle v\rangle^{\gamma+2}$ by \eqref{2.17}.
Due to \eqref{5.6}, we can see that
\begin{align*}
-\epsilon^{a-1}(\partial_{\beta}^{\alpha}\mathcal{L}f,w^{2|\beta|}\partial_{\beta}^{\alpha}f)\geq
\epsilon^{a-1}\|\partial^{\alpha}_{\beta}f\|^{2}_{\sigma,|\beta|}-\eta\epsilon^{a-1}\sum_{|\beta_{1}|\leq|\beta|}
\|\partial^{\alpha}_{\beta_{1}}f\|_{\sigma,|\beta_{1}|}^{2}-C_{\eta}\epsilon^{a-1}\|\partial^{\alpha}f\|_{\sigma}^{2}.
\end{align*}
With \eqref{5.9} and \eqref{5.19} in hand, one can show that
\begin{align*}
\epsilon^{a-1}|&(\partial^{\alpha}_{\beta}\Gamma(\frac{M-\mu}{\sqrt{\mu}},f),
w^{2|\beta|}\partial^{\alpha}_{\beta}f)|+\epsilon^{a-1}|(\partial^{\alpha}_{\beta}\Gamma(f,\frac{M-\mu}{\sqrt{\mu}}),
w^{2|\beta|}\partial^{\alpha}_{\beta}f)|
\nonumber\\
&\leq C\eta\epsilon^{a-1}\|\partial^{\alpha}_{\beta}f\|^{2}_{\sigma,|\beta|}
+C_{\eta}(\eta_{0}+k^{\frac{1}{12}}\epsilon^{\frac{3}{5}-\frac{2}{5}a})\mathcal{D}_{2}(\tau),
\end{align*}
and
\begin{align*}
&\epsilon^{a-1}|(\partial^{\alpha}_{\beta}\Gamma(\frac{G}{\sqrt{\mu}},\frac{G}{\sqrt{\mu}}),
w^{2|\beta|}\partial^{\alpha}_{\beta}f)|\\
&\leq  C\eta\epsilon^{a-1}\|\partial^{\alpha}_{\beta}f\|^{2}_{\sigma,|\beta|}
+C_{\eta}\epsilon^{\frac{7}{5}+\frac{1}{15}a}(\delta+\epsilon^{a}\tau)^{-\frac{4}{3}}
+C_{\eta}k^{\frac{1}{12}}\epsilon^{\frac{3}{5}-\frac{2}{5}a}\mathcal{D}_{2}(\tau).
\end{align*}
By using \eqref{1.9}, \eqref{1.26}, \eqref{2.17}, the Sobolev imbedding theorem, Lemma \ref{lem5.3}, \eqref{3.3} and \eqref{3.4}, we can obtain
\begin{align*}
&\big|(\partial_{\beta}^{\alpha}(\frac{P_{0}(v_{1}\sqrt{\mu}\partial_{y}f)}{\sqrt{\mu}})
,w^{2|\beta|}\partial^{\alpha}_{\beta}f)\big|=\big|\sum_{j=0}^{4}(\langle v\rangle^{\frac{1}{2}}w^{|\beta|}\partial_{\beta}^{\alpha}
(\langle v_{1}\sqrt{\mu}\partial_{y}f,\frac{\chi_{j}}{M}\rangle\frac{\chi_{j}}{\sqrt{\mu}})
,\langle v\rangle^{-\frac{1}{2}}w^{|\beta|}\partial^{\alpha}_{\beta}f)\big|
\\
&\leq \eta\epsilon^{a-1}\|\partial^{\alpha}_{\beta}f\|_{\sigma,|\beta|}^{2}+C_{\eta}\epsilon^{1-a}\|\partial^{\alpha}\partial_{y}f\|_{\sigma}^{2}
+C_{\eta}\epsilon^{\frac{7}{5}+\frac{1}{15}a}(\delta+\epsilon^{a}\tau)^{-\frac{4}{3}}
+C_{\eta}k^{\frac{1}{12}}\epsilon^{\frac{3}{5}-\frac{2}{5}a}\mathcal{D}_{2}(\tau),
\end{align*}
and
\begin{align*}
&\big|(\partial_{\beta}^{\alpha}(\frac{1}{\sqrt{\mu}}P_{1}v_{1}M\big\{\frac{|v-u|^{2}
\widetilde{\theta}_{y}}{2R\theta^{2}}+\frac{(v-u)\cdot\widetilde{u}_{y}}{R\theta}\big\})
,w^{2|\beta|}\partial^{\alpha}_{\beta}f)\big|
\\
&\leq C\|\langle v\rangle^{\frac{1}{2}}w^{|\beta|}\partial_{\beta}^{\alpha}\big\{\frac{1}{\sqrt{\mu}}P_{1}v_{1}M(\frac{|v-u|^{2}
\widetilde{\theta}_{y}}{2R\theta^{2}}+\frac{(v-u)\cdot\widetilde{u}_{y}}{R\theta})\big\}\|
\|\langle v\rangle^{-\frac{1}{2}}w^{|\beta|}\partial^{\alpha}_{\beta}f\|
\\
&\leq \eta\epsilon^{a-1}\|\partial^{\alpha}_{\beta}f\|_{\sigma,|\beta|}^{2}
+C_{\eta}\epsilon^{1-a}\|\partial^{\alpha}(\widetilde{u}_{y},\widetilde{\theta}_{y})\|^{2}
+C_{\eta}\epsilon^{\frac{7}{5}+\frac{1}{15}a}(\delta+\epsilon^{a}\tau)^{-\frac{4}{3}}
+C_{\eta}k^{\frac{1}{12}}\epsilon^{\frac{3}{5}-\frac{2}{5}a}\mathcal{D}_{2}(\tau),
\end{align*}
where we have used the fact that $|\langle v\rangle^{m}\mu^{-\frac{1}{2}}\partial_{\beta}M|_{2}\leq C$ for any $m\geq0$ and $\beta\geq0$
by \eqref{1.26}.
\\
Notice that the last two terms of \eqref{3.89} can be dominated by
\begin{align*}
&\big|\big(\partial_{\beta}^{\alpha}(\frac{P_{1}(v_{1}\partial_{y}\overline{G})}{\sqrt{\mu}})
-\partial_{\beta}^{\alpha}(\frac{\partial_{\tau}\overline{G}}{\sqrt{\mu}}),w^{2|\beta|}\partial^{\alpha}_{\beta}f\big)\big|
\nonumber\\
&\leq C\big(\|\langle v\rangle^{\frac{1}{2}}w^{|\beta|}\partial_{\beta}^{\alpha}(\frac{P_{1}(v_{1}\partial_{y}\overline{G})}{\sqrt{\mu}})\|
+\|\langle v\rangle^{\frac{1}{2}}w^{|\beta|}\partial_{\beta}^{\alpha}(\frac{\partial_{\tau}\overline{G}}{\sqrt{\mu}})\|\big)
\|\langle v\rangle^{-\frac{1}{2}}w^{|\beta|}\partial^{\alpha}_{\beta}f\|
\nonumber\\
&\leq\eta\epsilon^{a-1}\|\partial^{\alpha}_{\beta}f\|_{\sigma,|\beta|}^{2}
+C_{\eta}\epsilon^{\frac{7}{5}+\frac{1}{15}a}(\delta+\epsilon^{a}\tau)^{-\frac{4}{3}}
+C_{\eta}k^{\frac{1}{12}}\epsilon^{\frac{3}{5}-\frac{2}{5}a}\mathcal{D}_{2}(\tau),
\end{align*}
according to \eqref{1.9}, \eqref{5.26}, \eqref{2.17}, Lemma \ref{lem5.3} and the elementary inequalities. Hence, for $|\alpha|+|\beta|\leq 2$ and $|\beta|\geq1$, we have by the above related estimates and $\eta>0$ small enough that
\begin{align}
\label{3.90}
&\frac{1}{2}\frac{d}{d\tau}\|\partial_{\beta}^{\alpha}f\|^{2}_{2,|\beta|}
+c\epsilon^{a-1}\|\partial^{\alpha}_{\beta}f\|^{2}_{\sigma,|\beta|}\notag\\
&\leq C\epsilon^{a-1}\|\partial^{\alpha}f\|_{\sigma}^{2}
+C\epsilon^{1-a}\big\{\|\partial^{\alpha}\partial_{y}f\|_{\sigma}^{2}+\|\partial^{\alpha}(\widetilde{u}_{y},\widetilde{\theta}_{y})\|^{2}\big\}
\nonumber\\
&\quad+C\epsilon^{1-a}\|\partial_{y}\partial_{\beta-e_{1}}^{\alpha}f\|_{\sigma,|\beta-e_{1}|}^{2}
+C\epsilon^{a-1}\sum_{|\beta_{1}|<|\beta|}\|\partial^{\alpha}_{\beta_{1}}f\|_{\sigma,|\beta_{1}|}^{2}
\nonumber\\
&\quad+C\epsilon^{\frac{7}{5}+\frac{1}{15}a}(\delta+\epsilon^{a}\tau)^{-\frac{4}{3}}
+C(\eta_{0}+k^{\frac{1}{12}}\epsilon^{\frac{3}{5}-\frac{2}{5}a})\mathcal{D}_{2}(\tau).
\end{align}
\par
Notice that the coefficients on the third line of \eqref{3.90} is large and $|\beta_{1}|<|\beta|$. We will use
the induction in $|\beta|$  to control this term. By the suitable linear combinations, we can obtain
\begin{align}
\label{3.91}
&\sum_{|\alpha|=0}^{1}C_{|\alpha|}\sum_{j=1}^{2-|\alpha|}C_j\sum_{|\beta|=j}\frac{d}{d\tau}\|\partial_{\beta}^{\alpha}f\|_{2,|\beta|}^{2}
+c\epsilon^{a-1}\sum_{|\alpha|+|\beta|\leq 2,|\beta|\geq1}\|\partial_{\beta}^{\alpha}f\|_{\sigma,|\beta|}^{2}
\nonumber\\
&\leq C\epsilon^{a-1}\sum_{|\alpha|\leq 1}\|\partial^{\alpha}f\|_{\sigma}^{2}+C\epsilon^{1-a}\sum_{1\leq|\alpha|\leq 2}\big\{\|\partial^{\alpha}f\|_{\sigma}^{2}+\|\partial^{\alpha}(\widetilde{\rho},\widetilde{u},\widetilde{\theta})\|^{2}\big\}
\nonumber\\
&\quad+C\epsilon^{\frac{7}{5}+\frac{1}{15}a}(\delta+\epsilon^{a}\tau)^{-\frac{4}{3}}
+C(\eta_{0}+k^{\frac{1}{12}}\epsilon^{\frac{3}{5}-\frac{2}{5}a})\mathcal{D}_{2}(\tau).
\end{align}
Here we have required that
\begin{align}
\label{3.92}
\epsilon^{a-1}\geq \epsilon^{1-a}, \quad \mbox{that is}\quad  a\leq 1.
\end{align}
Integrating \eqref{3.91} with respect to $\tau$ and using \eqref{3.2} with \eqref{3.3} gives
\begin{align}
\label{3.93}
&\sum_{|\alpha|+|\beta|\leq2,|\beta|\geq1}\big\{\|\partial_{\beta}^{\alpha}f\|_{2,|\beta|}^{2}
+\epsilon^{a-1}\int^{\tau}_{0}\|\partial_{\beta}^{\alpha}f\|_{\sigma,|\beta|}^{2}ds\big\}
\nonumber\\
&\leq C\epsilon^{a-1}\sum_{|\alpha|\leq1}\int^{\tau}_{0}\|\partial^{\alpha}f\|_{\sigma}^{2}ds
+C\epsilon^{1-a}\sum_{1\leq|\alpha|\leq 2}\int^{\tau}_{0}
\big\{\|\partial^{\alpha}f\|_{\sigma}^{2}+\|\partial^{\alpha}(\widetilde{\rho},\widetilde{u},\widetilde{\theta})\|^{2}\big\}ds
\nonumber\\
&\quad+Ck^{\frac{1}{3}}\epsilon^{\frac{6}{5}-\frac{4}{5}a}
+C(\eta_{0}+k^{\frac{1}{12}}\epsilon^{\frac{3}{5}-\frac{2}{5}a})\int^{\tau}_{0}\mathcal{D}_{2}(s)ds.
\end{align}
This completes the proof of the weighted derivative estimates of the function $f$.

\section{Stability and convergence rate}\label{sec.4}
Based on the energy estimates derived in Section \ref{sec.3}, in this section we are now in a position to complete the

\medskip
\noindent{\it Proof of Theorem \ref{thm1.1}:}
By a suitable linear combination of \eqref{3.41}, \eqref{3.88} and \eqref{3.93}, we can obtain
\begin{align}
\label{4.1}
&\mathcal{E}_{2}(\tau)+\int^{\tau}_{0}\|\sqrt{\bar{u}_{1y}}(\widetilde{\rho},\widetilde{u}_{1},\widetilde{\theta})\|^{2}ds
+\int^{\tau}_{0}\mathcal{D}_{2}(s)ds
\nonumber\\
&\hspace{1cm}\leq Ck^{\frac{1}{3}}\epsilon^{\frac{6}{5}-\frac{4}{5}a}
+C(\eta_{0}+k^{\frac{1}{12}}\epsilon^{\frac{3}{5}-\frac{2}{5}a}+k^{\frac{1}{12}}\epsilon^{\frac{3}{5}a-\frac{2}{5}})\int^{\tau}_{0}\mathcal{D}_{2}(s)ds.
\end{align}
Here $\mathcal{E}_{2}(\tau)$ and $\mathcal{D}_{2}(\tau)$ are defined by \eqref{2.18} and \eqref{2.19}, respectively. At the moment,  one has to require that the second term on the right hand side of \eqref{4.1} should be absorbed  by the left hand side. Thus, this leads us to impose
\begin{align}
\label{4.2}
\frac{3}{5}-\frac{2}{5}a\geq 0 \quad \mbox{and} \quad \frac{3}{5}a-\frac{2}{5}\geq0,
\quad \mbox{that is} \quad \frac{2}{3}\leq a\leq\frac{3}{2}.
\end{align}
Due to \eqref{3.92} and \eqref{4.2}, we need to require that
$\frac{2}{3}\leq a\leq 1$ for the choice of the parameter $a$ in the scaling transformation \eqref{2.1} which we start with. Hence, by using the smallness of $k>0$, $\eta_{0}>0$ and $\epsilon>0$,
we have from \eqref{4.1} that
\begin{align}
\label{4.3}
\mathcal{E}_{2}(\tau)+\int^{\tau}_{0}\|\sqrt{\bar{u}_{1y}}(\widetilde{\rho},\widetilde{u}_{1},\widetilde{\theta})\|^{2}ds
+\frac{1}{2}\int^{\tau}_{0}\mathcal{D}_{2}(s)ds\leq Ck^{\frac{1}{3}}\epsilon^{\frac{6}{5}-\frac{4}{5}a}<
\frac{1}{2}k^{\frac{1}{6}}\epsilon^{\frac{6}{5}-\frac{4}{5}a}.
\end{align}
Then \eqref{4.3} implies that  for $a\in[\frac{2}{3},1]$ and $\tau_{1}\in(0,+\infty)$, one has
\begin{equation}
\label{4.4}
\sup_{0\leq\tau\leq\tau_{1}}\mathcal{E}_{2}(\tau)< \frac{1}{2}k^{\frac{1}{6}}\epsilon^{\frac{6}{5}-\frac{4}{5}a},
\end{equation}
which is strictly stronger than \eqref{3.4}. Thus the a priori assumption \eqref{3.4} can be closed. Therefore, by the uniform a priori estimates  and the local existence of the solution,
the standard continuity argument gives the existence and uniqueness of global solutions
to the Landau equation \eqref{1.1} with initial data \eqref{3.1}. Moreover, the desired estimate \eqref{thm.enineq} holds true.

We are going to justify the convergence rate as in \eqref{1.27}. By \eqref{4.4}, \eqref{2.18} and the Sobolev imbedding theorem, we get
\begin{equation}
\label{4.5}
\sup_{0\leq\tau\leq+\infty}\{\|(\widetilde{\rho},\widetilde{u},\widetilde{\theta})(\tau,y)\|_{L_{y}^{\infty}}
+\|f(\tau,y,v)\|_{L^{\infty}_{y}L^{2}_{v}}\}\leq Ck^{\frac{1}{12}}\epsilon^{\frac{3}{5}-\frac{2}{5}a}.
\end{equation}
On the other hand, we have by using \eqref{5.24}, Lemma \ref{lem5.3} and  $\delta=\frac{1}{k}\epsilon^{\frac{3}{5}-\frac{2}{5}a}$ with $\frac{2}{3}\leq a\leq 1$ that
\begin{equation}
\label{4.6}
\sup_{0\leq\tau\leq+\infty}\|\frac{\overline{G}(\tau,y,v)}{\sqrt{\mu}}\|_{L_{y}^{\infty}L_{v}^{2}}
\leq C\epsilon^{1-a}\sup_{0\leq\tau\leq+\infty}(\|\bar{u}_{1y}\|_{L^{\infty}_{y}}+\|\bar{\theta}_{y}\|_{L^{\infty}_{y}})
\leq Ck^{\frac{1}{12}}\epsilon^{\frac{3}{5}-\frac{2}{5}a}.
\end{equation}
It follows from \eqref{4.5} and \eqref{4.6} that
\begin{align}
\label{4.7}
\sup_{0\leq\tau\leq+\infty}\|\frac{F-M_{[\bar{\rho},\bar{u},\bar{\theta}]}}{\sqrt{\mu}}\|_{L_{x}^{\infty}L_{v}^{2}}
&\leq C\sup_{0\leq\tau\leq+\infty}\Big\{\|\frac{M-M_{[\bar{\rho},\bar{u},\bar{\theta}]}}{\sqrt{\mu}}\|_{L_{y}^{\infty}L_{v}^{2}}
+\|f\|_{L_{y}^{\infty}L_{v}^{2}}+\|\frac{\overline{G}}{\sqrt{\mu}}\|_{L_{y}^{\infty}L_{v}^{2}}
\Big\}\notag\\
&\leq Ck^{\frac{1}{12}}\epsilon^{\frac{3}{5}-\frac{2}{5}a},
\end{align}
where we have used the facts that $F=M+\overline{G}+\sqrt{\mu}f$ and \eqref{1.26}.
By Lemma \ref{lem5.2} and $\delta=\frac{1}{k}\epsilon^{\frac{3}{5}-\frac{2}{5}a}$ with $k$ independent of $\epsilon$ satisfying $\epsilon\ll k$,
we have for $t>0$ that
\begin{align}
\label{4.8}
\|(\bar{\rho},\bar{u},\bar{\theta})(t,x)-(\rho^{R},u^{R},\theta^{R})(\frac{x}{t})\|_{L^{\infty}_{x}}\leq
C \frac{1}{k}t^{-1}\epsilon^{\frac{3}{5}-\frac{2}{5}a}\{\ln(1+t)+|\ln \epsilon|\}.
\end{align}
With \eqref{4.7} and \eqref{4.8} in hand,
for any given constant $l>0$ and all $t\in[l,+\infty)$,
there exists a constant $C_{l,k}>0$, independent of $\epsilon$,   such that
\begin{align*}
\|\frac{F(t,x,v)-M_{[\rho^{R},u^{R},\theta^{R}](x/t)}(v)}{\sqrt{\mu}}\|_{L_{x}^{\infty}L_{v}^{2}}
&\leq \|\frac{F-M_{[\bar{\rho},\bar{u},\bar{\theta}]}}{\sqrt{\mu}}\|_{L_{x}^{\infty}L_{v}^{2}}
+\|\frac{M_{[\bar{\rho},\bar{u},\bar{\theta}]}-M_{[\rho^{R},u^{R},\theta^{R}]}}{\sqrt{\mu}}\|_{L_{x}^{\infty}L_{v}^{2}}\notag\\
&\leq C_{l,k}\epsilon^{\frac{3}{5}-\frac{2}{5}a}|\ln\epsilon|.
\end{align*}
This gives \eqref{1.27} and then completes the proof of Theorem \ref{thm1.1}.\qed

\section{Appendix}\label{sec.5}
In this section, we will give some basic estimates, which have been used in the previous energy estimates.
We first list some properties for the rarefaction
wave defined by \eqref{1.21} and \eqref{1.23}. Then, we give some properties of the Burnett functions and the fast velocity decay of $\overline{G}$
to overcome the slow time decay of the term $\|(\bar{u}_{y},\bar{\theta}_{y})\|^{2}$ in the term
$P_{1}(v_{1}M_{y})$ in \eqref{2.10}. Lastly, we recall some basic properties of the collision operators and prove
some linear and nonlinear estimates in the previous energy analysis.
\par
We now give the properties of the solution $\overline{\omega}_{\delta}(t,x)$ to Burgers equation \eqref{1.21}
and the smooth approximate 3-rarefaction wave $(\bar{\rho},\bar{u},\bar{\theta})(t,x)$ constructed by \eqref{1.23}. Their proofs can be found in \cite{HL1,LiuXin,Xin}.

\begin{lemma}\label{lem5.1}
The Burgers equation \eqref{1.21} has a unique smooth global solution $\overline{\omega}_{\delta}(t,x)$ such that
\\
\mbox{(1)}~~$\omega_{-}<\overline{\omega}_{\delta}(t,x)<\omega_{+}$,\ \ $\partial_{x}\overline{\omega}_{\delta}(t,x)>0$,\quad
$\forall~ x\in\mathbb{R}$,~$t\geq 0$.
\\
\mbox{(2)}~~The following estimates hold for any $t>0$, $\delta>0$ and $p\in[1,+\infty]$
\begin{align*}
&\|\partial_{x}\overline{\omega}_{\delta }(t,x)\|_{L^{p}(\mathbb{R}_{x})}\leq C(\omega_{+}-\omega_{-})^{\frac{1}{p}}(\delta+t)^{-1+\frac{1}{p}},
\nonumber\\
&\|\partial^{j}_{x}\overline{\omega}_{\delta}(t,x)\|_{L^{p}(\mathbb{R}_{x})}\leq C\delta^{-j+1+\frac{1}{p}}(\delta+t)^{-1},\quad j\geq 2.
\end{align*}
\mbox{(3)}~~There exists a constant $\delta_{0}\in(0,1)$ such that for $\delta\in(0,\delta_{0})$ and $t>0$
$$
\|\overline{\omega}_{\delta}(t,x)-\omega^{R}(\frac{x}{t})\|_{L^{\infty}(\mathbb{R}_{x})}\leq C\delta t^{-1}\{\ln(1+t)+|\ln\delta|\}.
$$
\end{lemma}
\begin{lemma}\label{lem5.2}
The smooth approximate 3-rarefaction wave $(\bar{\rho},\bar{u},\bar{\theta})(t,x)$ defined in \eqref{1.23} satisfying
\\
\mbox{(i)}~~$\bar{u}_{2}=\bar{u}_{3}=0$,\quad $\bar{u}_{1x}>0$, \mbox{and $\bar{\theta}_{x}=\sqrt{\frac{2}{5}}\bar{\theta}^{\frac{1}{2}}\bar{u}_{1x}$},\quad
$\forall ~x\in\mathbb{R}$,~$t\geq 0$.
\\
\mbox{(ii)}~~The following estimates hold for any $t> 0$, $\delta>0$ and $p\in[1,+\infty]$
\begin{align*}
&\|\partial_{x}(\bar{\rho},\bar{u}_{1},\bar{\theta})(t,x)\|_{L^{p}(\mathbb{R}_{x})}\leq C(\omega_{+}-\omega_{-})^{\frac{1}{p}}(\delta+t)^{-1+\frac{1}{p}},
\nonumber\\
&\|\partial_{x}^{j}(\bar{\rho},\bar{u}_{1},\bar{\theta})(t,x)\|_{L^{p}(\mathbb{R}_{x})}\leq C\delta^{-j+1+\frac{1}{p}}(\delta+t)^{-1},\quad j\geq 2.
\end{align*}
\mbox{(iii)}~~There exists a constant $\delta_{0}\in(0,1)$ such that for $\delta\in(0,\delta_{0})$ and $t>0$
$$
\|(\bar{\rho},\bar{u},\bar{\theta})(t,x)-(\rho^{R},u^{R},\theta^{R})(\frac{x}{t})\|_{L^{\infty}(\mathbb{R}_{x})}\leq C\delta t^{-1}\{\ln(1+t)+|\ln\delta|\}.
$$
\end{lemma}

Since the scaling transformation $y=\epsilon^{-a}x$ and $\tau=\epsilon^{-a}t$ is considered through the proof, the following lemma is equivalent to Lemma \ref{lem5.2} (ii),
which will be used frequently in the previous energy estimates.
\begin{lemma}\label{lem5.3}
The smooth approximate 3-rarefaction wave $(\bar{\rho},\bar{u},\bar{\theta})(t,x)$ defined in \eqref{1.23} satisfying
\begin{align*}
&\|\partial_{y}(\bar{\rho},\bar{u}_{1},\bar{\theta})(\epsilon^{a}\tau,\epsilon^{a}y)\|_{L^{p}(\mathbb{R}_{y})}\leq C
\epsilon^{a(1-\frac{1}{p})}(\delta+\epsilon^{a}\tau)^{-1+\frac{1}{p}},
\nonumber\\
&\|\partial_{y}^{j}(\bar{\rho},\bar{u}_{1},\bar{\theta})(\epsilon^{a}\tau,\epsilon^{a}y)\|_{L^{p}(\mathbb{R}_{y})}\leq C\epsilon^{a(j-\frac{1}{p})}\delta^{-j+1+\frac{1}{p}}(\delta+\epsilon^{a}\tau)^{-1},\quad j\geq 2,
\end{align*}
for any $\tau> 0$, $\delta>0$ and $p\in[1,+\infty]$.
\end{lemma}
We remark that  the temporal derivatives of $(\bar{\rho},\bar{u}_{1},\bar{\theta})(t,x)$ in Lemma
\ref{5.2} (ii) and \ref{5.3} obviously hold in terms of Euler system \eqref{1.24} and the elementary inequalities.

Recall the Burnett functions, cf.~\cite{BGL1,BGL2,C2,Guo,XZ1}, defined as
\begin{equation}
\label{5.1}
\hat{A}_{j}(v)=\frac{|v|^{2}-5}{2}v_{j}\quad \mbox{and} \quad \hat{B}_{ij}(v)=v_{i}v_{j}-\frac{1}{3}\delta_{ij}|v|^{2} \quad \mbox{for} \quad i,j=1,2,3.
\end{equation}
Noting that $\hat{A}_{j}M$ and $\hat{B}_{ij}M$ are orthogonal to the null space $\mathcal{N}$, we can define
functions $A_{j}(v)$ and $B_{ij}(v)$ such that $P_{0}A_{j}=0$,
$P_{0}B_{ij}=0$ and
\begin{equation}
\label{5.2}
A_{j}(\frac{v-u}{\sqrt{R\theta}})=L^{-1}_{M}[\hat{A}_{j}(\frac{v-u}{\sqrt{R\theta}})M]\quad
\mbox{and} \quad B_{ij}(\frac{v-u}{\sqrt{R\theta}})=L^{-1}_{M}[\hat{B}_{ij}(\frac{v-u}{\sqrt{R\theta}})M].
\end{equation}
The following lemma is borrowed from \cite[Lemma 6.2]{DuanY}. Readers also refer to \cite{BGL1,BGL2,Guo,XZ1}.

\begin{lemma}
The Burnett functions have the following properties:
\begin{itemize}
\item{$-\langle \hat{A}_{i}, A_{i}\rangle$ ~~is positive and independent of i;}
\item{$\langle \hat{A}_{i}, A_{j}\rangle=0$ ~~for ~any ~$i\neq j$;\quad $\langle
     \hat{A}_{i}, B_{jk}\rangle=0$~~for ~any ~i,~j,~k;}
\item{$\langle\hat{B}_{ij},B_{kj}\rangle=\langle\hat{B}_{kl},B_{ij}\rangle=\langle\hat{B}_{ji},B_{kj}\rangle$,~~
      which is independent of ~i,~j, for fixed~~k,~l;}
\item{$-\langle \hat{B}_{ij}, B_{ij}\rangle$ ~~is positive and independent of i,~j when $i\neq j$;}
\item{$\langle \hat{B}_{ii}, B_{jj}\rangle$ ~~is positive and independent of i,~j when $i\neq j$;}
\item{$-\langle \hat{B}_{ii}, B_{ii}\rangle$ ~~is positive and independent of i;}
\item{$\langle \hat{B}_{ij}, B_{kl}\rangle=0$ ~~unless~either~$(i,j)=(k,l)$~or~$(l,k)$,~or~i=j~and~k=l;}
\item{$\langle \hat{B}_{ii}, B_{ii}\rangle-\langle \hat{B}_{ii}, B_{jj}\rangle=2\langle \hat{B}_{ij},
      B_{ij}\rangle$ ~~holds for any~ $i\neq j$.}
\end{itemize}
\end{lemma}
In terms of the properties of Burnett functions, the viscosity coefficient $\mu(\theta)$ and heat conductivity
coefficient $\kappa(\theta)$ can be represented by
\begin{align}
\label{5.3}
\mu(\theta)=&- R\theta\int_{\mathbb{R}^{3}}\hat{B}_{ij}(\frac{v-u}{\sqrt{R\theta}})
B_{ij}(\frac{v-u}{\sqrt{R\theta}})dv>0,\quad i\neq j,
\nonumber\\
\kappa(\theta)=&-R^{2}\theta\int_{\mathbb{R}^{3}}\hat{A}_{j}(\frac{v-u}{\sqrt{R\theta}})
A_{j}(\frac{v-u}{\sqrt{R\theta}})dv>0.
\end{align}
Notice that these coefficients are positive smooth functions depending only on $\theta$.

The following lemma is borrowed from \cite[Lemma 6.1]{DuanY}, which is about
the fast velocity decay of the Burnett functions.

\begin{lemma}
Suppose that $U(v)$ is any polynomial of $\frac{v-\hat{u}}{\sqrt{R}\hat{\theta}}$ such that
$U(v)\widehat{M}\in(\ker{L_{\widehat{M}}})^{\perp}$ for any Maxwellian $\widehat{M}=M_{[\widehat{\rho},\widehat{u},\widehat{\theta}]}(v)$ where $L_{\widehat{M}}$ is as \eqref{1.14}.
For any $\varepsilon\in(0,1)$ and any multi-index $\beta$, there exists constant $C_{\beta}>0$ such that
$$
|\partial_{\beta}L^{-1}_{\widehat{M}}(U(v)\widehat{M})|\leq C_{\beta}(\widehat{\rho},\widehat{u},\widehat{\theta})\widehat{M}^{1-\varepsilon}.
$$
In particular, if the assumptions of \eqref{1.26} hold, there exists constant $C_{\beta}>0$ such that
\begin{equation}
\label{5.4}
|\partial_{\beta}A_{j}(\frac{v-u}{\sqrt{R\theta}})|+|\partial_{\beta}B_{ij}(\frac{v-u}{\sqrt{R\theta}})|
\leq C_{\beta}M^{1-\varepsilon}.
\end{equation}
\end{lemma}
Now, we shall turn to summarize some refined estimates for the collision operators $\mathcal{L}$  and $\Gamma$ defined as \eqref{2.8}.
We first recall the properties of the linearized operators $\mathcal{L}$.
Note that the null space $\mathcal{N}_{1}$ of $\mathcal{L}$ is spanned by the functions $\{\sqrt{\mu},v\sqrt{\mu},|v|^{2}\sqrt{\mu}\}$
in \cite{DL,G1}. Moreover, for any $g\in \mathcal{N}^{\perp}_{1}$, there exists $\sigma_{1}>0$ such that
\begin{equation}
\label{5.5}
-\langle\mathcal{L}g, g \rangle\geq \sigma_{1}|g|^{2}_{\sigma}.
\end{equation}
In addition, the weighted coercivity estimates on the linearized operators $\mathcal{L}$ can be stated as follows.

\begin{lemma}
Let $|\beta|>0$  and $w$ defined in \eqref{2.14}. Then for any $\eta>0$ ,
there exists $C_{\eta}>0$ such that
\begin{equation}
\label{5.6}
-\langle \partial_{\beta}\mathcal{L}g,w^{2|\beta|}\partial_{\beta}g\rangle\geq
|\partial_{\beta}g|^{2}_{\sigma,|\beta|}-\eta\sum_{|\beta_{1}|\leq|\beta|}
|\partial_{\beta_{1}}g|_{\sigma,|\beta_{1}|}^{2}-C_{\eta}|g|_{\sigma}^{2}.
\end{equation}
\end{lemma}

\begin{proof}
The proof of \eqref{5.6} can be found in \cite{G1}  and we omit the proof here for brevity.
\end{proof}

In what follows we recall the weighted estimates on the nonlinear collision operators $\Gamma$.

\begin{lemma}
Let $w$ defined in \eqref{2.14} and $\ell\geq0$, for arbitrarily large constant $b>0$, one has
\begin{align}
\label{5.7}
|\langle \partial^{\alpha}\Gamma(g_{1},g_{2}), \partial^{\alpha}g_{3}\rangle|
\leq C\sum_{|\alpha_{1}|\leq|\alpha|}|\langle v\rangle^{-b}\partial^{\alpha_{1}}g_{1}|_{2}
|\partial^{\alpha-\alpha_{1}}g_{2}|_{\sigma}
|\partial^{\alpha}g_{3}|_{\sigma},
\end{align}
and
\begin{align}
\label{5.8}
|\langle \partial^{\alpha}_{\beta}\Gamma(g_{1},g_{2}), w^{2\ell}\partial^{\alpha}_{\beta}g_{3}\rangle|
\leq C\sum_{|\alpha_{1}|\leq|\alpha|}
\sum_{|\beta'|\leq|\beta_{1}|\leq|\beta|}|\langle v\rangle^{-b}\partial^{\alpha_{1}}_{\beta^{'}}g_{1}|_{2}
|\partial^{\alpha-\alpha_{1}}_{\beta-\beta_{1}}g_{2}|_{\sigma,\ell}
|\partial^{\alpha}_{\beta}g_{3}|_{\sigma,\ell}.
\end{align}
\end{lemma}
\begin{proof}
The proof of \eqref{5.7} and \eqref{5.8} can be found in \cite[Proposition 1]{SZ}.
\end{proof}
Finally, we prove some linear and nonlinear estimates, which are used in section \ref{sec.3}.
The first estimates involving the linear terms $\Gamma(\frac{M-\mu}{\sqrt{\mu}},f)$ and
$\Gamma(f,\frac{M-\mu}{\sqrt{\mu}})$.

\begin{lemma}
\label{lem5.8}
Let $|\alpha|+|\beta|\leq 2$ with $|\beta|\geq1$ and  $w$ defined in \eqref{2.14}.
Suppose that $\mathcal{E}_{2}(\tau)\leq k^{\frac{1}{6}}\epsilon^{\frac{6}{5}-\frac{4}{5}a}$
and $\delta=\frac{1}{k}\epsilon^{\frac{3}{5}-\frac{2}{5}a}$ for $a\in[\frac{2}{3},1]$ as well as the assumption \eqref{1.26} holds.
If we choose $\eta_{0}$ in \eqref{1.26} and $k$ in \eqref{3.3} small enough,
for any small $\eta>0$, we get
\begin{align}
\label{5.9}
&\epsilon^{a-1}|(\partial^{\alpha}_{\beta}\Gamma(\frac{M-\mu}{\sqrt{\mu}},f),
w^{2|\beta|}\partial^{\alpha}_{\beta}h)|+\epsilon^{a-1}|(\partial^{\alpha}_{\beta}\Gamma(f,\frac{M-\mu}{\sqrt{\mu}}),
w^{2|\beta|}\partial^{\alpha}_{\beta}h)|
\nonumber\\
&\hspace{1cm}\leq C\eta\epsilon^{a-1}\|\partial^{\alpha}_{\beta}h\|^{2}_{\sigma,|\beta|}
+C_{\eta}(\eta_{0}+k^{\frac{1}{12}}\epsilon^{\frac{3}{5}-\frac{2}{5}a})\mathcal{D}_{2}(\tau).
\end{align}
Moreover, for $|\alpha|\leq 1$, one has
\begin{align}
\label{5.10}
\epsilon^{a-1}&|(\partial^{\alpha}\Gamma(\frac{M-\mu}{\sqrt{\mu}},f),\partial^{\alpha}h)|
+\epsilon^{a-1}|(\partial^{\alpha}\Gamma(f,\frac{M-\mu}{\sqrt{\mu}}),\partial^{\alpha}h)|
\nonumber\\
&\leq C\eta\epsilon^{a-1}\|\partial^{\alpha}h\|^{2}_{\sigma}
+C_{\eta}(\eta_{0}+k^{\frac{1}{12}}\epsilon^{\frac{3}{5}-\frac{2}{5}a})\mathcal{D}_{2}(\tau).
\end{align}
\end{lemma}

\begin{proof}
We only consider the first term on the left-hand side of \eqref{5.9} and the second term of \eqref{5.9} can be handled in the same way.
Since  $|\alpha|+|\beta|\leq 2$ with $|\beta|\geq1$, one has $|\alpha|\leq1$.
First of all, we have from \eqref{5.8} that
\begin{align}
\label{5.11}
&\epsilon^{a-1}|(\partial^{\alpha}_{\beta}\Gamma(\frac{M-\mu}{\sqrt{\mu}},f),
w^{2|\beta|}\partial^{\alpha}_{\beta}h)|
\nonumber\\
&\leq C\epsilon^{a-1}\sum_{|\alpha_{1}|\leq|\alpha|}
\sum_{|\beta'|\leq|\beta_{1}|\leq|\beta|}\int_{\mathbb{R}}|\langle v\rangle^{-b}\partial^{\alpha_{1}}_{\beta^{'}}(\frac{M-\mu}{\sqrt{\mu}})|_{2}
|\partial^{\alpha-\alpha_{1}}_{\beta-\beta_{1}}f|_{\sigma,|\beta-\beta_{1}|}
|\partial^{\alpha}_{\beta}h|_{\sigma,|\beta|}dy,
\end{align}
due to the fact that $w^{2|\beta|}\leq w^{2|\beta-\beta_{1}|}$ for $|\beta-\beta_{1}|\leq|\beta|$.
For any $\bar{\beta}\geq0$ and $m>0$, there exists a small constant $\varepsilon_{1}>0$ such that
\begin{align*}
|\langle v\rangle^{m}\partial_{\bar{\beta}}(\frac{M-\mu}{\sqrt{\mu}})|^{2}_{\sigma}
+|\langle v\rangle^{m}\partial_{\bar{\beta}}(\frac{M-\mu}{\sqrt{\mu}})|^{2}_{2}\leq C_{m}\sum_{|\bar{\beta}|\leq|\beta'|\leq|\bar{\beta}|+1}
\int_{\mathbb{R}^{3}}\mu^{-\varepsilon_{1}}|\partial_{\beta'}(\frac{M-\mu}{\sqrt{\mu}})|^{2}dv.
\end{align*}
For $\eta_{0}>0$ in \eqref{1.26}, there exists a suitably large constant  $R>0$ such that
\begin{align*}
\int_{|v|\geq R}\mu^{-\varepsilon_{1}}|\partial_{\beta'}(\frac{M-\mu}{\sqrt{\mu}})|^{2}dv\leq C\eta^{2}_{0},
\end{align*}
and
\begin{align*}
\int_{|v|\leq R}\mu^{-\varepsilon_{1}}|\partial_{\beta'}(\frac{M-\mu}{\sqrt{\mu}})|^{2}dv
\leq C(|\rho-1|+|u-0|+|\theta-\frac{3}{2}|)^{2}
\leq C\eta^{2}_{0}.
\end{align*}
Thus for any $\bar{\beta}\geq0$ and $m>0$, we deduce from the above related estimates that
\begin{align}
\label{5.12}
|\langle v\rangle^{m}\partial_{\bar{\beta}}(\frac{M-\mu}{\sqrt{\mu}})|^{2}_{\sigma}
+|\langle v\rangle^{m}\partial_{\bar{\beta}}(\frac{M-\mu}{\sqrt{\mu}})|^{2}_{2}\leq C\eta^{2}_{0}.
\end{align}
Notice that $|\alpha_{1}|\leq|\alpha|\leq 1$ in \eqref{5.11}. If $|\alpha_{1}|=0$, we use \eqref{5.12}, \eqref{2.19}
and the smallness of $\eta_{0}$ to get
\begin{align}
\label{5.13}
&\epsilon^{a-1}\int_{\mathbb{R}}|\langle v\rangle^{-b}\partial^{\alpha_{1}}_{\beta^{'}}(\frac{M-\mu}{\sqrt{\mu}})|_{2}
|\partial^{\alpha-\alpha_{1}}_{\beta-\beta_{1}}f|_{\sigma,|\beta-\beta_{1}|}
|\partial^{\alpha}_{\beta}h|_{\sigma,|\beta|}dy
\nonumber\\
&\leq C\eta_{0}\epsilon^{a-1}\|\partial^{\alpha}_{\beta-\beta_{1}}f\|_{\sigma,|\beta-\beta_{1}|}
\|\partial^{\alpha}_{\beta}h\|_{\sigma,|\beta|}
\nonumber\\
&\leq \eta\epsilon^{a-1}\|\partial^{\alpha}_{\beta}h\|_{\sigma,|\beta|}^{2}
+C_{\eta}\eta_{0}\mathcal{D}_{2}(\tau).
\end{align}
If $|\alpha_{1}|=|\alpha|=1$, by the one-dimensional Sobolev imbedding theorem
and \eqref{1.26}, we can deduce from Lemma \ref{lem5.3} and the Cauchy inequality that
\begin{align}
\label{5.15}
&\epsilon^{a-1}\int_{\mathbb{R}}|\langle v\rangle^{-b}\partial^{\alpha_{1}}_{\beta^{'}}(\frac{M-\mu}{\sqrt{\mu}})|_{2}
|\partial^{\alpha-\alpha_{1}}_{\beta-\beta_{1}}f|_{\sigma,|\beta-\beta_{1}|}
|\partial^{\alpha}_{\beta}h|_{\sigma,|\beta|}dy
\nonumber\\
&\leq C\epsilon^{a-1}\int_{\mathbb{R}}|\partial^{\alpha_{1}}(\rho,u,\theta)|
|\partial_{\beta-\beta_{1}}f|_{\sigma,|\beta-\beta_{1}|}|\partial^{\alpha}_{\beta}h|_{\sigma,|\beta|}dy
\nonumber\\
&\leq C\epsilon^{a-1}\|\partial^{\alpha_{1}}(\rho,u,\theta)\|
\|\partial_{\beta-\beta_{1}}f\|^{\frac{1}{2}}_{\sigma,|\beta-\beta_{1}|}\|\partial_{y}\partial_{\beta-\beta_{1}}f\|^{\frac{1}{2}}_{\sigma,|\beta-\beta_{1}|}
\|\partial^{\alpha}_{\beta}h\|_{\sigma,|\beta|}
\nonumber\\
&\leq \eta\epsilon^{a-1}\|\partial^{\alpha}_{\beta}h\|^{2}_{\sigma,|\beta|}
+C_{\eta}\epsilon^{a-1}\{\epsilon^{a}(\delta+\epsilon^{a}\tau)^{-1}+\mathcal{E}_{2}(\tau)\}
\|\partial_{\beta-\beta_{1}}f\|_{\sigma,|\beta-\beta_{1}|}\|\partial_{y}\partial_{\beta-\beta_{1}}f\|_{\sigma,|\beta-\beta_{1}|}
\nonumber\\
&\leq\eta\epsilon^{a-1}\|\partial^{\alpha}_{\beta}h\|^{2}_{\sigma,|\beta|}
+C_{\eta}(\epsilon^{a}\delta^{-1}+k^{\frac{1}{12}}\epsilon^{\frac{3}{5}-\frac{2}{5}a})\mathcal{D}_{2}(\tau).
\end{align}
Here we have used the fact that $\mathcal{E}_{2}(\tau)\leq k^{\frac{1}{6}}\epsilon^{\frac{6}{5}-\frac{4}{5}a}\leq k^{\frac{1}{12}}\epsilon^{\frac{3}{5}-\frac{2}{5}a}$ by \eqref{3.4}.
Due to \eqref{3.3}, one has
\begin{equation}
\label{5.16}
\epsilon^{a}\delta^{-1}= k\epsilon^{a}\epsilon^{-\frac{3}{5}+\frac{2}{5}a}
\leq k^{\frac{1}{12}}\epsilon^{\frac{3}{5}-\frac{2}{5}a}.
\end{equation}
It follows from \eqref{5.16}, \eqref{5.13}, \eqref{5.15} and \eqref{5.11} that
\begin{equation}
\label{5.17}
\epsilon^{a-1}|(\partial^{\alpha}_{\beta}\Gamma(\frac{M-\mu}{\sqrt{\mu}},f),
w^{2|\beta|}\partial^{\alpha}_{\beta}h)|
\leq C\eta\epsilon^{a-1}\|\partial^{\alpha}_{\beta}h\|^{2}_{\sigma,|\beta|}
+C_{\eta}(\eta_{0}+k^{\frac{1}{12}}\epsilon^{\frac{3}{5}-\frac{2}{5}a})\mathcal{D}_{2}(\tau).
\end{equation}
On the other hand, similar arguments as \eqref{5.17} imply
\begin{equation}
\label{5.18}
\epsilon^{a-1}|(\partial^{\alpha}_{\beta}\Gamma(f,\frac{M-\mu}{\sqrt{\mu}}),
w^{2|\beta|}\partial^{\alpha}_{\beta}h)|
\leq C\eta\epsilon^{a-1}\|\partial^{\alpha}_{\beta}h\|^{2}_{\sigma,|\beta|}
+C_{\eta}(\eta_{0}+k^{\frac{1}{12}}\epsilon^{\frac{3}{5}-\frac{2}{5}a})\mathcal{D}_{2}(\tau).
\end{equation}
This ends the proof of \eqref{5.9} in terms of \eqref{5.17} and \eqref{5.18}.
By \eqref{5.7} and the similar arguments as \eqref{5.17} and \eqref{5.18},
we can prove that \eqref{5.10} holds and we omit the details for brevity. This completes the proof of Lemma \ref{lem5.8}.
\end{proof}

The second estimates are concerned with the nonlinear term $\Gamma(\frac{G}{\sqrt{\mu}},\frac{G}{\sqrt{\mu}})$.

\begin{lemma}
\label{lem5.9}
Under the assumptions of Lemma \ref{lem5.8} and let $|\alpha|+|\beta|\leq 2$ with $|\beta|\geq1$. Then for any $\eta>0$, one has
\begin{align}
&\epsilon^{a-1}|(\partial^{\alpha}_{\beta}\Gamma(\frac{G}{\sqrt{\mu}},\frac{G}{\sqrt{\mu}}),
w^{2|\beta|}\partial^{\alpha}_{\beta}h)|\notag\\
&\leq  C\eta\epsilon^{a-1}\|\partial^{\alpha}_{\beta}h\|^{2}_{\sigma,|\beta|}
+C_{\eta}\epsilon^{\frac{7}{5}+\frac{1}{15}a}(\delta+\epsilon^{a}\tau)^{-\frac{4}{3}}
+C_{\eta}k^{\frac{1}{12}}\epsilon^{\frac{3}{5}-\frac{2}{5}a}\mathcal{D}_{2}(\tau).\label{5.19}
\end{align}
Moreover, for $|\alpha|\leq 1$, it holds that
\begin{align}
\label{5.20}
\epsilon^{a-1}|(\partial^{\alpha}\Gamma(\frac{G}{\sqrt{\mu}},\frac{G}{\sqrt{\mu}}),\partial^{\alpha}h)|
\leq  C\eta\epsilon^{a-1}\|\partial^{\alpha}h\|^{2}_{\sigma}
+C_{\eta}\epsilon^{\frac{7}{5}+\frac{1}{15}a}(\delta+\epsilon^{a}\tau)^{-\frac{4}{3}}
+C_{\eta}k^{\frac{1}{12}}\epsilon^{\frac{3}{5}-\frac{2}{5}a}\mathcal{D}_{2}(\tau).
\end{align}
\end{lemma}
\begin{proof}
Recall $G=\overline{G}+\sqrt{\mu}f$, a simple computation shows that
\begin{equation}
\label{5.21}
\Gamma(\frac{G}{\sqrt{\mu}},\frac{G}{\sqrt{\mu}})=\Gamma(\frac{\overline{G}}{\sqrt{\mu}},\frac{\overline{G}}{\sqrt{\mu}})
+\Gamma(\frac{\overline{G}}{\sqrt{\mu}},f)+\Gamma(f,\frac{\overline{G}}{\sqrt{\mu}})
+\Gamma(f,f).
\end{equation}
From \eqref{5.8}, one can see that
\begin{align}
\label{5.22}
&\epsilon^{a-1}|(\partial^{\alpha}_{\beta}\Gamma(\frac{\overline{G}}{\sqrt{\mu}},\frac{\overline{G}}{\sqrt{\mu}}),
w^{2|\beta|}\partial^{\alpha}_{\beta}h)|
\nonumber\\
&\leq C\epsilon^{a-1}\sum_{|\alpha_{1}|\leq|\alpha|}
\sum_{|\beta'|\leq|\beta_{1}|\leq|\beta|}\int_{\mathbb{R}}|\partial^{\alpha_{1}}_{\beta^{'}}(\frac{\overline{G}}{\sqrt{\mu}})|_{2,|\beta'|}
|\partial^{\alpha-\alpha_{1}}_{\beta-\beta_{1}}(\frac{\overline{G}}{\sqrt{\mu}})|_{\sigma,|\beta-\beta_{1}|}
|\partial^{\alpha}_{\beta}h|_{\sigma,|\beta|}dy.
\end{align}
In view of \eqref{2.3}, \eqref{5.1} and \eqref{5.2}, we can write
\begin{align*}
\overline{G}=\epsilon^{1-a}\frac{\sqrt{R}\bar{\theta}_{y}}{\sqrt{\theta}}A_{1}(\frac{v-u}{\sqrt{R\theta}})
+\epsilon^{1-a}\bar{u}_{1y}B_{11}(\frac{v-u}{\sqrt{R\theta}}),
\end{align*}
which implies that for $\beta_{1}=(1,0,0)$,
\begin{equation}
\label{5.23}
\partial_{\beta_{1}}\overline{G}=\epsilon^{1-a}\Big\{\frac{\sqrt{R}\bar{\theta}_{y}}{\sqrt{\theta}}
\partial_{v_{1}}A_{1}(\frac{v-u}{\sqrt{R\theta}})(\frac{1}{\sqrt{R\theta}})
+\bar{u}_{1y}\partial_{v_{1}}B_{11}(\frac{v-u}{\sqrt{R\theta}})\frac{1}{\sqrt{R\theta}}\Big\},
\end{equation}
and
\begin{align}
\label{5.24}
\partial_{y}\overline{G}&=\epsilon^{1-a}\Big\{\frac{\sqrt{R}\bar{\theta}_{yy}}{\sqrt{\theta}}A_{1}(\frac{v-u}{\sqrt{R\theta}})
-\frac{\sqrt{R}\bar{\theta}_{y}\theta_{y}}{2\sqrt{\theta^{3}}}A_{1}(\frac{v-u}{\sqrt{R\theta}})
\nonumber\\
&\quad-\frac{\sqrt{R}\bar{\theta}_{y}}{\sqrt{\theta}}
\nabla_{v}A_{1}(\frac{v-u}{\sqrt{R\theta}})\cdot\frac{u_{y}}{\sqrt{R\theta}}
-\frac{\sqrt{R}\bar{\theta}_{y}\theta_{y}}{\sqrt{\theta}}
\nabla_{v}A_{1}(\frac{v-u}{\sqrt{R\theta}})\cdot\frac{v-u}{\sqrt{2R\theta^{3}}}
\nonumber\\
&\quad+\bar{u}_{1yy}B_{11}(\frac{v-u}{\sqrt{R\theta}})
-\frac{\bar{u}_{1y}u_{y}}{\sqrt{R\theta}}\cdot\nabla_{v}B_{11}(\frac{v-u}{\sqrt{R\theta}})
-\frac{\bar{u}_{1y}\theta_{y}(v-u)}{2\sqrt{R\theta^{3}}}\cdot\nabla_{v}B_{11}(\frac{v-u}{\sqrt{R\theta}})
\Big\}.
\end{align}
By using \eqref{5.4} and the similar expansion as \eqref{5.23} and \eqref{5.24}, for any $|\bar{\alpha}|\geq1$
and $|\bar{\beta}|\geq0$, we can obtain
\begin{equation}
\label{5.25}
|\langle v\rangle^{m}\partial_{\bar{\beta}}(\frac{\overline{G}}{\sqrt{\mu}})|_{2,|\bar{\beta}|}
+|\langle v\rangle^{m}\partial_{\bar{\beta}}(\frac{\overline{G}}{\sqrt{\mu}})|_{\sigma,|\bar{\beta}|}
\leq C\epsilon^{1-a}|(\bar{u}_{1y},\bar{\theta}_{y})|,
\end{equation}
and
\begin{equation}
\label{5.26}
|\langle v\rangle^{m}\partial^{\bar{\alpha}}_{\bar{\beta}}(\frac{\overline{G}}{\sqrt{\mu}})|_{2,|\bar{\beta}|}+
|\langle v\rangle^{m}\partial^{\bar{\alpha}}_{\bar{\beta}}(\frac{\overline{G}}{\sqrt{\mu}})|_{\sigma,|\bar{\beta}|}
\leq C\epsilon^{1-a}\{|\partial^{\bar{\alpha}}(\bar{u}_{1y},\bar{\theta}_{y})|+\cdot\cdot\cdot
+|(\bar{u}_{1y},\bar{\theta}_{y})||\partial^{\bar{\alpha}}(u,\theta)|\},
\end{equation}
due to the fact that $|\langle v\rangle^{m}w^{|\bar{\beta}|}\mu^{-\frac{1}{2}}M^{1-\varepsilon}|_{2}\leq C$
for any $m\geq0$ and $\varepsilon>0$ small enough .
\par
With the help of \eqref{5.25} and \eqref{5.26}, we get from the Sobolev imbedding theorem that
\begin{align}
\label{5.27}
&\epsilon^{a-1}\int_{\mathbb{R}}|\partial^{\alpha_{1}}_{\beta^{'}}(\frac{\overline{G}}{\sqrt{\mu}})|_{2,|\beta'|}
|\partial^{\alpha-\alpha_{1}}_{\beta-\beta_{1}}(\frac{\overline{G}}{\sqrt{\mu}})|_{\sigma,|\beta-\beta_{1}|}
|\partial^{\alpha}_{\beta}h|_{\sigma,|\beta|}dy
\nonumber\\
&\leq C\epsilon^{a-1}\int_{\mathbb{R}}\Big\{\epsilon^{1-a}\{|\partial^{\alpha_{1}}(\bar{u}_{1y},\bar{\theta}_{y})|
+|(\bar{u}_{1y},\bar{\theta}_{y})||\partial^{\alpha_{1}}(u,\theta)|\}
\nonumber\\
&\qquad\times\epsilon^{1-a}\{|\partial^{\alpha-\alpha_{1}}(\bar{u}_{1y},\bar{\theta}_{y})|
+|(\bar{u}_{1y},\bar{\theta}_{y})||\partial^{\alpha-\alpha_{1}}(u,\theta)|\}
|\partial^{\alpha}_{\beta}h|_{\sigma,|\beta|}\Big\}dy
\nonumber\\
&\leq C\eta\epsilon^{a-1}\|\partial^{\alpha}_{\beta}h\|_{\sigma,|\beta|}^{2}
+C_{\eta}\epsilon^{\frac{7}{5}+\frac{1}{15}a}(\delta+\epsilon^{a}\tau)^{-\frac{4}{3}}
+C_{\eta}k^{\frac{1}{12}}\epsilon^{\frac{3}{5}-\frac{2}{5}a}\mathcal{D}_{2}(\tau),
\end{align}
according to \eqref{5.3}, \eqref{3.4} and $|\alpha_{1}|\leq|\alpha|\leq1$. This together with \eqref{5.22} give that
\begin{align}
&\epsilon^{a-1}|(\partial^{\alpha}_{\beta}\Gamma(\frac{\overline{G}}{\sqrt{\mu}},\frac{\overline{G}}{\sqrt{\mu}}),
w^{2|\beta|}\partial^{\alpha}_{\beta}h)|\notag\\
&\leq C\eta\epsilon^{a-1}\|\partial^{\alpha}_{\beta}h\|_{\sigma,|\beta|}^{2}
+C_{\eta}\epsilon^{\frac{7}{5}+\frac{1}{15}a}(\delta+\epsilon^{a}\tau)^{-\frac{4}{3}}
+C_{\eta}k^{\frac{1}{12}}\epsilon^{\frac{3}{5}-\frac{2}{5}a}\mathcal{D}_{2}(\tau).\label{5.28}
\end{align}
For the second term of \eqref{5.21}, by using \eqref{5.8}, \eqref{5.26} and the Sobolev imbedding theorem, one can deduce
from \eqref{3.3} and \eqref{3.4} that
\begin{align}
&\epsilon^{a-1}|(\partial^{\alpha}_{\beta}\Gamma(\frac{\overline{G}}{\sqrt{\mu}},f),
w^{2|\beta|}\partial^{\alpha}_{\beta}h)|\notag\\
&\leq C \epsilon^{a-1}\sum_{|\alpha_{1}|\leq|\alpha|}\sum_{|\beta_{1}|\leq|\beta|}
\int_{\mathbb{R}}\epsilon^{1-a}\Big\{|\partial^{\alpha_{1}}(\bar{u}_{1y},\bar{\theta}_{y})|\notag\\
&\qquad\qquad\qquad\qquad\qquad\qquad\quad+|(\bar{u}_{1y},\bar{\theta}_{y})||\partial^{\alpha_{1}}(u,\theta)|\Big\}
|\partial^{\alpha-\alpha_{1}}_{\beta-\beta_{1}}f|_{\sigma,|\beta-\beta_{1}|}|\partial^{\alpha}_{\beta}h|_{\sigma,|\beta|}dy
\nonumber\\
&\leq \eta\epsilon^{a-1}\|\partial^{\alpha}_{\beta}h\|^{2}_{\sigma,|\beta|}
+C_{\eta}k^{\frac{1}{12}}\epsilon^{\frac{3}{5}-\frac{2}{5}a}\mathcal{D}_{2}(\tau).\label{5.29}
\end{align}
Since the third term of \eqref{5.21} shares the same estimates as \eqref{5.29}. Thus, we still deal with
the last term of \eqref{5.21}.
In view of \eqref{5.8}, the imbedding theorem
and \eqref{3.4}, one has
\begin{align}
&\epsilon^{a-1}|(\partial^{\alpha}_{\beta}\Gamma[f,f],
w^{2|\beta|}\partial^{\alpha}_{\beta}h)|\notag\\
&\leq C\epsilon^{a-1}\sum_{|\alpha_{1}|\leq|\alpha|}
\sum_{|\beta'|\leq|\beta_{1}|\leq|\beta|}\int_{\mathbb{R}}
|\partial^{\alpha_{1}}_{\beta^{'}}f|_{2,|\beta'|}|\partial^{\alpha-\alpha_{1}}_{\beta-\beta_{1}}f|_{\sigma,|\beta-\beta_{1}|}
|\partial^{\alpha}_{\beta}h|_{\sigma,|\beta|}dy
\nonumber\\
&\leq \eta\epsilon^{a-1}\|\partial^{\alpha}_{\beta}h\|^{2}_{\sigma,|\beta|}
+C_{\eta}\sqrt{\mathcal{E}_{2}(\tau)}\mathcal{D}_{2}(\tau)\notag\\
&\leq\eta\epsilon^{a-1}\|\partial^{\alpha}_{\beta}h\|^{2}_{\sigma,|\beta|}
+C_{\eta}k^{\frac{1}{12}}\epsilon^{\frac{3}{5}-\frac{2}{5}a}\mathcal{D}_{2}(\tau).\label{5.30}
\end{align}
Collecting  the estimates of \eqref{5.28}, \eqref{5.29} and \eqref{5.30}, we can obtain
\begin{align*}
&\epsilon^{a-1}|(\partial^{\alpha}_{\beta}\Gamma(\frac{G}{\sqrt{\mu}},\frac{G}{\sqrt{\mu}}),
w^{2|\beta|}\partial^{\alpha}_{\beta}h)|\notag\\
&\leq  C\eta\epsilon^{a-1}\|\partial^{\alpha}_{\beta}h\|^{2}_{\sigma,|\beta|}
+C_{\eta}\epsilon^{\frac{7}{5}+\frac{1}{15}a}(\delta+\epsilon^{a}\tau)^{-\frac{4}{3}}
+C_{\eta}k^{\frac{1}{12}}\epsilon^{\frac{3}{5}-\frac{2}{5}a}\mathcal{D}_{2}(\tau).
\end{align*}
This completes the proof of \eqref{5.19}.
One can deduce \eqref{5.20} by employing \eqref{5.7} and the similar arguments as the above related estimates.
This ends the proof of Lemma \ref{lem5.9}.
\end{proof}

\noindent {\bf Acknowledgment:}\,
The research of Renjun Duan was partially supported by the General Research Fund (Project No.~14302817) from RGC of Hong Kong and a Direct Grant from CUHK. The research of Hongjun Yu was supported by the GDUPS 2017 and the NNSFC Grant 11371151.

\medskip

\noindent{\bf Conflict of Interest:} The authors declare that they have no conflict of interest.


\end{document}